\documentclass[12pt]{elsarticle}
\usepackage{mystyleElsArticle}
\usepackage{mathrsfs}
\usepackage{booktabs}

\newcommand{\gvec}[1]{\boldsymbol{#1}}
\newcommand{\vxi}{\gvec{\xi}}

\usepackage[margin=1.00 in]{geometry}    

\usepackage{tikz}
\usetikzlibrary{arrows, arrows.meta,automata, calc, positioning,decorations.pathreplacing}

\newcommand\E{\mathbb{E}}

\makeatletter
\def\ps@pprintTitle{%
  \let\@oddhead\@empty
  \let\@evenhead\@empty
  \let\@oddfoot\@empty
  \let\@evenfoot\@oddfoot
}
\makeatother

\begin{document}

\begin{frontmatter}
  \title{A data-driven approach for multiscale elliptic PDEs with random coefficients based on intrinsic dimension reduction}
  
   \author[hku]{Sijing Li}
   \ead{lsj17@hku.hk}  
  \author[hku]{Zhiwen Zhang\corref{cor1}}
  \ead{zhangzw@hku.hk}
  \author[uci]{Hongkai Zhao}
  \ead{zhao@math.uci.edu}  
 
  \address[hku]{Department of Mathematics, The University of Hong Kong, Pokfulam Road, Hong Kong SAR, China.}
  \address[uci]{Department of Mathematics, University of California at Irvine, Irvine, CA 92697, USA.}
  
  \cortext[cor1]{Corresponding author}

\begin{abstract}
\noindent	
We propose a data-driven approach to solve multiscale elliptic PDEs with random coefficients based on the intrinsic low dimension structure of the underlying elliptic differential operators. Our method consists of offline and online stages. At the offline stage, a low dimension space and its basis are extracted from the data to achieve significant dimension reduction in the solution space. At the online stage, the extracted basis will be used to solve a new multiscale elliptic PDE efficiently. The existence of low dimension structure is established by showing the high separability of the underlying Green's functions. Different online construction methods are proposed depending on the problem setup.
We provide error analysis based on the sampling error and the truncation threshold in building the data-driven basis. Finally, we present numerical examples to demonstrate the accuracy and efficiency of the proposed method. 
\\
\noindent \textit{\textbf{AMS subject classification:}}  35J08, 35J15, 65N30, 65N80, 78M34. 
\end{abstract}
\begin{keyword}
Partial differential equations (PDEs) with random coefficients; Green's function; separability; principle component analysis; proper orthogonal decomposition (POD); uncertainty quantification (UQ); neural network. 
\end{keyword}
\end{frontmatter}

\section{Introduction} \label{sec:introduction}
\noindent
In this paper, we shall develop a data-driven method to solve the following multiscale elliptic PDEs with random coefficients $a(x,\omega)$, 
\begin{align}
\mathcal{L}(x,\omega) u(x,\omega) \equiv -\nabla\cdot\big(a(x,\omega)\nabla u(x,\omega)\big) &= f(x), 
\quad x\in D, \quad \omega\in\Omega,   \label{MsStoEllip_Eq}\\
u(x,\omega)&= 0, \quad \quad x\in \partial D, \label{MsStoEllip_BC}
\end{align}
where $D \in \mathbb{R}^d$ is a bounded spatial domain and $\Omega$ is a sample space. The forcing function $f(x)$ is assumed to be in $L^2(D)$.
We also assume that the problem is uniformly elliptic almost surely; see Section \ref{sec:randomproblem} for precise definition of the problem.

In recent years, there has been an increased interest in quantifying the uncertainty in systems with randomness, i.e., solving stochastic partial differential equations (SPDEs, i.e., PDEs driven by Brownian motion) or partial differential equations with random coefficients (RPDEs). Uncertainty quantification (UQ) is an emerging research 
area to address these issues; see \cite{Ghanem:91,Xiu:03,babuska:04,matthies:05,WuanHou:06,Wan:06,Babuska:07,Webster:08,Xiu:09,Najm:09,sapsis:09,Zabaras:13,Grahamquasi:2015} and references therein. 
However, when SPDEs or RPDEs involving multiscale features and/or high-dimensional random inputs, the problems become challenging due to high computational cost.

Recently, some progress has been made in developing numerical methods for multiscale PDEs with random coefficients; see \cite{Kevrekidis:2003,Zabaras:06,Ghanem:08,graham2011quasi,abdulle2013multilevel,hou2015heterogeneous,ZhangCiHouMMS:15,efendiev2015multilevel,chung2018cluster} and references therein.  
For example,  data-driven stochastic methods to solve PDEs with random and/or multiscale coefficients were proposed in \cite{ChengHouYanZhang:13,ZhangCiHouMMS:15,ZhangHouLiu:15,hou2019model}. They demonstrated through numerical experiments that those methods were efficient in solving RPDEs with many different force functions. However, the polynomial chaos expansion \cite{Ghanem:91,Xiu:03} is used to represent the randomness in the solutions. Although the polynomial chaos expansion is general, it is a priori instead of problem specific. Hence many terms may be required in practice for an accurate approximation which induces the curse of dimensionality. 

We aim to develop a new data-driven method to solve multiscale elliptic PDEs with random coefficients based on intrinsic dimension reduction. The underlying low-dimensional structure for elliptic problems is implied by the work \cite{bebendorf2003}, in which high separability of the Green's 
function for uniformly elliptic operators with $L^{\infty}$ coefficients and the structure of blockwise low-rank approximation to the inverses of FEM matrices were established. We show that under the uniform ellipticity assumption, the family of Green's functions parametrized by a random variable $\omega$ is still highly separable, which reveals the approximate low dimensional structure of the family of solutions to \eqref{MsStoEllip_Eq} (again parametrized by $\omega$)
and motivates our method. 
 
Our method consists of two stages. In the offline stage, a set of data-driven basis 
is constructed from solution samples. For example, the data can be generated by solving \eqref{MsStoEllip_Eq}-\eqref{MsStoEllip_BC} corresponding to a sampling of the coefficient $a(x,\omega)$. 
Here, different sampling methods can be applied, including Monte Carlo (MC) method and quasi-Monte Carlo (qMC) method. The sparse-grid based stochastic collocation method  \cite{Griebel:04,Xiu:05,Webster:08} also works when the dimension of the random variables in $a(x,\omega)$ is moderate. Or the data come from field measurements directly.
Then the low-dimensional structure and the corresponding basis will be extracted using model reduction methods, such as the proper orthogonal decomposition (POD) \cite{HolmesLumleyPOD:1993,Sirovich:1987,Willcox2015PODsurvey}, a.k.a. principle component analysis (PCA). The basis functions are data driven and problem specific. 
The key point is that once the dimension reduction is achieved, the online stage of computing the solution corresponding to a new coefficient becomes finding a linear combination of the (few) basis to approximate the solution. However, the mapping from the input coefficients of the PDE to the expansion coefficients of the solution in terms of the data driven basis is highly nonlinear. We propose a few possible online strategies (see Section \ref{sec:DerivationNewMethod}). For examples, if the coefficient is in parametric form, one can approximate the nonlinear map from the parameter domain to the expansion coefficients. Or one can apply Galerkin method using the extracted basis to solve \eqref{MsStoEllip_BC} for a new coefficient. In practice, the random coefficient of the PDE may not be available but censors can be deployed to record the solution at certain locations. In this case, one can compute the expansion coefficients of a new solution by least square fitting those measurements at designed locations.
We also provide analysis and guidelines for sampling, dimension reduction, and other implementations of our methods.


 
The rest of the paper is organized as follows. In Section 2, we introduce the high separability of the Green's function of deterministic elliptic PDEs and present its extension to  elliptic problems with random coefficients. In section 3, we describe our new data-driven method and its detailed implementation. In Section 4, we present numerical results to demonstrate the efficiency of our method.  Concluding remarks are made in Section 5.


\section{Low-dimensional structures in the solution space} \label{sec:LowDimStructures}
\subsection{High separability of the Green's function of deterministic elliptic operators. }
\noindent
Let $\mathcal{L}(x): V \to V' $ be an uniformly elliptic operator in a divergence form
\begin{align}
\mathcal{L}(x)u(x) \equiv  -\nabla\cdot(a(x)\nabla u(x))\label{DeterministicEllipticPDE}
\end{align}
in a bounded Lipschitz domain $D \subset \mathbb{R}^d$, where $V = H_0^1(D)$. The uniformly elliptic assumption means that there exist $a_{\min}, a_{\max}>0$, such that $a_{\min}<a(x)<a_{\max}$ for almost all $x \in D$. The contrast ratio $\kappa_a=\frac{a_{\max}}{a_{\min}}$ is an important factor in the stability and convergence analysis. We consider the Dirichlet boundary value problem defined as 
\begin{align}
\mathcal{L}(x)u(x)=f(x), \quad \text{in}~ D,  \quad u(x)=0, \quad \text{on} ~ \partial D.
\label{DeterministicDirichletProblem}
\end{align}
For all $x,y\in D$, the Green's function $G(x,y)$ is the solution of 
\begin{align}
\mathcal{L}G(\cdot,y)=\delta(\cdot,y), \quad \text{in}~ D,  \quad  G(\cdot,y)=0, \quad \text{on} ~ \partial D,
\label{DeterministicGreenFunction}
\end{align}
where $\mathcal{L}$ refers to the first variable $\cdot$ and $\delta(\cdot,y)$ is the Dirac delta function denoting an impulse source point at $y\in D$. The Green's function $G(x,y)$ is the Schwartz kernel of the inverse $\mathcal{L}^{-1}$, i.e., the solution of \eqref{DeterministicDirichletProblem} is represented by 
\begin{align}
u(x)=\mathcal{L}^{-1}f(x)=\int_{D}G(x,y)f(y)dy. \label{solutionrepresentationGreen}
\end{align}
Since the coefficient $a(x)$ is only bounded, the $G(x,y)$ has a lower regularity, compared with the 
Green's function associated with the Poisson's equation. 
In \cite{gruter1982green}, the authors proved the existence of Green's function for $ d \geq 3 $ and the estimate $|G(x,y)|\leq \frac{C(d,\kappa_a)}{a_{\min}}|x-y|^{2-d}$, where $C(d,\kappa_a)$ is a constant depends on $d$ and $\kappa_a$. For $d=2$ the existence of the Green's function was proved in \cite{dolzmann1995estimates} together with the estimate $|G(x,y)|\leq \frac{C(\kappa_a)}{a_{\min}}\log|x-y|$.
Thus, when $\mathcal{L}$ is an uniform elliptic operator, 
$\mathcal{L}^{-1}$ exists and $||\mathcal{L}^{-1}||\leq Ca_{\min}^{-1}$, where $C$ depends on $d$ and $\kappa_a$.  

Under mild assumptions, one can prove that the solution $u(x)$ to Eq.\eqref{DeterministicDirichletProblem} has finite dimensional approximations as follows.
\begin{proposition}\label{FiniteDimensionalApprox}
	Let $D \subset \mathbb{R}^d$ be a convex domain and $X$ be a closed subspace of $L^2(D)$. Then for any 
	integer $k \in \mathbb{N}$ there is a subspace $V_k \subset X$ satisfying $\dim V_k \leq k$ such that 
	\begin{align}
	\text{\normalfont dist}_{L^2(D)}(u,V_k) \le C \frac{\text{\normalfont diam}(D)}{\sqrt[d]{k}}\| \nabla u \| _{L^2(D)}, \quad \text{\normalfont for all } u \in X \cap H^1(D), 
	\label{ApproxByFiniteDimensionSubspace}
	\end{align}
	where the constant $C$ depends only on the spatial dimension $d$.
\end{proposition} 
The proof is based on the Poincar\'{e} inequality; see \cite{BebendorfHackbusch:2003}.  All distances and diameters use the Euclidean norm in $\mathbb{R}^d$ except the distance of functions which uses the $L^2(D)$-norm. We emphasize that in the Prop. \ref{FiniteDimensionalApprox} one can choose the finite dimensional space $V_k$ to be the space of piece-wise constant functions defined on a grid with grid size $\frac{\text{\normalfont diam}(D)}{\sqrt[d]{k}}$.   

Now we present the definition of $\mathcal{L}$-harmonic function on a domain $E\subset D$ introduced in  \cite{BebendorfHackbusch:2003}. A function $u$ is $\mathcal{L}$-harmonic on $E$ if $u\in H^1(\hat{E}), \forall \hat{E} \subset E$ with $dist(\hat{D}, \partial E)>0$ and satisfies
\[
a(u, \varphi) =  \int_{E} a(x)\nabla u(x)\cdot \nabla \varphi(x) dx =0 \quad \forall \varphi \in C_0^{\infty} (E).
\]
Denote the space of $\mathcal{L}$-harmonic functions on $E$ by $X(E)$, which is closed in $L^2(E)$. The following key Lemma shows that the space of $\mathcal{L}$-harmonic function has an approximate low dimensional structure. 

\begin{lemma}[Lemma 2.6 of \cite{BebendorfHackbusch:2003}]\label{lemma1}
Let $\hat{E}\subset E \subset D$ in $R^d$ and assume that $\hat{E}$ is convex such that 
\[
dist(\hat{E}, \partial E)\ge \rho~ diam(\hat{E})>0, \quad \mbox{for some constant } \rho >0.
\]
Then for any $1>\epsilon>0$, there is a subspace $W\subset X(\hat{E})$ so that for all $u\in X(\hat{E})$,
\[
dist_{L^2(\hat{E})}(u, W)\le \epsilon \|u\|_{L^2(E)}
\]
and
\[
dim(W)\le c^d(\kappa_a,\rho) (|\log \epsilon |)^{d+1},
\]
where $c(\kappa_a,\rho) >0 $ is a constant that depends on $\rho$ and $\kappa_a$.
\end{lemma}
In other words, the above Lemma says the Kolmogorov n-width of the space of $\mathcal{L}$-harmonic function $X(\hat{E})$  is less than $O(\exp(-n^{\frac{1}{d+1}}))$.
The key property of $\mathcal{L}$-harmonic functions used to prove the above result is the Caccioppoli inequality, which provides the estimate $\|\nabla u\|_{L^2(\hat{\E})} \le C(\kappa_a, \rho)\|u\|_{L^2(E)}$. Moreover, the projection of the space of piecewise constant functions defined on a multi-resolution rectangular mesh onto  $X(\hat{E})$ can be constructed as a candidate for $W$ based on Prop. \ref{FiniteDimensionalApprox}.

In particular, the Green's function $G(\cdot,y)$ is $\mathcal{L}$-harmonic on $E$ if $y\notin E$. Moreover, given two disjoint domains in $D_1, D_2$ in $D$, the Green's function $G(x,y)$ with $x\in D_1, y\in D_2$ can be viewed as a family of $\mathcal{L}$-harmonic functions on $D_1$ parametrized by $y\in D_2$. From the above Lemma one can easily deduce the following result which shows the high separability of the Green's function for the elliptic operator \eqref{DeterministicEllipticPDE}.



\begin{figure}[tbph] 
	\centering
	\begin{tikzpicture}[scale=0.9]
	\coordinate [label={[xshift=0.7cm, yshift=0.3cm]$D$}] (a1) at (0,0);
	\coordinate (b1) at (0,4);
	\coordinate (c1) at (8,4);
	\coordinate (d1) at (8,0);
	\draw(a1)--(b1)--(c1)--(d1)--cycle;
	\coordinate (a2) at (1,0.8);
	\coordinate (b2) at (1,3.2);
	\coordinate (c2) at (3,3.2);
	\coordinate (d2) at (3,0.8);
	\draw(a2)--(b2)--(c2)--(d2)--cycle;
	\coordinate (a3) at (5,0.8);
	\coordinate (b3) at (5,3.2);
	\coordinate (c3) at (7,3.2);
	\coordinate (d3) at (7,0.8);
	\draw(a3)--(b3)--(c3)--(d3)--cycle;
	\tikzstyle{textnode}  = [thick, fill=white, minimum size = 0.1cm]
	\node[textnode] (D1) at (2,2) {$D_1$};
	\node[textnode] (D2) at (6,2) {$D_2$};
	\node[textnode] (Gf) at (4,3.3) {$G(x,y)$};
	\path [->] (Gf) edge node {} (D1);
	\path [->] (Gf) edge node {} (D2);
	\end{tikzpicture}
	\caption{Green's function $G(x,y)$ with dependence on $x\in D_1$ and $y\in D_2$.}
	\label{fig:Greenfunction1}
\end{figure}
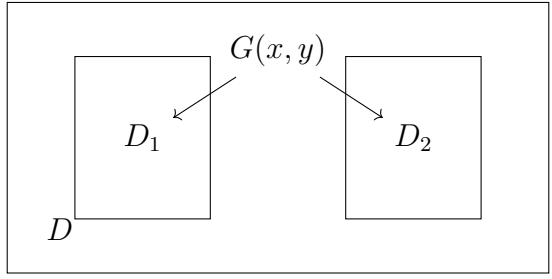 

\begin{proposition}[Theorem 2.8 of \cite{BebendorfHackbusch:2003}]\label{GreenFuncSepaApp}
	Let $D_1, D_2 \subset D$ be two subdomains and $D_1$ be convex (see Figure \ref{fig:Greenfunction1}). Assume that there exists $\rho>0$ such that 
	\begin{align}
	0 < \text{ \normalfont diam} (D_1) \leq \rho\text{ \normalfont dist} (D_1, D_2). 
	\label{AdmissiblePairs}
	\end{align}
	Then for any $\epsilon \in (0,1)$ there is a separable approximation
	\begin{align}
	G_k(x,y) = \sum_{i=1}^k u_i(x) v_i(y) \quad \text{with } k \leq  
	c^d(\kappa_a, \rho) |\log \epsilon|^{d+1},
	\label{GreenFuncSepaApp1}
	\end{align}
	so that for all $y\in D_2$
	\begin{align}
	\| G( \cdot,y) - G_k(\cdot,y) \|_{L^2 (D_1)} \leq \epsilon \| G(\cdot,y) \| _{L^2(\hat{D}_1)},
	\end{align}
	where  $\hat{D}_1 := \{ x \in D :  2\rho~\text{\normalfont dist} (x, D_1) \leq \text{\normalfont diam} (D_1)\}$.
\end{proposition}
\begin{remark}
In the recent work \cite{EngquistZhao:2018}, it is shown that the Green's function for high frequency Helmholtz equation is not highly separable due to the highly oscillatory phase.
\end{remark}

\subsection{Extension to elliptic PDEs with random coefficients}\label{sec:randomproblem}
\noindent
Let's consider the following elliptic PDEs with random coefficients:
\begin{align}
\mathcal{L}(x,\omega) u(x,\omega) 
\equiv -\nabla\cdot\big(a(x,\omega)\nabla u(x,\omega)\big) &= f(x), 
\quad x\in D, \quad \omega\in\Omega,   \label{MsStoEllip_ModelEq}\\
u(x,\omega)&= 0, \quad \quad x\in \partial D, \label{MsStoEllip_ModelBC}
\end{align}
where $D \in \mathbb{R}^d$ is a bounded spatial domain and $\Omega$ is a sample space. The forcing function $f(x)$ is assumed to be in $L^2(D)$. The above equation can be used to model the flow pressure in porous media such as water aquifer and oil reservoirs, where the permeability field $a(x,\omega)$ is a random field whose exact values are infeasible to obtain in practice due to the low resolution of seismic data. We also assume that the problem is uniformly elliptic almost surely, namely, there exist $a_{\min}, a_{\max}>0$, such that
\begin{align}
P\big(\omega\in \Omega: a(x, \omega)\in [a_{\min},a_{\max}], \forall x \in D\big) = 1.
\label{asUniformlyElliptic1}
\end{align}
Note that we do not make any assumption on the regularity of the coefficient $a(x,\omega)$ in the physical space, which can be arbitrarily rough for each realization. 
For the problem \eqref{MsStoEllip_ModelEq}-\eqref{MsStoEllip_ModelBC}, the corresponding Green function is defined as
\begin{align}
\mathcal{L}(x,\omega)G(x,y,\omega) \equiv-\nabla_x\cdot(a(x,\omega)\nabla_x G(x,y,\omega)) &= \delta(x,y), \quad x\in D,\quad \omega\in\Omega,\\
G(x,y,\omega) &= 0, \quad  \quad x\in \partial D, 
\end{align} 
where $y\in D$ and $\delta(x,y)$ is the Dirac delta function. 
A key observation for the proof of Lemma \ref{lemma1} and Prop. \ref{FiniteDimensionalApprox} is that the projection of the space of piecewise constant functions defined on a multi-resolution rectangular mesh, depending only on the geometry of $D_1, D_2$, $\kappa_a$, and $\rho$, onto the $\mathcal{L}$-harmonic function provides a candidate for the finite dimensional subspace $W$.  Based on this observation, one can easily extend the statement in Prop. \ref{FiniteDimensionalApprox} to the family of Green's functions $G(x,y,\omega) $ parametrized by $\omega$
under the uniform ellipticity assumption \eqref{asUniformlyElliptic1}.
\begin{theorem}\label{ThmRandomGreenFuncSepaApp}
	Let $D_1, D_2 \subset D$ be two subdomains and $D_1$ be convex. Assume that there is $\rho>0$ such that $0 < \text{ \normalfont diam} (D_1) \leq \rho\text{ \normalfont dist} (D_1, D_2)$. 
	Then for any $\epsilon \in (0,1)$ there is a separable approximation
	\begin{align}
	G_k(x,y,\omega) = \sum_{i=1}^k u_i(x) v_i(y,\omega) \quad \text{with } k \leq 
	c^d(\kappa_a, \rho)  |\log \epsilon|^{d+1},
	\label{RandomGreenFuncSepaApp}
	\end{align}
	so that for all $y\in D_2$
	\begin{align}
	\| G(\cdot,y,\omega) - G_k(\cdot,y, \omega) \|_{L^2 (D_1)} \leq \epsilon \| G(\cdot,y, \omega) \| _{L^2(\hat{D}_1)} \quad  \text{a.s. in } \Omega,
	\end{align}
	where $\hat{D}_1 := \{ x \in D :  2\rho\text{\normalfont dist} (x, D_1) \leq \text{\normalfont diam} (D_1)\}$.
\end{theorem}

The above Theorem shows that there exists a low dimensional linear subspace, e.g., spanned by $u_i(\cdot)$, that can approximate the family of functions $G(\cdot,y,\omega)$ well in $L^2(D_1)$ uniformly with respect to $y\in D_2$  and a.s. in $\omega$. Moreover, if $\mathrm{supp}(f)\subset D_2$, one can approximate the solution to \eqref{MsStoEllip_ModelBC} by the same space well in $L^2(D_1)$ uniformly with respect to $f$ and a.s. in $\omega$. Let 
\begin{equation}
u_f(x,\omega)=\int_{D_2} G(x,y,\omega)f(y) dy
\end{equation}
and
\begin{equation}
u^{\epsilon}_f(x,\omega)=\int_{D_2} G_k(x,y,\omega)f(y) dy=\sum_{i=1}^k u_i(x)\int_{D_2} v_i(y,\omega) f(y) dy.
\end{equation}
Hence
\begin{equation}
\begin{array}{l}
\|u_f(\cdot,\omega)-u^{\epsilon}_f(\cdot,\omega)\|^2_{L^2(D_1)}=\int_{D_1} \left[\int_{D_2} (G(x,y,\omega)-G_k(x,y,\omega))f(y) dy\right]^2 dx 
\\ \\
\le \|f\|_{L^2(D_2)}^2 \int_{D_2}\| G(\cdot,y,\omega) - G_k(\cdot,y, \omega) \|^2_{L^2 (D_1)} dy\le C(D_1, D_2, \kappa_a, d)\epsilon^2\|f\|_{L^2(D_2)}^2,
\end{array}
\end{equation}
a.s. in $\omega$ since $\| G(\cdot,y, \omega) \| _{L^2(\hat{D}_1)}$ is bounded by a positive constant that depends on $D_1, D_2, \kappa_a, d$ a.s. in $\omega$ due to uniform ellipticity \eqref{asUniformlyElliptic1}.
Although, the proof of high separability of the Green's function requires $x\in D_1, y\in D_2$  for well separated $D_1$ and $D_2$,  i.e., avoiding the singularity of the Green's function at $x=y$,  the above approximation of the solution $u$ in a domain disjoint with the support of $f$ seems to be valid for $u$ in the whole domain even when $f$ is a globally supported smooth function as shown in our numerical tests. 


\begin{remark}\label{remark1}
It is important to note that both the linear subspace $W$ and the bound for its dimension are independent of the randomness. Moreover, it is often possible to find a problem specific and data driven subspace with a dimension much smaller than the theoretical upper bound for $W$ (as demonstrated by our experiments). This key observation motivates our data-driven approach which can achieve a significant dimension reduction in the solution space. 
\end{remark}
\begin{remark}
Although we present the problem and our data driven approach for the elliptic problem \eqref{MsStoEllip_ModelBC} with scalar random coefficients $a(x,\omega)$, all the statements can be directly extended when the random coefficient is replaced by a symmetric positive definite tensor $a_{i,j}(x,\omega), i,j,=1, \ldots, d$ with uniform ellipticity.
\end{remark}

\begin{remark}
In the recent work \cite{BrysonZhaoZhong:2019}, it is shown that a random field can have a large intrinsic complexity if it is rough, i.e., $a(x_1,\omega)$ and $a(x_2, \omega)$ decorrelate quickly in terms of $\|x_1-x_2\|$. However, when a random field, as rough as it can be, is input as the coefficient of an elliptic PDE, the intrinsic complexity of the resulting solution space, which depends on the coefficient highly nonlinearly and nonlocally, is highly reduced. This phenomenon can also be used to explain the severe ill-posedness of the inverse problem in which one tries to recover the coefficient of an elliptic PDE from the boundary measurements such as electrical impedance tomography (EIT).
\end{remark}

Before we end this subsection, we give a short review of existing methods for solving problem \eqref{MsStoEllip_ModelEq}-\eqref{MsStoEllip_ModelBC} involving random coefficients. There are basically two types
of methods. In intrusive methods, one represents the solution of \eqref{MsStoEllip_ModelEq} by $u(x,\omega)= \sum_{\alpha \in J} u_{\alpha}(x)H_{\alpha}(\omega)$, where $J$ is an index set, and $H_{\alpha}(\omega)$ are certain basis functions (e.g. orthogonal polynomials). Typical examples are the Wiener chaos expansion (WCE) and polynomial chaos expansion (PCE) method. Then, one uses Galerkin method to compute the expansion coefficients $u_{\alpha}(x)$; see \cite{Ghanem:91,Xiu:03,babuska:04,matthies:05,WuanHou:06,Najm:09} and reference therein. These methods have been successfully applied to many UQ problems, where the dimension of the random input is small. However, the number of basis functions increases exponentially with the dimension of random input, i.e., they suffer from the curse of dimensionality of both the input space and the output (solution) space.

In the non-intrusive methods, one can use the MC method or qMC method to solve \eqref{MsStoEllip_ModelEq}-\eqref{MsStoEllip_ModelBC}. However, the convergence rate is slow and the method becomes more expensive when the coefficient $a(x,\omega)$ contains multiscale features. Stochastic collocation methods explore the smoothness of the solutions in the random space and use certain quadrature points and weights to compute the solutions \cite{Xiu:05,Babuska:07}. Exponential convergence can be achieved for smooth solutions, but the quadrature points grow exponentially as the number of random variables increases. Sparse grids \cite{Griebel:04,Webster:08} can reduce the quadrature points to some extent \cite{Griebel:04}. However, the sparse grid method still becomes very expensive when the dimension of randomness is modestly high.

Instead of building random basis functions a priori or 
choosing collocation quadrature points based on the random coefficient $a(x,\omega)$ 
(see Eq.\eqref{ParametrizeRandomCoefficient}), we extract the low dimensional structure and a set of basis functions in the solution space directly from the data (or sampled solutions). Notice that the dimension of the extracted low dimensional space mainly depends on $\kappa_a$ (namely $a_{\min}$ and $a_{\max}$), and very mildly on the dimension of the random input in $a(x,\omega)$. Therefore, the curse of dimension can be alleviated.

\section{Derivation of the new data-driven method} \label{sec:DerivationNewMethod}
In many physical and engineering applications, one needs to obtain the solution of the  Eq.\eqref{MsStoEllip_ModelEq} on a subdomain $\hat{D}\subseteq D$. For instance, in the reservoir simulation one is interested in computing the pressure value $u(x,\omega)$ on a specific subdomain $\hat{D}$. 
  Our method consists of offline and online stages. In the offline stage, we extract the low dimensional structure and a set of data-driven
basis functions from solution samples.  For example, a set of solution samples $\{u(x,\omega_i)\}_{i=1}^{N}$ can be obtained from measurements or generated by solving  \eqref{MsStoEllip_ModelEq}-\eqref{MsStoEllip_ModelBC} with coefficient samples $\{a(x,\omega_i)\}_{i=1}^{N}$.  

Let $V_l=\{u|_{\hat{D}}(x,\omega_1),...,u|_{\hat{D}}(x,\omega_N)\}$ denote the solution samples. We use POD  \cite{HolmesLumleyPOD:1993,Sirovich:1987,Willcox2015PODsurvey}, or a.k.a PCA,  to find the optimal subspace and its orthonormal basis to approximate $V_l$ to certain accuracy. Define the correlation matrix 
$\sigma_{ij}=<u(\cdot,\omega_i), u(\cdot,\omega_j)>_{\hat{D}}, i, j= 1, \ldots, N$. Let the eigenvalues and corresponding eigenfunctions of the correlation matrix be $\lambda_1\ge  \lambda_2 \ge \ldots \ge \ldots \ge \lambda_N \ge 0$ and $\phi_{1}(x)$, $\phi_{2}(x), \ldots, \phi_N(x)$ respectively. The space spanned by the leading $K$ eigenfunctions have the following approximation property to $V_l$.
\begin{proposition}\label{POD_proposition}
\begin{align}
\frac{\sum_{i=1}^{N}\Big|\Big|u(x,\omega_{i})- \sum_{j=1}^{K}<u(\cdot,\omega_{i}),
\phi_j(\cdot)>_{\hat{D}}\phi_j(x)\Big|\Big|_{L^2(\hat{D})}^{2} }{\sum_{i=1}^{N}\Big|\Big|u(x,\omega_{i})\Big|\Big|_{L^2(\hat{D})}^{2}}=\frac{\sum_{s=K+1}^{N}  \lambda_s}{\sum_{s=1}^{N}  \lambda_s}.
\label{Prop_PODError}
\end{align}
\end{proposition} 
First, we expect a fast decay in $\lambda_s$ so that a small $K\ll N$ will be enough to approximate the solution samples well in root mean square sense. 
Secondly, based on the existence of low dimensional structure implied by Theorem \ref{ThmRandomGreenFuncSepaApp}, we expect that the data-driven basis, $\phi_{1}(x)$, $\phi_{2}(x), \ldots, \phi_{K}(x)$,  can almost surely approximate the solution $u|_{\hat{D}}(x,\omega)$ well too 
under some sampling condition (see Section \ref{sec:DetermineNumberOfSamples}) by
\begin{align}
u|_{\hat{D}}(x,\omega) \approx \sum_{j=1}^{K}c_{j}(\omega)\phi_{j}(x), \quad \text{a.s. }
\omega \in \Omega,  \label{RB_expansion}
\end{align} 
where the data-driven basis functions $\phi_{j}(x)$, $j=1,...,K$ are defined on $\hat{D}$. 
The Prop.\ref{POD_proposition} still remains valid in the case $\hat{D}=D$, where the data-driven basis $\phi_{j}(x)$, $j=1,...,K$ can be used in the Galerkin approach to solve \eqref{MsStoEllip_ModelEq}-\eqref{MsStoEllip_ModelBC} on the whole domain $D$ (see Section \ref{sec:GlobalProblem}).

Now the problem is how to find $c_{j}(\omega)$ through an efficient online process given a new realization of $a(x,\omega)$. We prescribe several strategies in different setups.

\subsection{Parametrized randomness}\label{sec:parametrized}
In many applications,  $a(x,\omega)$ is parameterized by $r$ independent random variables,  i.e., 
\begin{align}
a(x,\omega) =  a(x,\xi_{1}(\omega),...,\xi_{r}(\omega)).
\label{ParametrizeRandomCoefficient}
\end{align}
Thus,  the solution can be represented as a function of these random variables as well, i.e., $u(x,\omega) = u(x,\xi_{1}(\omega),...,\xi_{r}(\omega))$.
Let $\vxi(\omega)=[\xi_1(\omega),\cdots,\xi_r(\omega)]^T$ denote the 
random input vector and $\textbf{c}(\omega)=[c_{1}(\omega),\cdots,c_{K}(\omega)]^T$ denote the vector of solution coefficients in \eqref{RB_expansion}. Now,  the problem can be viewed as constructing 
 a map from $\vxi(\omega)$ to $\textbf{c}(\omega)$, denoted by $\textbf{F}:\vxi(\omega)\mapsto \textbf{c}(\omega)$, which is nonlinear. We approximate this nonlinear map through the sample solution set. 
 Given a set of solution samples $\{u(x,\omega_i)\}_{i=1}^{N}$  corresponding to $\{\vxi(\omega_i)\}_{i=1}^{N}$, e.g., by solving \eqref{MsStoEllip_ModelEq}-\eqref{MsStoEllip_ModelBC} with $a(x,\xi_{1}(\omega_i),...,\xi_{r}(\omega_i))$,
from which the set of data driven basis $\phi_{j}(x), j=1,...,K$ is obtained using POD as described above,  we can easily compute the projection coefficients $\{\textbf{c}(\omega_i)\}_{i=1}^{N}$ of $u|_{\hat{D}}(x,\omega_i)$ on $\phi_{j}(x)$, $j=1,...,K$, i.e., $c_j(\omega_i)=<u(x,\omega_i), \phi_{j}(x)>_{\hat{D}}$.  From the data set,
$F(\vxi(\omega_i))= \textbf{c}(\omega_i)$, $i=1,...,N$, we construct the map $\textbf{F}$. Note the significant dimension reduction by reducing the map $\vxi(\omega)\mapsto u(x,\omega)$ to the map  $\vxi(\omega)\mapsto \textbf{c}(\omega)$.
We provide a few ways to construct $\textbf{F}$. 
\begin{itemize}
\item Interpolation. 
\\
When the dimension of the random input $r$ is small or moderate, one can use interpolation. In particular, if the solution samples correspond to $\vxi$ located on a (sparse) grid, standard polynomial interpolation can be used to approximate the coefficient $c_j$ at a new point of $\vxi$. If the solution samples correspond to $\vxi$ at scattered points or the dimension of the random input $r$ is moderate or high, one can first find the a few nearest neighbors to a new point efficiently using $k-d$ tree \cite{wald2006building} and then use moving least square approximation centered at the new point. 
\item Neural network.
\\
 When the dimension of the random input $r$ is high, interpolation approach becomes expensive and less accurate, we show that neural network seems to provide a satisfactory solution.
\end{itemize}
More implementation details will be explained in Section \ref{sec:NumericalExperiments} and the map $\textbf{F}$ is plotted based on interpolation.

%
  
In the online stage, one can compute the solution $u(x,\omega)$ to \eqref{MsStoEllip_ModelEq}-\eqref{MsStoEllip_ModelBC} using the constructed mapping $\textbf{F}$. 
Given a new realization of $a(x,\xi_{1}(\omega_i),...,\xi_{r}(\omega_i))$, we plug $\vxi(\omega)$ into the constructed map $\textbf{F}$ and directly obtain $\textbf{c}(\omega)=\textbf{F}(\vxi(\omega))$ which are the projection coefficients of the solution on the data-driven basis. So we can quickly obtain the new solution $u|_{\hat{D}}(x,\omega)$ using Eq.\eqref{RB_expansion}, where the computational time is negligible. Once we obtain the numerical solutions, we can use them to compute statistical quantities of interest, such as mean, variance, and joint probability distributions.  
\begin{remark}
In Prop.\ref{POD_proposition} we construct the data-driven basis functions from eigen-decomposition of the correlation matrix  associated with the solution samples. Alternatively we can subtract the mean from the solution samples, compute the covariance matrix, and construct the basis functions from eigen-decomposition of the covariance matrix. 
\end{remark}

\subsection{Galerkin approach} \label{sec:GlobalProblem}
\noindent
In the case $\hat{D}=D$, we can solve \eqref{MsStoEllip_ModelEq}-\eqref{MsStoEllip_ModelBC} on 
the whole domain $D$ by the standard Galerkin formulation using the data driven basis for a new realization of $a(x,\omega)$.

Once the data driven basis $\phi_{j}(x)$, $j=1,...,K$, which are defined on the domain $D$, are obtained from solution samples in the offline stage, 
 given a new realization of the coefficient $a(x,\omega)$, we approximate the corresponding solution as 
\begin{align}
u(x,\omega) \approx \sum_{j=1}^{K}c_{j}(\omega)\phi_{j}(x), \quad \text{a.s. }
\omega \in \Omega,  \label{RB_expansion2}
\end{align} 
and use the Galerkin projection to determine the coefficients $c_{j}(\omega)$, $j=1,...,K$ by solving the following linear system in the online stage,
\begin{align}
\sum_{j=1}^K \int_{D}a(x,\omega)c_{j}(\omega)\nabla\phi_{j}(x)\cdot\nabla\phi_{l}(x)dx  = \int_{D}f(x)\phi_{l}(x)dx, 
 \quad l=1,...,K.
 \label{GalerkinSystem}
\end{align}

\begin{remark}
The computational cost of solving the linear system \eqref{GalerkinSystem} is small compared to using a Galerkin method, such as the finite element method, directly for $u(x,\omega)$ because $K$ is much smaller than the degree of freedom needed to discretize $u(x,\omega)$. 
\end{remark}

If the coefficient $a(x,\omega)$ has the affine parameter dependence property \cite{RozzaPatera:2007}, 
i.e., $ a(x,\omega) = \sum_{n=1}^{r} a_{n}(x)\xi_{n}(\omega) $, we compute the terms that do not depend on randomness, including $\int_{D}a_{n}(x)\nabla\phi_{j}(x)\cdot\nabla\phi_{l}(x)dx$,
$\int_{D}f(x)\phi_{l}(x)dx$, $j,l=1,...,K$ and save them in the offline stage.  This leads to considerable savings in assembling the stiffness matrix for each new realization of the coefficient $a(x,\omega)$ in the online stage.
Of course, the affine form is automatically parametrized. Hence, one can also construct the map $\textbf{F}:\vxi(\omega)\mapsto \textbf{c}(\omega)$ as described in the previous Section \ref{sec:parametrized}. 
If the coefficient $a(x,\omega)$ does not admit an affine form, we can apply the empirical interpolation method (EIM) \cite{PateraMaday:2004} to convert $a(x,\omega)$ into an affine form. 

\subsection{Least square fitting from direct measurements at selected locations}\label{sec:LS}
In many applications, only samples (data) or measurements of $u(x,\omega)$ is available while the model of $a(x,\omega)$ or its realization is not known. In this case, we propose to compute the coefficients $\textbf{c}$ by least square fitting the measurements (values) of $u(x,\omega)$ at appropriately selected locations. First, as before, from a set of solutions samples, $u(x_j, \omega_i)$, measured on a mesh  $x_j \in \hat{D}, j=1, \ldots, J$, one finds a set of data driven basis $\phi_1(x_j), \ldots, \phi_K(x_j)$, e.g. using POD. For a new solution $u(x,\omega)$ measured at $x_1, x_2, \ldots, x_M$, one can set up the following least square problem to find $\vec{c}=[c_1, \ldots, c_K]^T$ such that $u(x,\omega)\approx \sum_{k=1}^K c_k\phi_k(x)$:
\begin{equation}
\label{eq:LS}
B \vec{c}=\vec{y}, \quad  \vec{y}=[u(x_1,\omega), \ldots, u(x_M,\omega)]^T, B=[\boldsymbol{\phi}^M_1, \ldots, \boldsymbol{\phi}^M_K]\in R^{M\times K},
\end{equation}
where $\boldsymbol{\phi}^M_k=[\phi_k(x_1), \ldots, \phi_k(x_M)]^T$. The key issue in practice is the conditioning of the least square problem \eqref{eq:LS}. One way is to select the measurement (sensor) locations $x_1, \ldots x_M$ such that rows of $B$ are as decorrelated as possible. We adopt the approach proposed in \cite{Kutz2017Sensor} in which a QR factorization with pivoting for the matrix of data driven basis is used to determine the measurement locations. More specifically, let $\Phi=[\boldsymbol{\phi}_1, \ldots, \boldsymbol{\phi}_K]\in R^{J\times K}$, $\boldsymbol{\phi}_k=[\phi_k(x_1), \ldots, \phi_k(x_J)]^T$. If $M=K$, QR factorization with column pivoting is performed on $\Phi^T$. If $M>K$, QR factorization with pivoting is performed on $\Phi\Phi^T$. The first $M$ pivoting indices provide the measurement  locations. More details can be found in \cite{Kutz2017Sensor} and Section \ref{sec:NumericalExperiments}.

\subsection{Extension to problems with parameterized force functions} \label{sec:ExtensionTOManyFx}
\noindent
In many applications, we are interested in solving multiscale elliptic PDEs with random coefficients in the multiquery setting. A model problem is given as follows, 
\begin{align}
-\nabla\cdot\big(a(x,\omega)\nabla u(x,\omega)\big) &= f(x,\theta), 
\quad x\in D, \quad \omega\in\Omega, \quad \theta \in \Theta,  \label{MsStoEllipMultiquery_Eq}\\
u(x,\omega)&= 0, \quad \quad x\in \partial D, \label{MsStoEllipMultiquery_BC}
\end{align}
where the setting of the coefficient $a(x,\omega)$ is the same as \eqref{ParametrizeRandomCoefficient}. 
Notice that the force function $f(x,\theta)$ is parameterized by $\theta\in \Theta$ and $\Theta$ is a 
parameter set. In practice, we often need to solve the problem \eqref{MsStoEllipMultiquery_Eq}-\eqref{MsStoEllipMultiquery_BC} with multiple force functions $f(x,\theta)$, which is known as the multiquery problem. It is computationally expensive to solve this kind of problem using traditional methods. 

Some attempts have been made in \cite{ZhangCiHouMMS:15,hou2019model}, where a data-driven stochastic method has been proposed to solve PDEs with random and multiscale coefficients. When the number of random variables in the coefficient $a(x,\omega)$ is small, say less than 10, the methods developed in \cite{ZhangCiHouMMS:15,hou2019model} can provide considerable savings in solving multiquery problems. However, they suffer from the curse of dimensionality of both the input space and the output (solution) space. 
Our method using data driven basis, which is based on extracting a low dimensional structure in the output space,  can be directly adopted to this situation. Numerical experiments are presented in Section \ref{sec:NumericalExperiments}.


\subsection{Determine a set of good learning samples} \label{sec:DetermineNumberOfSamples}
\noindent
A set of good solution samples is important for the construction of data-driven basis in the offline stage. 
Here we provide an error analysis which is based on the finite element formulation. However, the results extend to general Galerkin formulation. 
First, we make a few assumptions.  
\begin{assumption} \label{assumption2}
	Suppose $a(x,\omega)$ has the following property: given $ \delta_1 > 0$, there exists an integer $N_{\delta_1}$ and a choice of snapshots $\{a(x,\omega_i)\}$, $i=1,...,N_{\delta_1}$ such that
	\begin{align} 
	\mathds{E}\left[\inf_{1\le i\le N_{\delta_1}} \big|\big|a(x,\omega) - a(x,\omega_i)\big|\big|_{L^\infty(D)}\right] \le \delta_1.  \label{asd}
	\end{align}
\end{assumption}
Let $\{a(x,\omega_i)\}_{i=1}^{N_{\delta_1}}$ denote the samples of the random coefficient. When the coefficient has an affine form, we can verify Asm. \ref{assumption2} and  provide a constructive way to sample snapshots 
$\{a(x,\omega_i)\}_{i=1}^{N_{\delta_1}}$ if we know the distribution of the random variables $\xi_{i}(\omega)$, $i=1,...,r$.

Let $V_h\subset H_{0}^{1}(D)$ denote a finite element space that is spanned by nodal basis functions on a mesh with size $h$ and $\tilde{V}_h \subset V_h$ denote the space spanned by the data-driven basis $\{\phi_{j}(x)\}_{j=1}^{K}$. We assume the mesh size is fine enough so that the finite element space can approximate the solutions to the underlying PDEs well. For each $a(x,\omega_i)$, let $u_h(x,\omega_i)\in V_h$ denote the FEM solution and $\tilde{u}_h(x,\omega_i)\in \tilde{V}_h$ denote the projection on the data-driven basis $\{\phi_{j}(x)\}_{j=1}^{K}$. 
\begin{assumption} \label{assumption3}
	Given $\delta_2 > 0$, we can find a set of data-driven basis, $\phi_1, \ldots, \phi_{K_{\delta_2}}$ such that 
	\begin{align}
	||u_h(x,\omega_i)-\tilde{u}_h(x,\omega_i)||_{L^2(D)} \le \delta_2,\ \forall 1\le i \le K_{\delta_2}, \label{equation_asumption2}
	\end{align}
	where $\tilde{u}_h(x,\omega_i)$ is the $L^2$ projection of $u_h(x,\omega_i)$ onto the space spanned by $\phi_1, \ldots, \phi_{K_{\delta_2}}$.
\end{assumption}
Asm.\ref{assumption3} can be verified by setting the threshold in the POD method; see Prop.\ref{POD_proposition}. Now we present the following error estimate. 

\begin{theorem} \label{error_theorem1}
	Under Assumptions \ref{assumption2}-\ref{assumption3}, for any $\delta_i > 0$, $i=1,2$, we can choose the samples of the random coefficient $\{a(x,\omega_i)\}_{i=1}^{N_{\delta_1}}$ and the threshold in constructing the data-driven basis  accordingly, such that  
	\begin{align}
	\mathds{E}\left[\big|\big|u_h(x,\omega) - \tilde{u}_h(x,\omega)\big|\big|_{L^2(D)}\right] 
	\leq C\delta_1 + \delta_2,  \label{error_theorem}
	\end{align}
	where $C$ depends on $a_{\min}$, $f(x)$ and the domain $D$.
\end{theorem}
\begin{proof}
	Given a coefficient $a(x,\omega)$, let $u_h(x,\omega)$ and $\tilde{u}_h(x,\omega)$ be the corresponding FEM solution and data-driven solution, respectively. We have
	\begin{align} \label{proof_basis_error}
	&\big|\big|u_h(x,\omega) - \tilde{u}_h(x,\omega)\big|\big|_ {L^2(D)} \nonumber\\
	\le &\big|\big|u_h(x,\omega) - u_h(x,\omega_i)\big|\big|_{L^2(D)} + 
	\big|\big|u_h(x,\omega_i) - \tilde{u}_h(x,\omega_i)\big|\big|_{L^2(D)} + \big|\big|\tilde{u}_h(x,\omega_i) - \tilde{u}_h(x,\omega)\big|\big|_{L^2(D)}, \nonumber\\
	:=& I_1 + I_2+ I_3,
	\end{align}
	where $u_h(x,\omega_i)$ is the solution corresponding to the coefficient $a(x,\omega_i)$ and $\tilde{u}_h(x,\omega_i)$ is its projection. Now we estimate the error term $I_1$ first. In the sense of weak form, we have
	\begin{align}
	\int_{D}a(x,\omega)\nabla u_h(x,\omega)\cdot \nabla v_h(x)dx=\int_{D}f(x)v_h(x), \quad \text{for all} \quad v_h(x)\in V_h,
	\label{FEMsolutionWeakForm1} 
	\end{align}
	and
	\begin{align} 
	\int_{D}a(x,\omega_i)\nabla u_h(x,\omega_i)\cdot\nabla v_h(x)dx=\int_{D}f(x)v_h(x), \quad \text{for all} \quad v_h(x)\in V_h.
	\label{FEMsolutionWeakForm2} 
	\end{align}
	Subtracting the variational formulations \eqref{FEMsolutionWeakForm1}-\eqref{FEMsolutionWeakForm2} for 
	$u_h(x,\omega)$ and $u_h(x,\omega_i)$, we find that for all $v_h(x)\in V_h$, 
	\begin{align}
	\int_{D}a(x,\omega)\nabla (u_h(x,\omega)-u_h(x,\omega_i))\cdot\nabla v_h(x)dx
	=-\int_{D}(a(x,\omega)-a(x,\omega_i))\nabla u_h(x,\omega_i)\cdot\nabla v_h(x).  
	\label{FEMsolutionWeakForm3} 
	\end{align}
	Let $w_h(x)=u_h(x,\omega)-u_h(x,\omega_i)$ and $L(v_h)=-\int_{D}(a(x,\omega)-a(x,\omega_i))\nabla u_h(x,\omega_i)\cdot\nabla v_h(x)$ denote the linear form. Eq.\eqref{FEMsolutionWeakForm3} means that 
	$w_h(x,\omega)$ is the solution of the weak form $\int_{D}a(x,\omega)\nabla w_h\cdot\nabla v_h(x)dx=L(v_h)$. Therefore, we have  
	\begin{align}
	\big|\big|w_h(x)\big|\big|_ {H^1(D)}\leq \frac{||L||_{H^1(D)}}{a_{\min}}.
	\label{EstimateError}
	\end{align}
	Notice that 
	\begin{align}
	||L||_{H^1(D)} =\max_{||v_h||_{H^1(D)}=1}|L(v_h)|&\leq ||a(x,\omega)-a(x,\omega_i)||_{L^\infty(D)}
	||u_h(x,\omega_i)||_{H^1(D)},\nonumber \\
	&\leq ||a(x,\omega)-a(x,\omega_i)||_{L^\infty(D)}\frac{||f(x)||_{H^1(D)}}{a_{\min}}.
	\label{EstimateError2}
	\end{align}
	Since $w_h(x)=0$ on $\partial D$, combining Eqns.\eqref{EstimateError}-\eqref{EstimateError2} and using the Poincar\'e inequality on $w_h(x)$, we obtain an estimate for the term $I_1$ as 
	\begin{align}
	\big|\big|u_h(x,\omega) - u_h(x,\omega_i)\big|\big|_{L^2(D)} &\leq C_1\big|\big|u_h(x,\omega) - u_h(x,\omega_i)\big|\big|_{H^1(D)} \nonumber \\
	&\leq C_1||a(x,\omega)-a(x,\omega_i)||_{L^\infty(D)}\frac{||f(x)||_{H^1(D)}}{a_{\min}^2},
	\label{EstimateError3}
	\end{align}
	where $C_1$ only depends on the domain $D$. For the term $I_3$ in Eq.\eqref{proof_basis_error}, we can similarly get  
	\begin{align}
	\big|\big|\tilde{u}_h(x,\omega_i) - \tilde{u}_h(x,\omega)\big|\big|_{L^2(D)} 
	\leq C_1||a(x,\omega)-a(x,\omega_i)||_{L^\infty(D)}\frac{||f(x)||_{H^1(D)}}{a_{\min}^2}.
	\label{I3}
	\end{align}
	The term $I_2$ in Eq.\eqref{proof_basis_error} can be controlled according to the Asm.\ref{assumption3}. 
	Combining the estimates for terms $I_1$, $I_2$ and $I_3$ and integrating 
	over the random space, we prove the theorem. 
\end{proof}   
 
Theorem \ref{error_theorem1} indicates that the error between $u_h(x,\omega)$ and its approximation $\tilde{u}_h(x,\omega)$ using the data driven basis consists of two parts. The first part depends on how well the random coefficient is sampled. While the second part depends on the truncation threshold in constructing the data-driven basis from the solution samples. In practice, a balance of these two factors and the discretization error (of the numerical method used to solve the PDEs) gives us the guidance on how to choose solution samples and truncation threshold in the POD method to achieve optimal accuracy.  Again, the key advantage for our data driven approach for this form of elliptic PDEs is the low dimensional structure in the solution space which provides a significant dimension reduction.

\section{Numerical experiments} \label{sec:NumericalExperiments}
\noindent
In this section we will present various numerical experiments to demonstrate the accuracy and efficiency of our proposed data-driven method. 
\subsection{An example with five random variables}\label{sec:Example1}
\noindent
We consider a multiscale elliptic PDE with a random coefficient that is defined on a square domain $D=[0,1]\times[0,1]$,
\begin{align}\label{randommultiscaleelliptic}
\begin{split}
-\nabla\cdot(a(x,y,\omega)\nabla u(x,y,\omega)) &= f(x,y), \quad (x,y)\in D, \omega\in\Omega,\\
u(x,y,\omega)&=0, \quad \quad (x,y)\in\partial D.
\end{split}
\end{align}
In this example, the coefficient $a(x,y,\omega)$ is defined as 
\begin{align} 
a(x,y,\omega) =& 0.1 + \frac{2+p_1\sin(\frac{2\pi x}{\epsilon_1})}{2-p_1\cos(\frac{2\pi y}{\epsilon_1})} \xi_1(\omega)
+ \frac{2+p_2\sin(\frac{2\pi (x+y)}{\sqrt{2}\epsilon_2})}{2-p_2\sin(\frac{2\pi (x-y)}{\sqrt{2}\epsilon_2})}\xi_2(\omega)
+ \frac{2+p_3\cos(\frac{2\pi (x-0.5)}{\epsilon_3})}{2-p_3\cos(\frac{2\pi (y-0.5)}{\epsilon_3})}\xi_3(\omega) \nonumber \\
&+ \frac{2+p_4\cos(\frac{2\pi (x-y)}{\sqrt{2}\epsilon_4})}{2-p_4\sin(\frac{2\pi (x+y)}{\sqrt{2}\epsilon_4})}\xi_4(\omega)
+ \frac{2+p_5\cos(\frac{2\pi (2x-y)}{\sqrt{5}\epsilon_5})}{2-p_5\sin(\frac{2\pi (x+2y)}{\sqrt{5}\epsilon_5})}\xi_5(\omega), \label{coefficientofexample1}
\end{align}
where $[\epsilon_1,\epsilon_2,\epsilon_3,\epsilon_4,\epsilon_5]=[\frac{1}{47},\frac{1}{29},\frac{1}{53},\frac{1}{37},\frac{1}{41}]$, $[p_1,p_2,p_3,p_4,p_5]=[1.98,1.96,1.94,1.92,1.9]$, and $\xi_i(\omega)$, $i=1,...,5$ are i.i.d. uniform random variables in $[0,1]$. The contrast ratio in the coefficient \eqref{coefficientofexample1} is $\kappa_a\approx 4.5\times 10^3$. The force function  is $f(x,y) = \sin(2\pi x)\cos(2\pi y)\cdot I_{D_2}(x,y)$, where $I_{D_2}$ is an indicator function defined on $D_2=[\frac{1}{4},\frac{3}{4}]\times[\frac{1}{16},\frac{5}{16}]$. 
The coefficient \eqref{coefficientofexample1} is highly oscillatory in the physical space. 
Therefore, one needs a fine discretization to resolve the small-scale variations in the problem. 
We shall show results for the solution to \eqref{randommultiscaleelliptic} with coefficient \eqref{coefficientofexample1} in: (1) a restricted subdomain $D_1=[\frac{1}{4},\frac{3}{4}]\times[\frac{11}{16},\frac{15}{16}]$ away from the support $D_2$ of the source term $f(x,y)$; and (2) the full domain $D$.

In all of our numerical experiments, we use the same uniform triangulation to implement the standard FEM and  choose mesh size $h=\frac{1}{512}$ in order to resolve the multiscale information. We use $N=2000$ samples in the offline stage to construct the data-driven basis and determine the number of basis $K$ according to the decay rate of the eigenvalues of the correlation matrix of the solution samples, i.e., $\sigma_{ij}=<u(x,\omega_i),u(x,\omega_j)>, i,j=1, \dots, N$.

In Figure \ref{fig:Example1localeigenvalues}, we show the decay property of eigenvalues. Specifically, we show the magnitude of the eigenvalues in Figure \ref{fig:Example1localeigenvalues1a} and the ratio of the accumulated sum of the leading eigenvalues over the total sum in Figure \ref{fig:Example1localeigenvalues1b}. These results and Prop.\ref{POD_proposition} imply that a few leading eigenvectors will provide a set of  data-driven basis that can approximate all solution samples well. 
\begin{figure}[tbph]
	\centering 
	\begin{subfigure}[b]{0.45\textwidth}
		\includegraphics[width=1.0\linewidth]{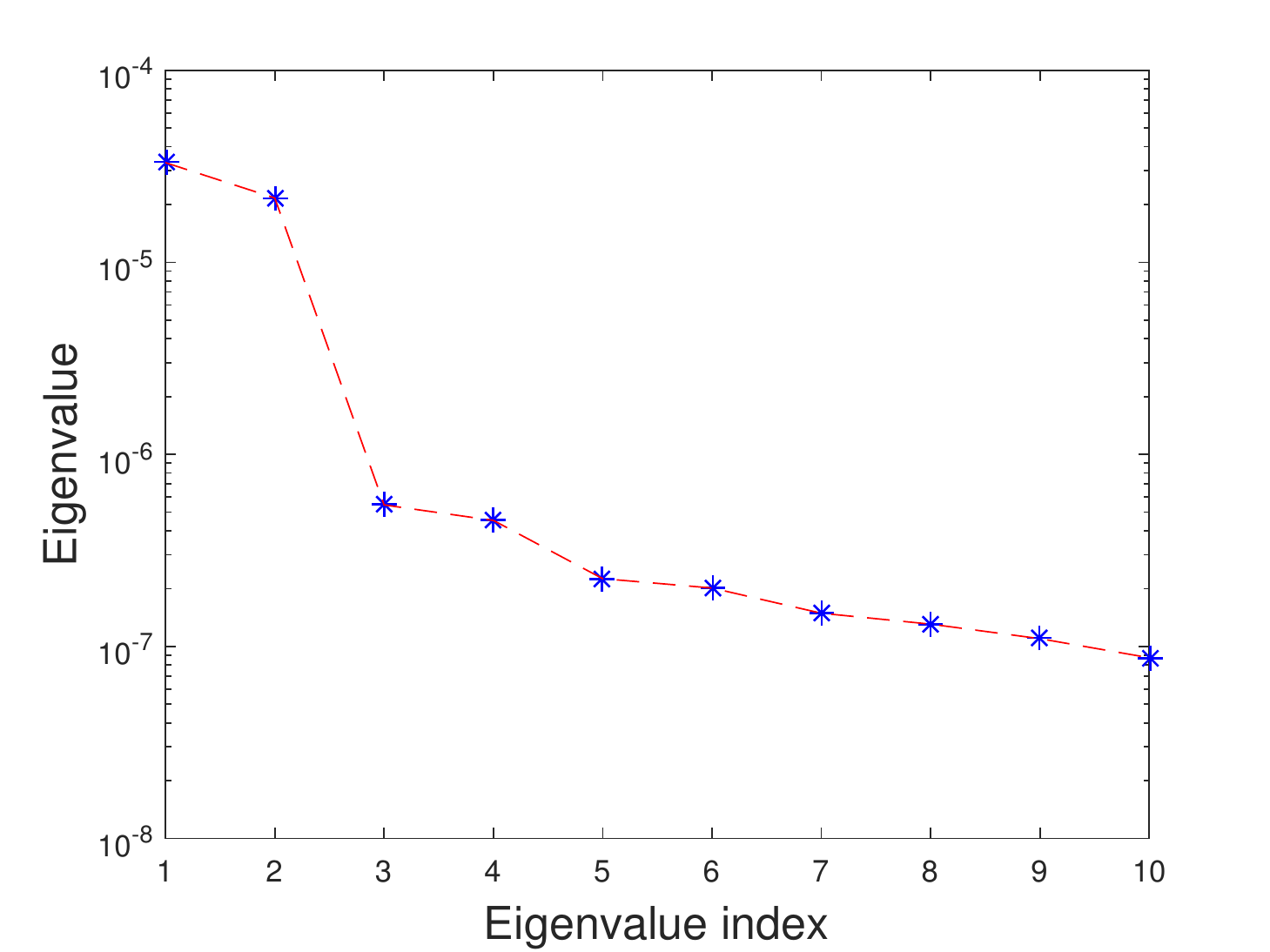} 
		\caption{ Decay of eigenvalues.}
		\label{fig:Example1localeigenvalues1a}
	\end{subfigure}
	\begin{subfigure}[b]{0.45\textwidth}
		\includegraphics[width=1.0\linewidth]{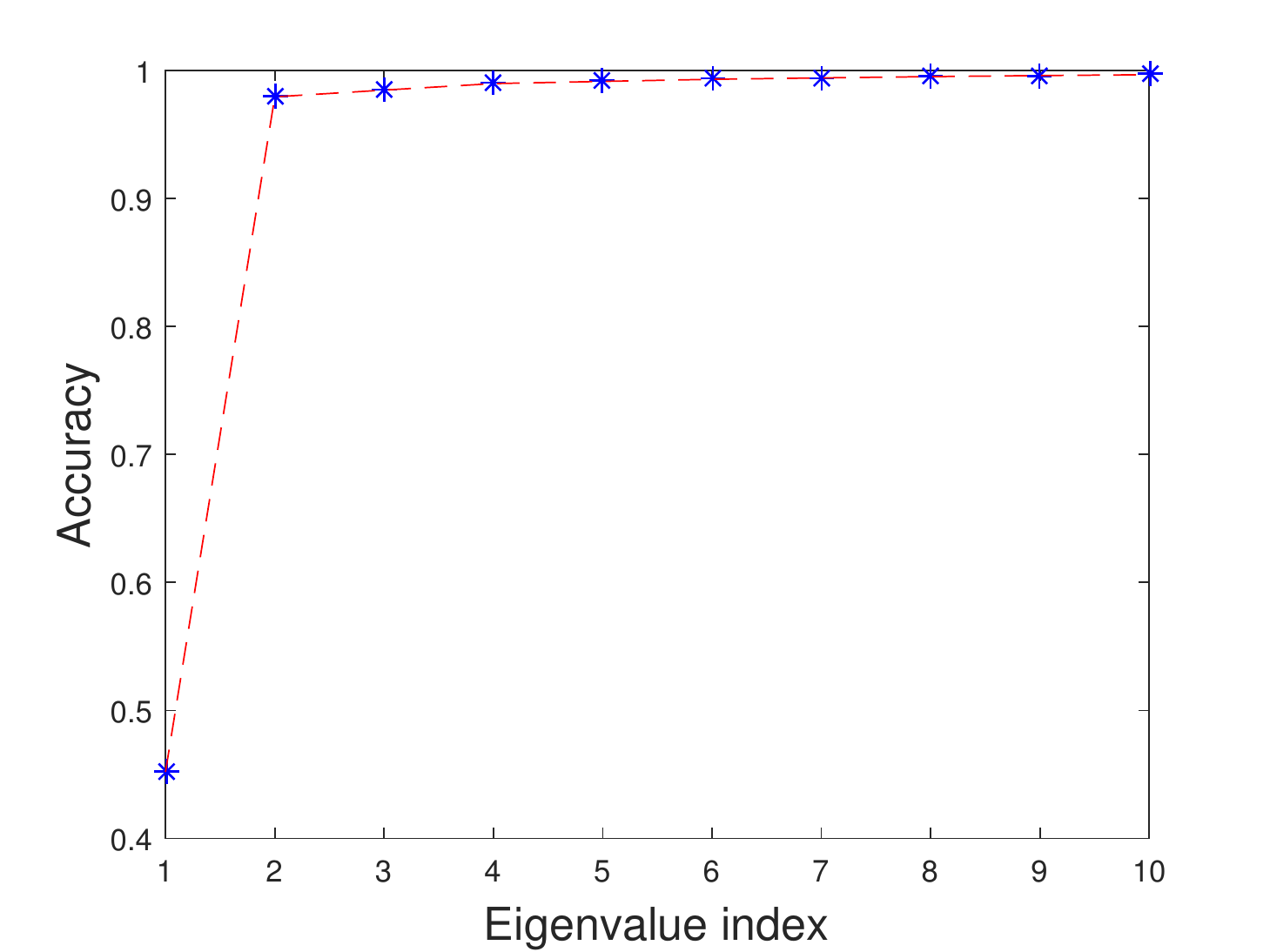} 
		\caption{
			 $1-\sqrt{\sum_{j=n+1}^{N}\lambda_{j}/\sum_{j=1}^{N}\lambda_{j}}$, $n=1,2,...$.} 
		\label{fig:Example1localeigenvalues1b}
	\end{subfigure}
	\caption{The decay properties of the eigenvalues in the local problem of Sec.\ref{sec:Example1}.}
	\label{fig:Example1localeigenvalues}
\end{figure} 

After we construct the data-driven basis, we use the spline interpolation to approximate the mapping $\textbf{F}:\vxi \mapsto \textbf{c}(\vxi)$. Notice that the coefficient of \eqref{coefficientofexample1}
is parameterized by five i.i.d. random variables. We can partition the random space  $[\xi_1(\omega),\xi_2(\omega),\cdots,\xi_5(\omega)]^T\in [0,1]^5$ into a set of uniform grids in order to 
construct the mapping $\textbf{F}$. Here we choose $N_1=9^5$ samples. We remark that we can choose other sampling strategies, such as sparse-grid points and Latin hypercube points.  In Figure \ref{fig:Example1localbasismapping}, we show the profiles of the first two data-driven basis functions $\phi_{1}$ and $\phi_{2}$ and the plots of the mappings $c_1(\xi_1,\xi_2;\xi_3,\xi_4,\xi_5)$ and $c_2(\xi_1,\xi_2;\xi_3,\xi_4,\xi_5)$ with fixed $[\xi_3,\xi_4,\xi_5]^T=[0.25, 0.5, 0.75]^T$. One can see that the data-driven basis functions contain multiscale features and the mapping $c_1(\xi_1,\xi_2;\xi_3,\xi_4,\xi_5)$ and $c_2(\xi_1,\xi_2;\xi_3,\xi_4,\xi_5)$ are smooth with respect to $\xi_i$, $i=1,2$. The behaviors of other data-driven basis functions and the mappings are similar (not shown here).


\begin{figure}[tbph]
	\centering	
	\begin{subfigure}[b]{0.45\textwidth}
		\includegraphics[width=1.0\linewidth]{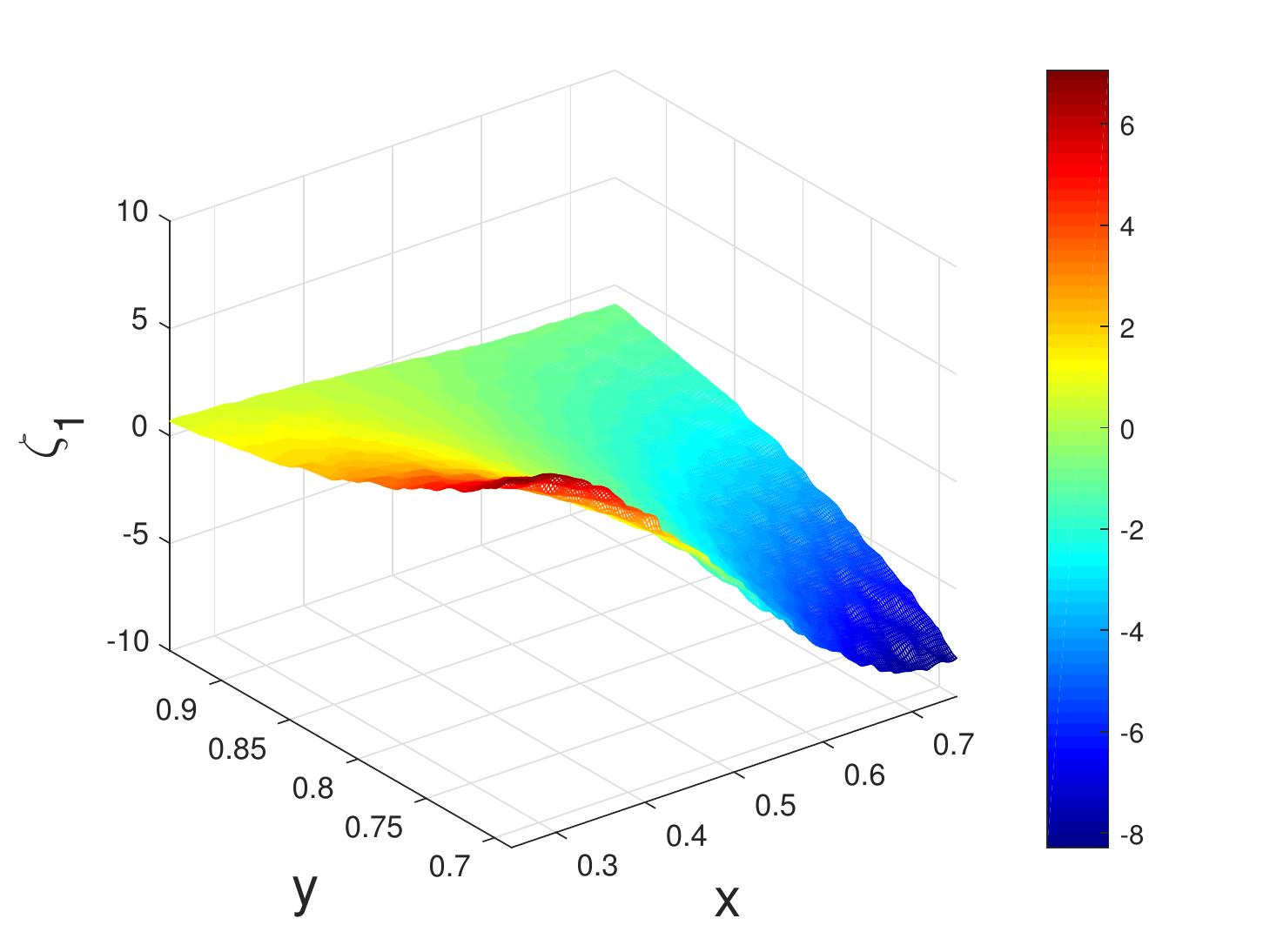}\\
		\includegraphics[width=1.0\linewidth]{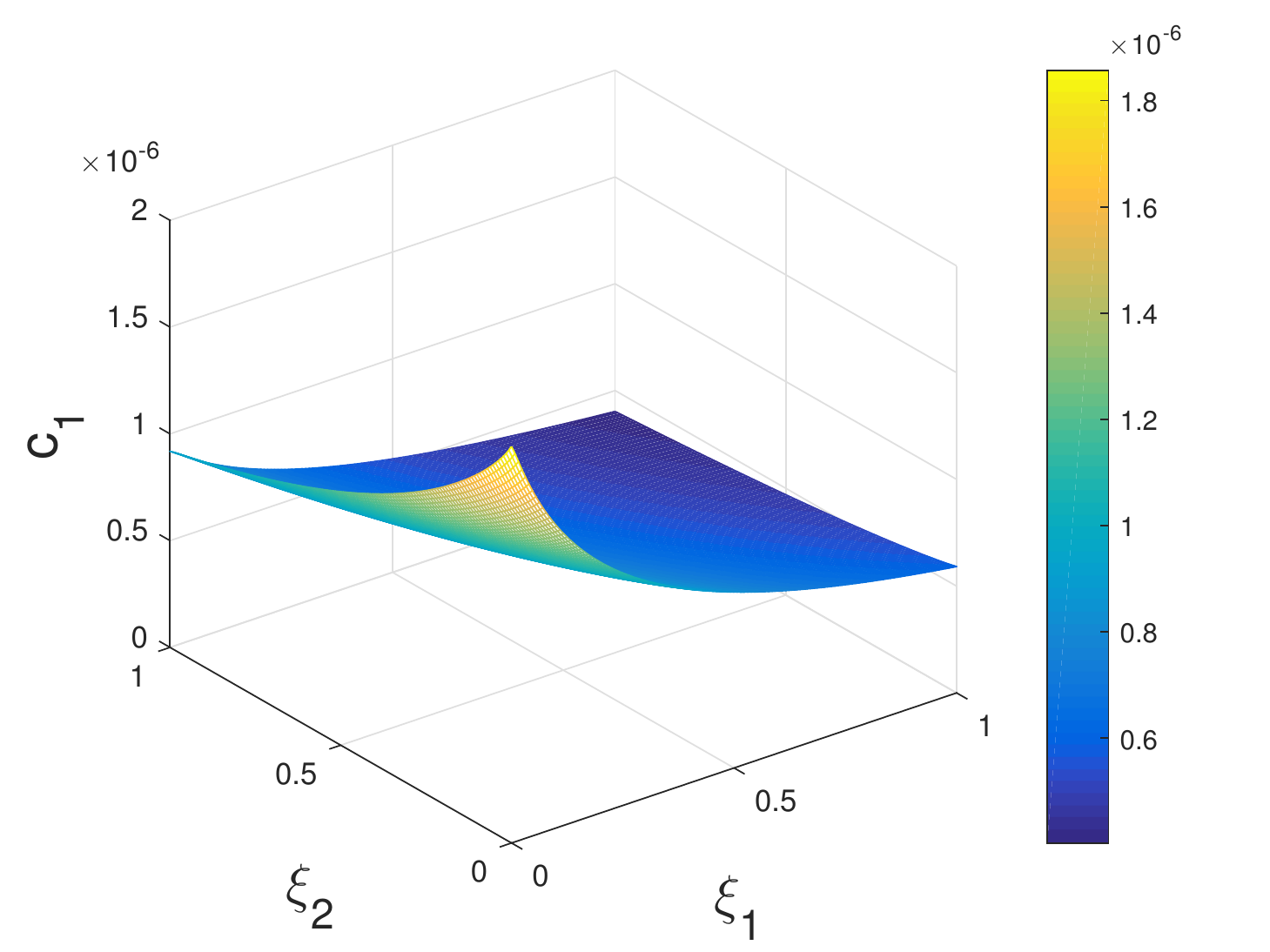}\\
	\end{subfigure}
	\begin{subfigure}[b]{0.45\textwidth}
		\includegraphics[width=1.0\linewidth]{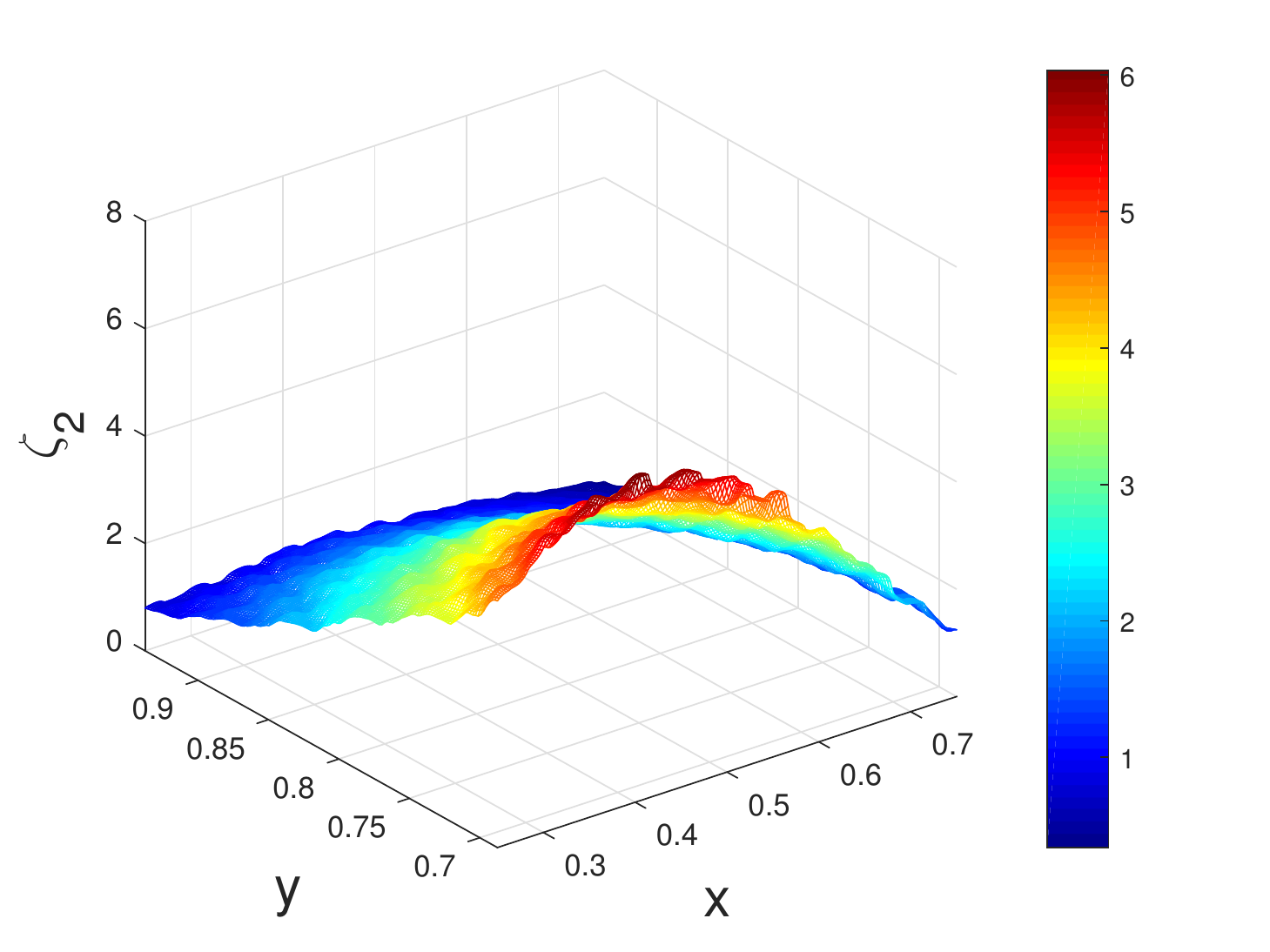}\\
		\includegraphics[width=1.0\linewidth]{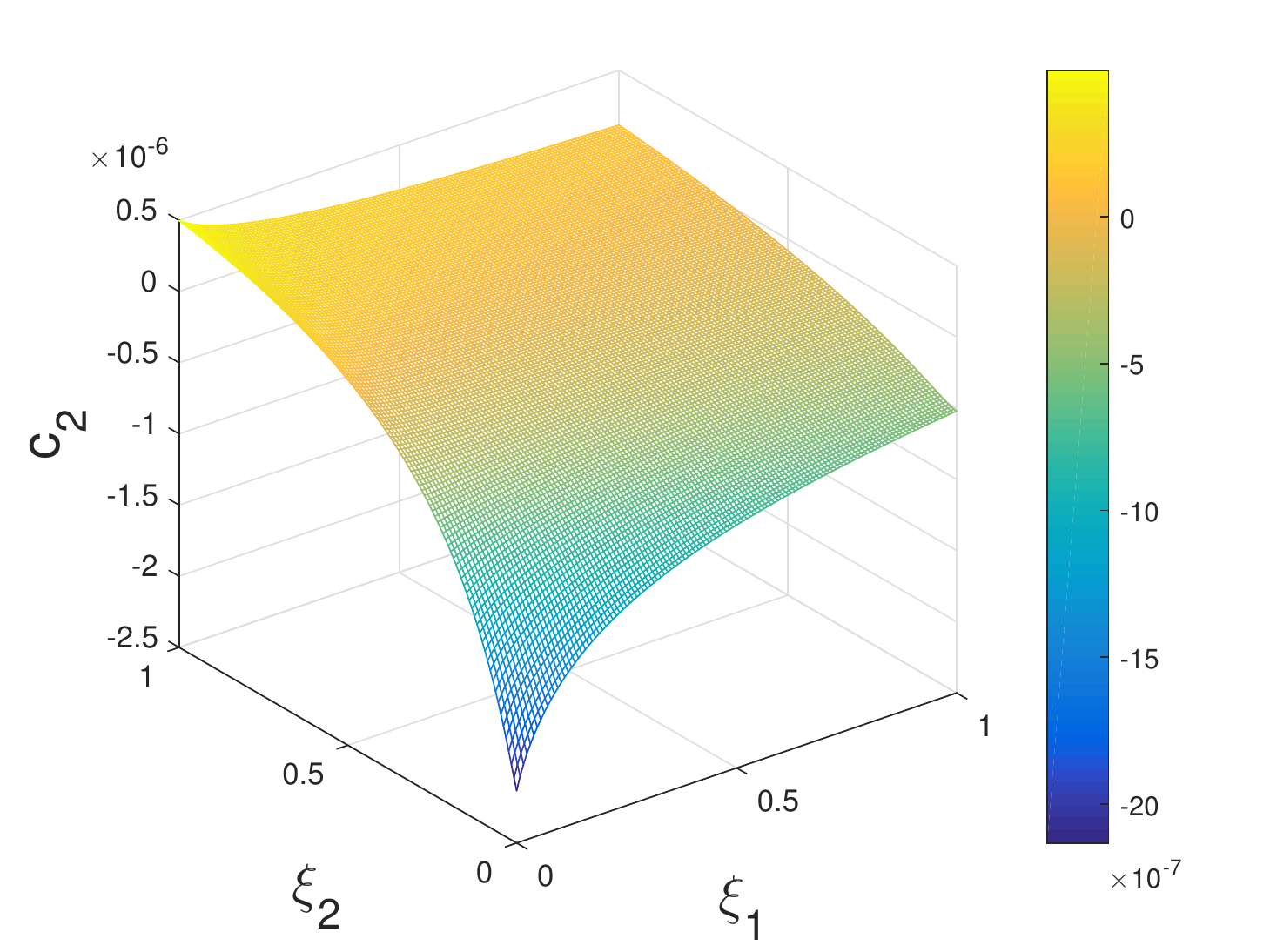}\\
	\end{subfigure}
	\caption{Plots of data-driven basis $\phi_{1}$ and $\phi_{2}$ and  mappings $c_1(\xi_1,\xi_2;\xi_3,\xi_4,\xi_5)$ and $c_2(\xi_1,\xi_2;\xi_3,\xi_4,\xi_5)$ with fixed $[\xi_3,\xi_4,\xi_5]^T=[0.25, 0.5, 0.75]^T$.}
	\label{fig:Example1localbasismapping}
\end{figure} 

Once we get the mapping $\textbf{F}$, the solution corresponding to a new realization $a(x,\vxi(\omega))$ can be constructed easily by finding $ \textbf{c}(\vxi)$ and plugging in the approximation \eqref{RB_expansion}. In Figure \ref{fig:Example1locall2err}, we show the mean relative $L^2$ and $H^1$ errors of the testing error and projection error. The testing error is the error between the numerical solution obtained by our mapping method and the reference solution obtained by the FEM on the same fine mesh used to compute the sample solutions. The projection error is the error between the FEM solution and its projection on the space spanned by data-driven basis, i.e. the best possible approximation error. For the experiment, only four data-driven basis are needed to achieve a relative error less than $1\%$ in $L^2$ norm and less than $2\%$ in $H^1$ norm. Moreover, the numerical solution obtained by our mapping method is close to the projection solution, which is the best approximation of the reference solution by the data-driven basis. This is due to the smoothness of the mapping. Notice that the computational time of the mapping method is almost negligible. In practice, when the number of basis is 10, it takes about $0.0022s$ to get a new solution by the mapping method, whereas the standard FEM takes $0.73s$.    
\begin{figure}[tbph]
	\centering
	\begin{subfigure}[b]{0.45\textwidth}
	\includegraphics[width=1.0\linewidth]{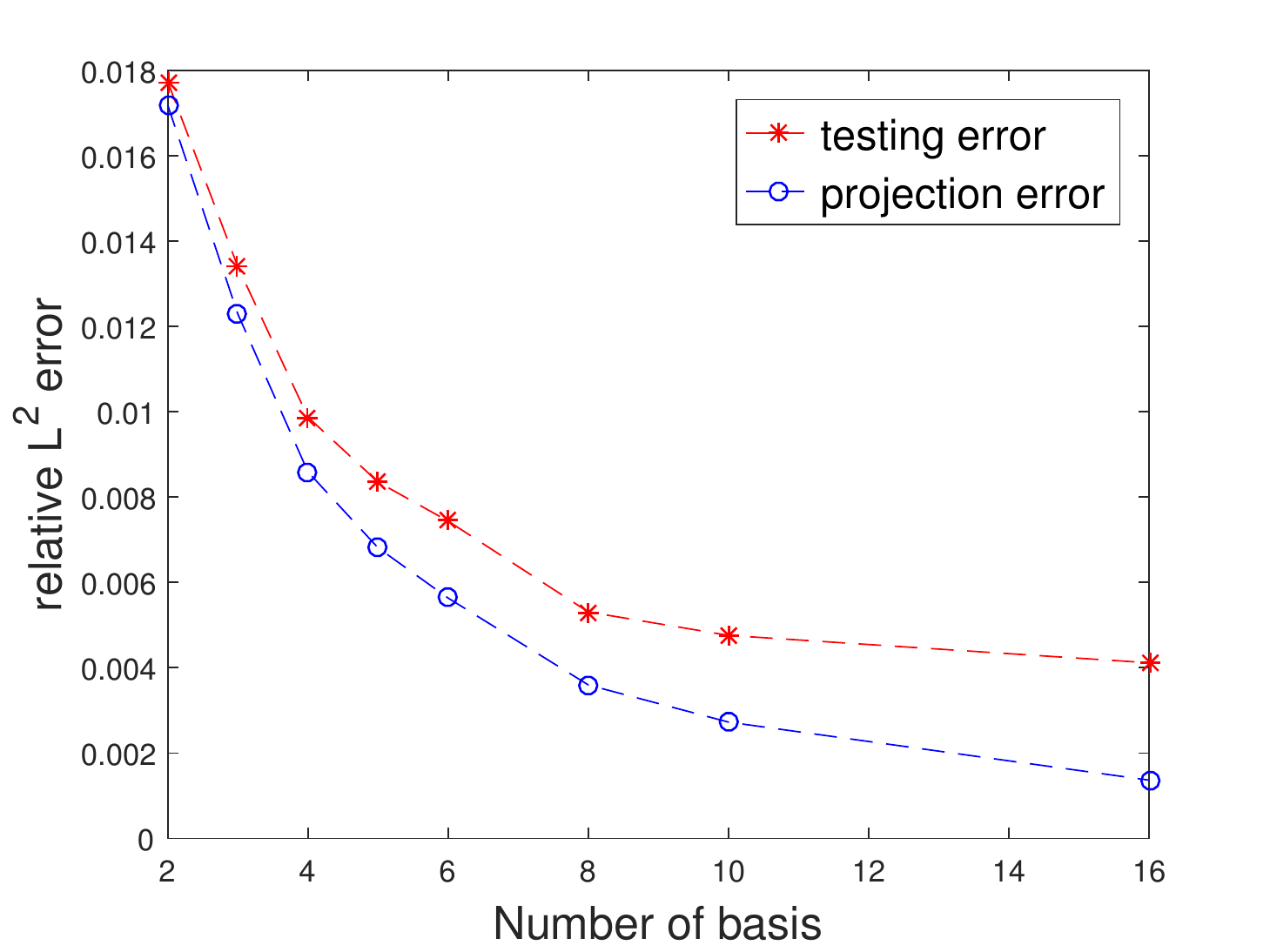}
	\end{subfigure}
	\begin{subfigure}[b]{0.45\textwidth}
	\includegraphics[width=1.0\linewidth]{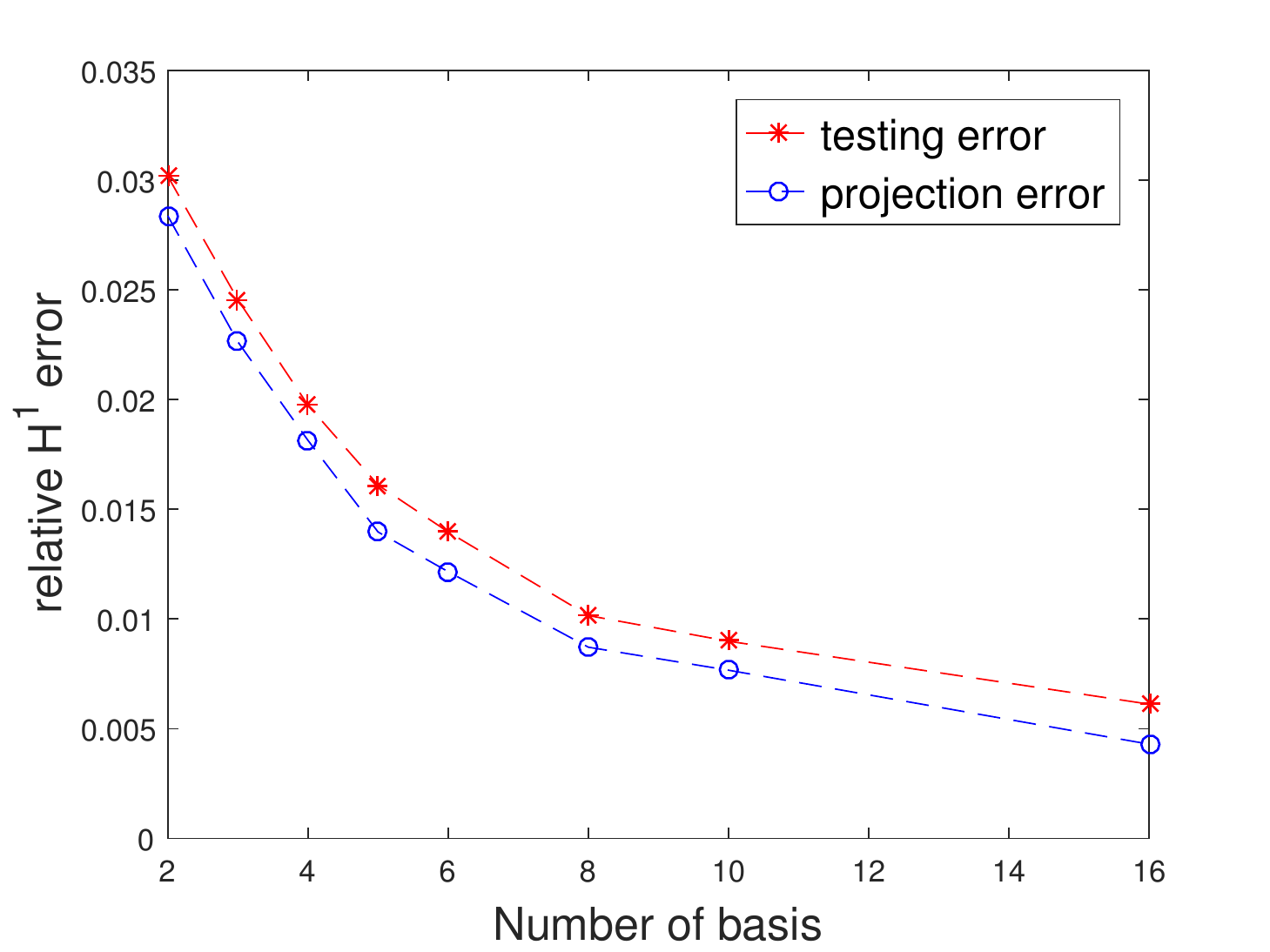}
	\end{subfigure}
	\caption{ Relative $L^2$ and $H^1$ error with increasing number of basis for the local problem of Sec.\ref{sec:Example1}.}
	\label{fig:Example1locall2err}
\end{figure} 

In Figure \ref{fig:Example1localdiffN}, we show the accuracy of the proposed method when we use different number of samples $N$ in constructing the data-driven basis.   Although the numerical error decreases when the sampling number $N$ is increased in general, the difference is very mild. 
\begin{figure}[tbph]
	\begin{subfigure}[b]{0.49\textwidth}
	\centering
	\includegraphics[width=0.8\linewidth]{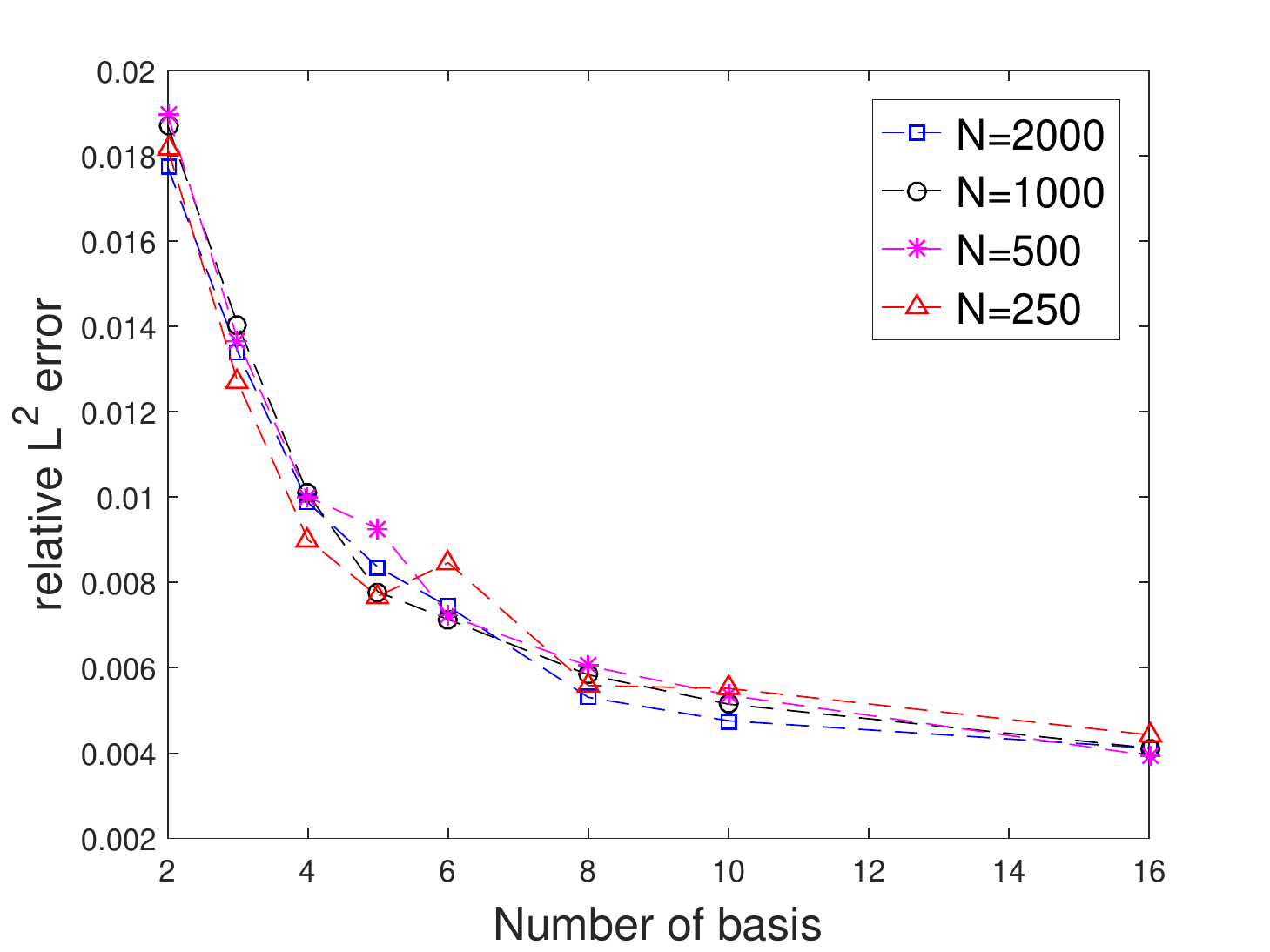}
	\caption{Testing errors in $L^2$ norm.}
	\end{subfigure}
	\begin{subfigure}[b]{0.49\textwidth}
	\centering
	\includegraphics[width=0.8\linewidth]{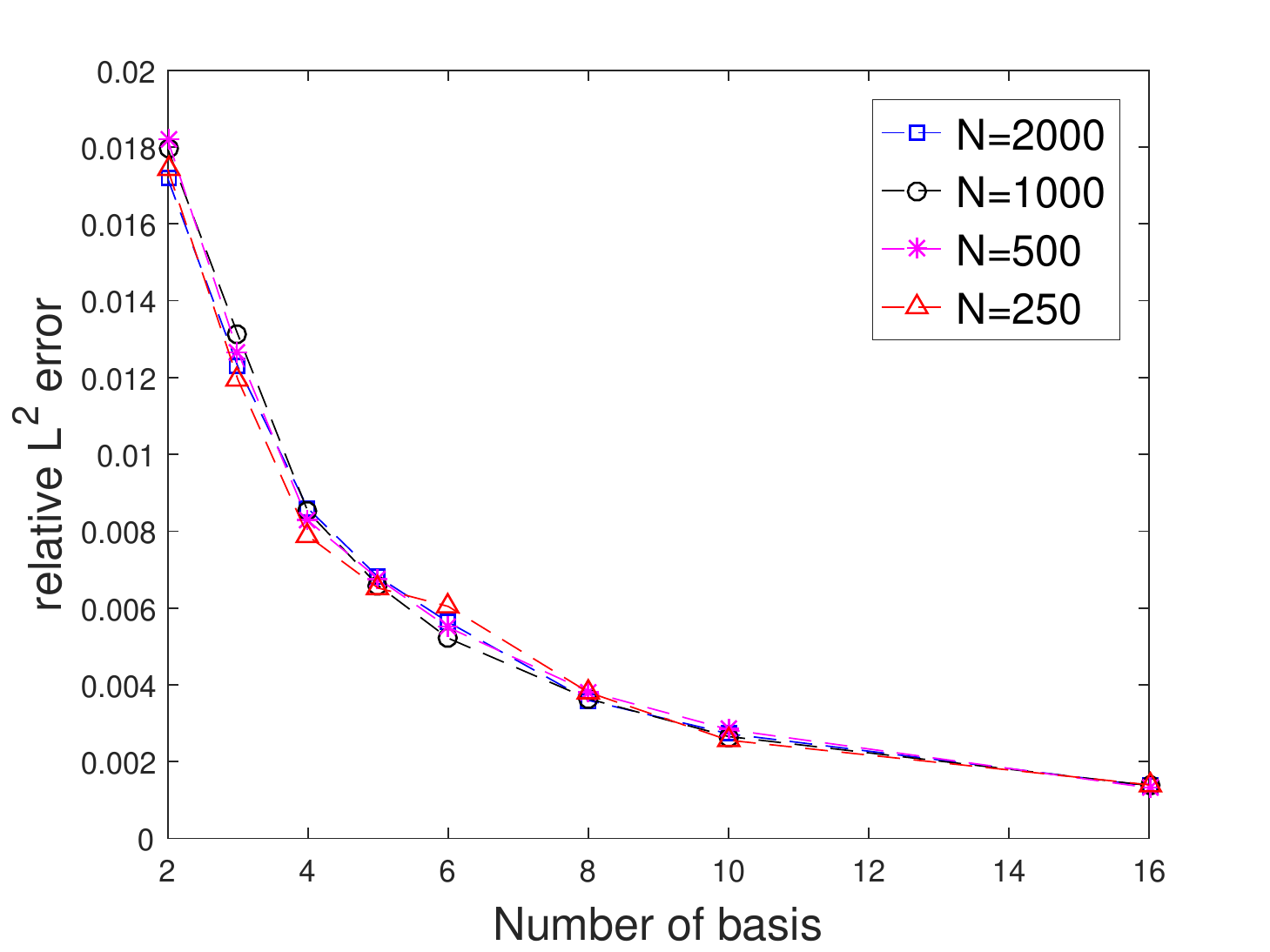}
	\caption{Projection errors in $L^2$ norm.}
	\end{subfigure}
	\begin{subfigure}[b]{0.49\textwidth}
	\centering
	\includegraphics[width=0.8\linewidth]{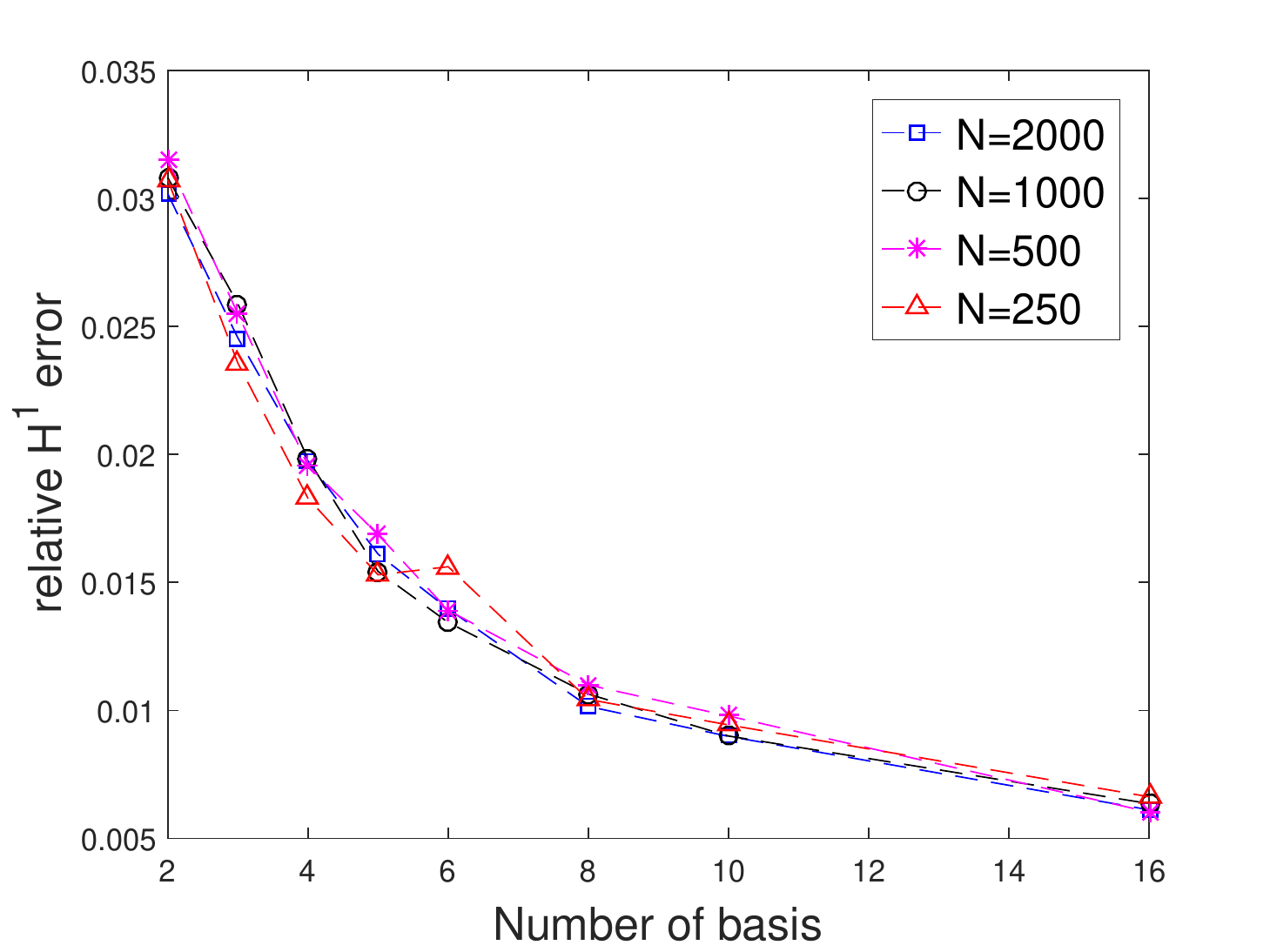}
	\caption{Testing errors in $H^1$ norm.}
	\end{subfigure}
	\begin{subfigure}[b]{0.49\textwidth}
	\centering
	\includegraphics[width=0.8\linewidth]{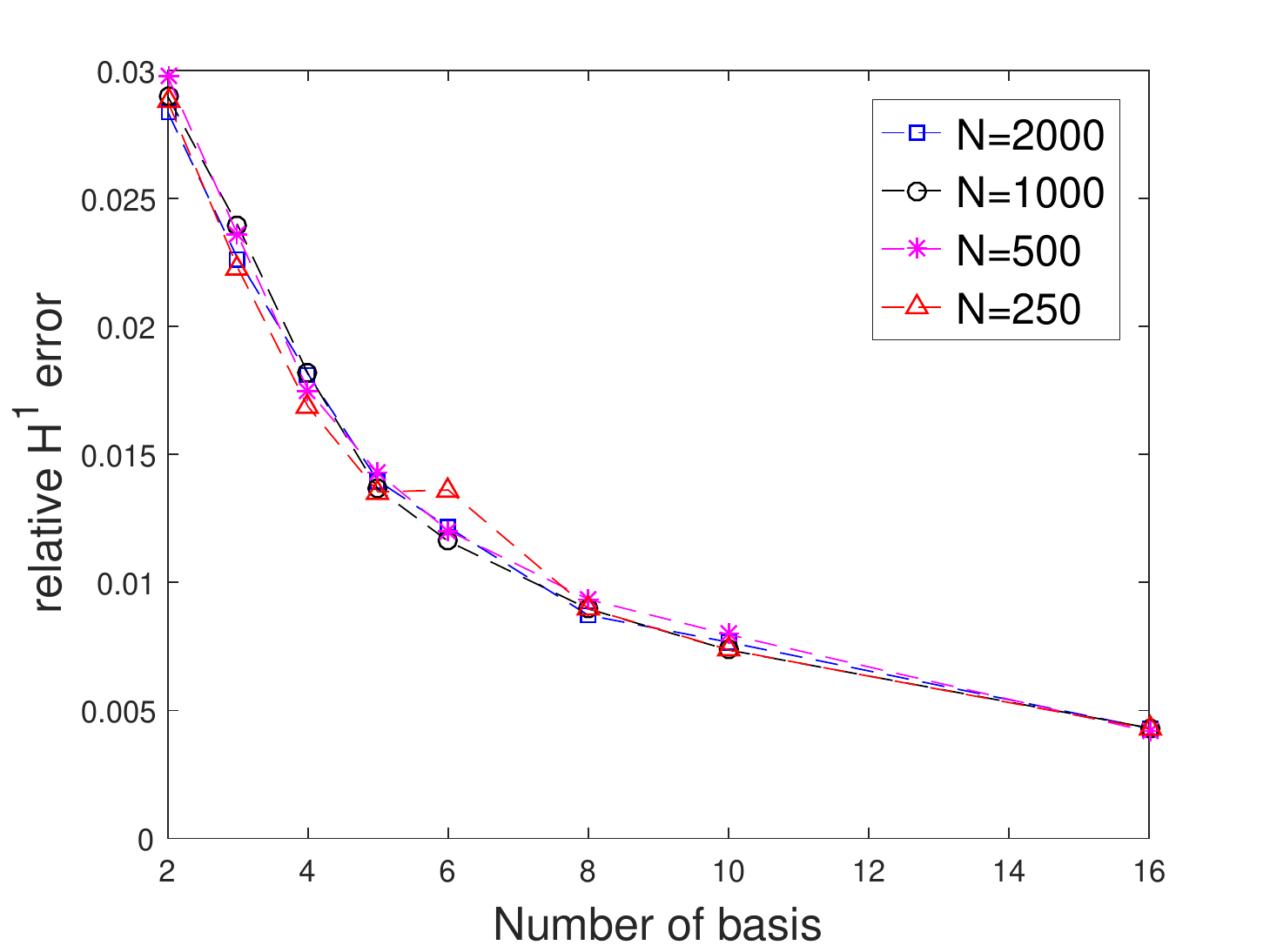}
	\caption{Projection errors in  $H^1$ norm.}
	\end{subfigure}
	\caption{ The relative testing/projection errors in $L^2$ and $H^1$ norms with different number of samples (i.e. $N$) for the local problem of Sec.\ref{sec:Example1}.}
	\label{fig:Example1localdiffN}
\end{figure} 

Next, we test our method on the whole computation domain  for \eqref{randommultiscaleelliptic} with coefficient \eqref{coefficientofexample1}. Figure \ref{fig:Example1globaleigenvalues} shows the decay property of eigenvalues. Similarly, we show magnitudes of the leading eigenvalues in Figure \ref{fig:Example1globaleigenvalues3a} and the ratio of the accumulated sum of the eigenvalues over the total sum in Figure \ref{fig:Example1globaleigenvalues3b}. We observe similar behaviors as before. Since we approximate the solution in the whole computational domain, we take the Galerkin approach described in Section \ref{sec:GlobalProblem} using the data-driven basis. 
In Figure \ref{fig:Example1globall2engerr}, we show the mean relative error between our numerical solution and the reference solution in $L^2$ norm and $H^1$ norm, respectively. In practice, when the number of basis is 15, it takes about $0.084s$ to compute a new solution by our method, whereas the standard FEM method costs about $0.82s$ for one solution. 

\begin{figure}[tbph] 
	\centering
	\begin{subfigure}[b]{0.45\textwidth}
		\includegraphics[width=1.0\linewidth]{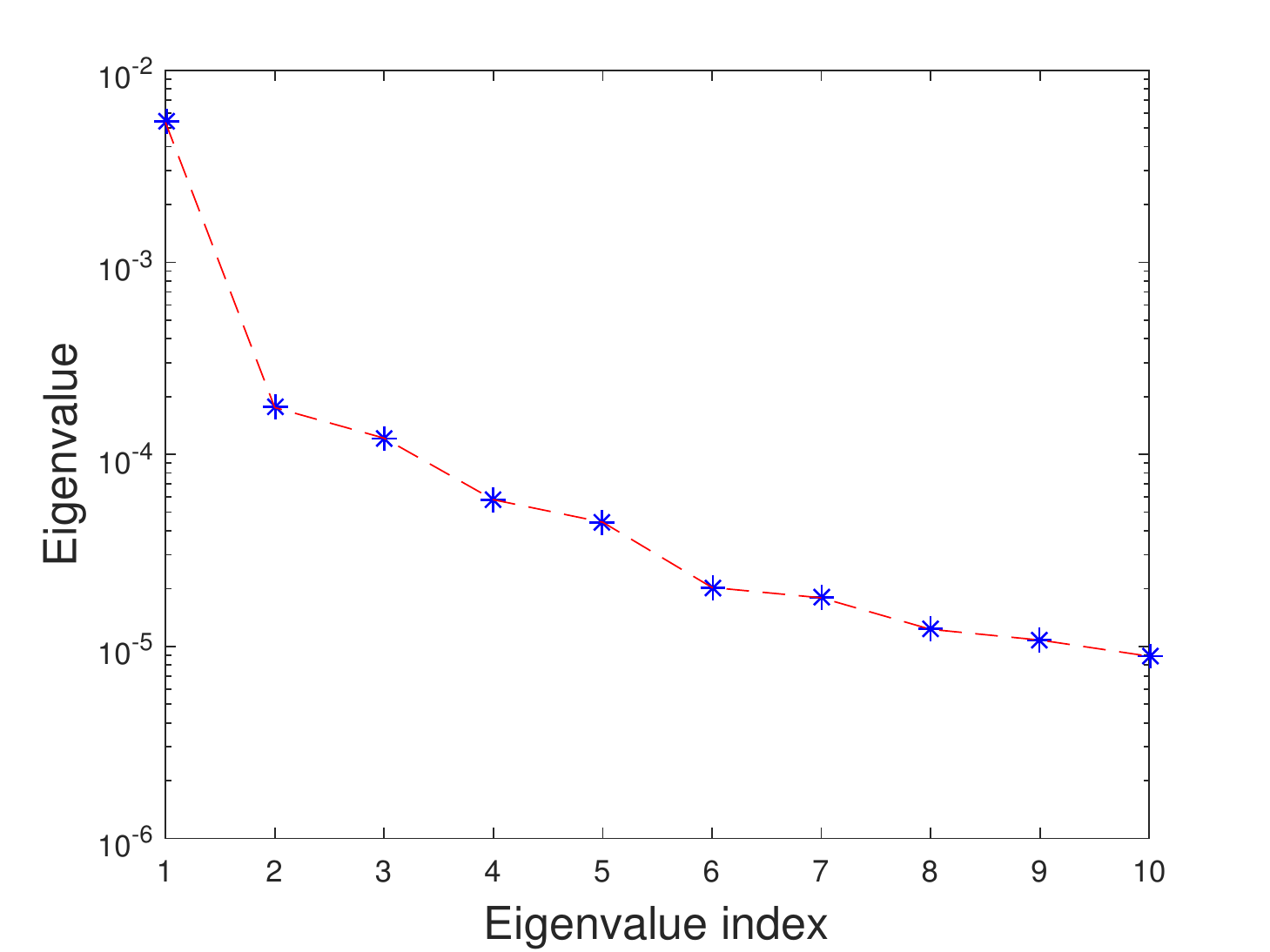} 
		\caption{ Decay of the eigenvalues.}
		\label{fig:Example1globaleigenvalues3a}
	\end{subfigure}
	\begin{subfigure}[b]{0.45\textwidth}
		\includegraphics[width=1.0\linewidth]{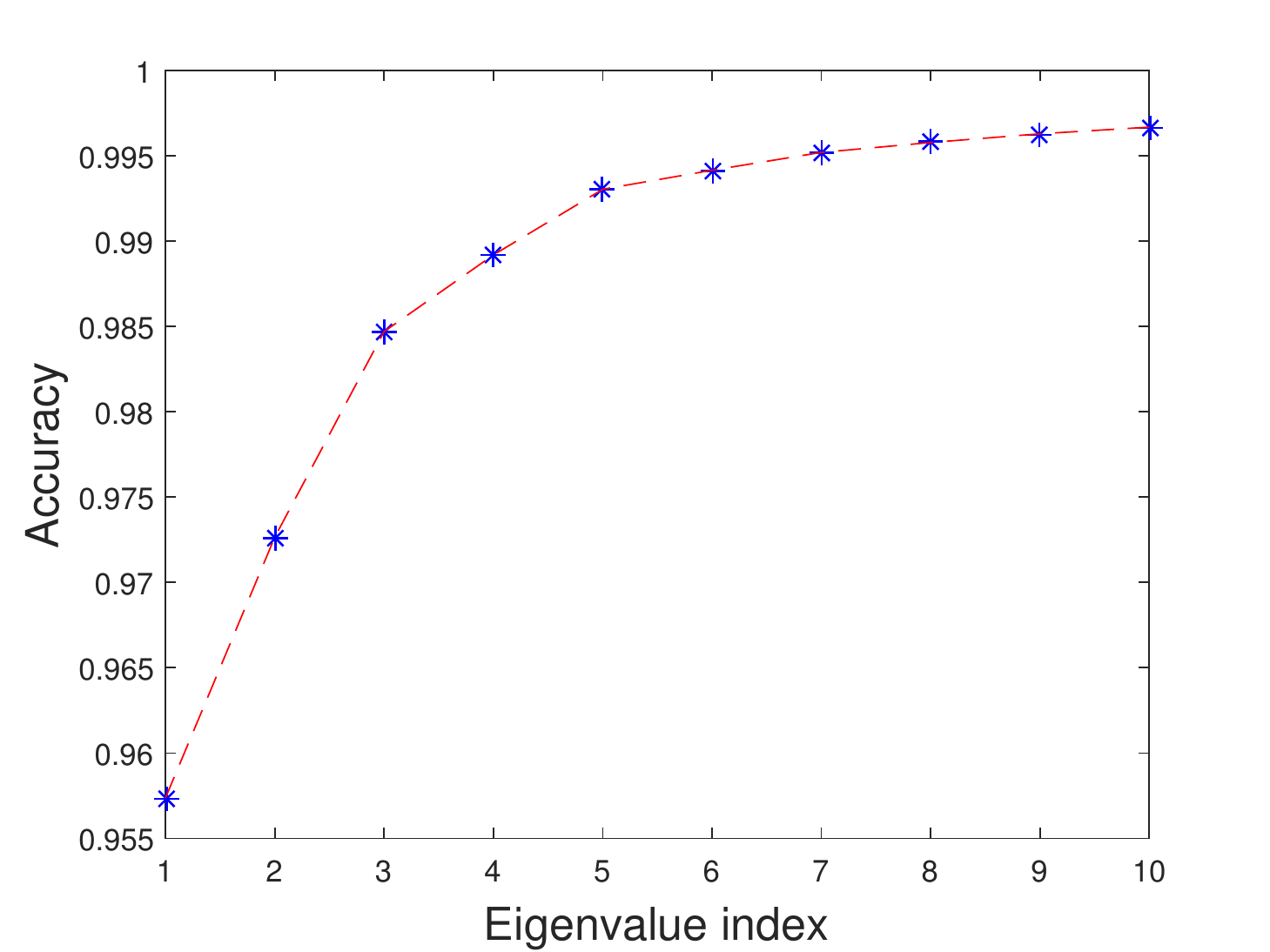} 
		\caption{ $1-\sqrt{\sum_{j=n+1}^{N}\lambda_{j}/\sum_{j=1}^{N}\lambda_{j}}$, $n=1,2,...$.} 
		\label{fig:Example1globaleigenvalues3b}
	\end{subfigure}
	\caption{The decay properties of the eigenvalues for the global problem of Sec.\ref{sec:Example1}.}
	\label{fig:Example1globaleigenvalues}
\end{figure}

\begin{figure}[tbph] 
	\centering	
	\begin{subfigure}[b]{0.45\textwidth}
		\includegraphics[width=1.0\linewidth]{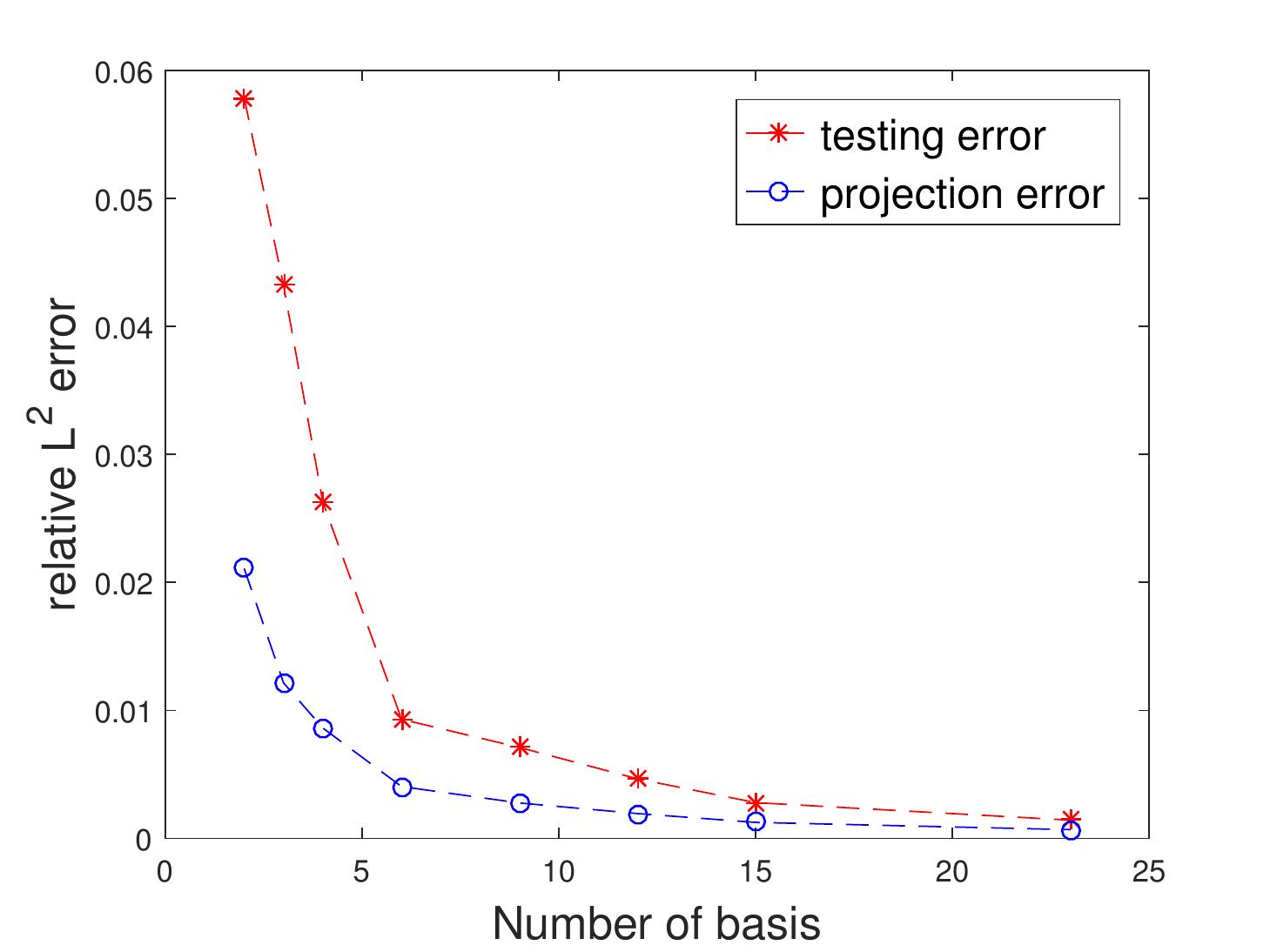} 
		\caption{ Relative error in $L^2$ norm.}
		\label{fig:Example1globall2err}
	\end{subfigure}
	\begin{subfigure}[b]{0.45\textwidth}
		\includegraphics[width=1.0\linewidth]{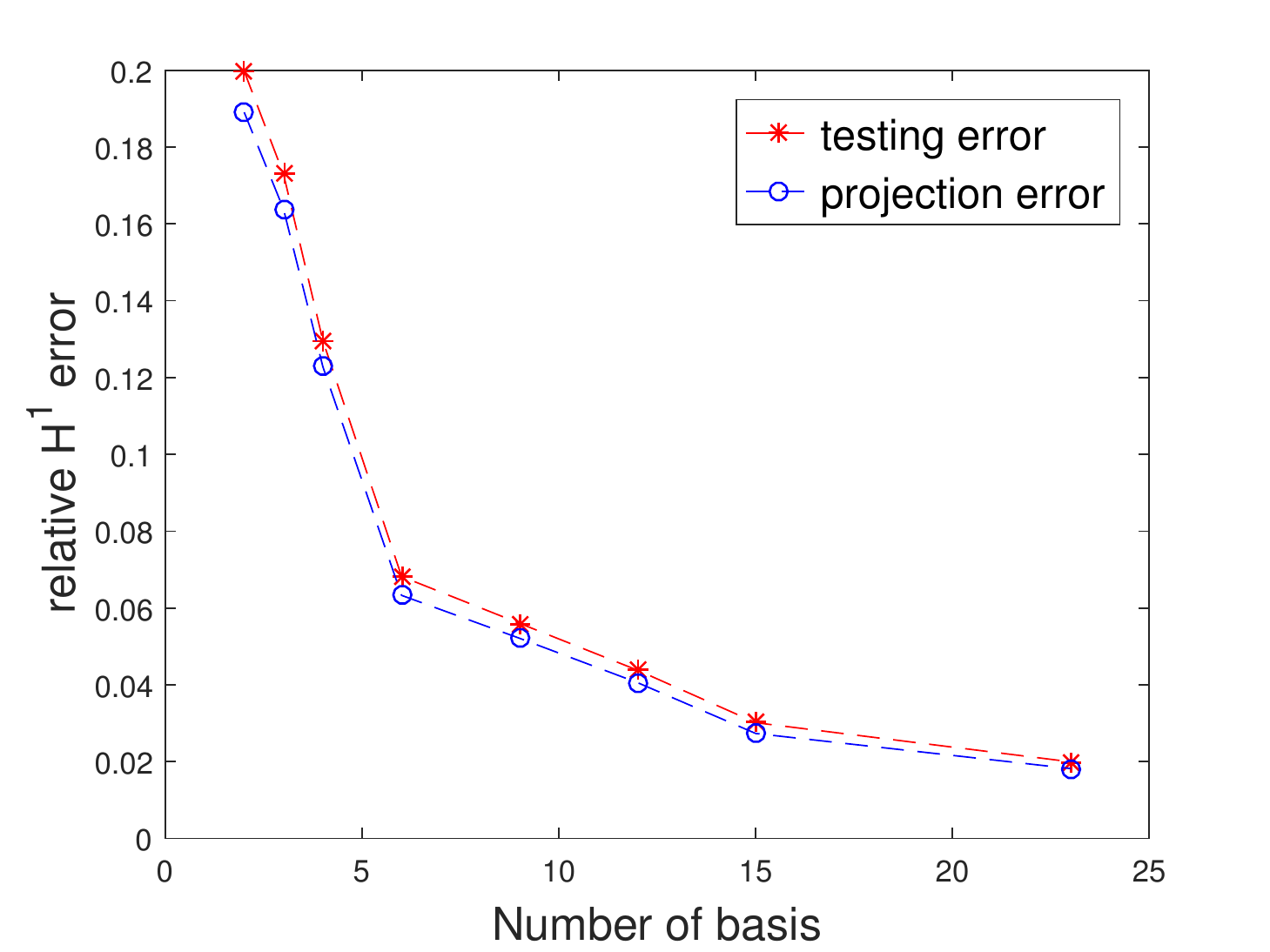} 
		\caption{ Relative error in $H^1$ norm.} 
		\label{fig:Example1globalengerr}
	\end{subfigure}
	\caption{The relative errors with increasing number of basis for the global problem of Sec.\ref{sec:Example1}.}
	\label{fig:Example1globall2engerr}
\end{figure}

\subsection{An example with an exponential type coefficient}\label{sec:Example2}
\noindent
We now solve the problem \eqref{randommultiscaleelliptic} with an exponential type coefficient. 
The coefficient is parameterized by eight random variables, which has the following form 
\begin{align}
a(x,y,\omega) =&\exp\Big( \sum_{i=1}^8 \sin(\frac{2\pi (9-i)x}{9\epsilon_i})\cos(\frac{2\pi iy}{9\epsilon_i})\xi_i(\omega) \Big),
\label{coefficientofexample2}
\end{align}
where the multiscale parameters $[\epsilon_1,\epsilon_2,\cdots,\epsilon_{8}] =[\frac{1}{43},\frac{1}{41},\frac{1}{47},\frac{1}{29},\frac{1}{37},\frac{1}{31},\frac{1}{53},\frac{1}{35}]$ and $\xi_i(\omega)$, $i=1,...,8$ are i.i.d. uniform random variables in $[-\frac{1}{2},\frac{1}{2}]$. Hence the contrast ratio is $\kappa_a\approx 3.0\times 10^3$ in the coefficient \eqref{coefficientofexample2}. The force function is $f(x,y) = \cos(2\pi x)\sin(2\pi y)\cdot I_{D_2}(x,y)$, where $I_{D_2}$ is an indicator function defined on $D_2=[\frac{1}{4},\frac{3}{4}]\times[\frac{1}{16},\frac{5}{16}]$. 
In the local problem, the subdomain of interest is $D_1=[\frac{1}{4},\frac{3}{4}]\times[\frac{11}{16},\frac{15}{16}]$. 

In Figure \ref{fig:Example2eigenvalues}, we show the decay property of eigenvalues. Specifically, in Figure \ref{fig:Example2eigenvalues-a} we show the magnitude of leading eigenvalues and in Figure \ref{fig:Example2eigenvalues-b} we show the ratio of the accumulated sum of the eigenvalues over the total sum. These results imply that the solution space has a low-dimensional structure, which can be approximated by the data-driven basis functions. 

\begin{figure}[tbph] 
	\centering
	\begin{subfigure}[b]{0.45\textwidth}
		\includegraphics[width=1.0\linewidth]{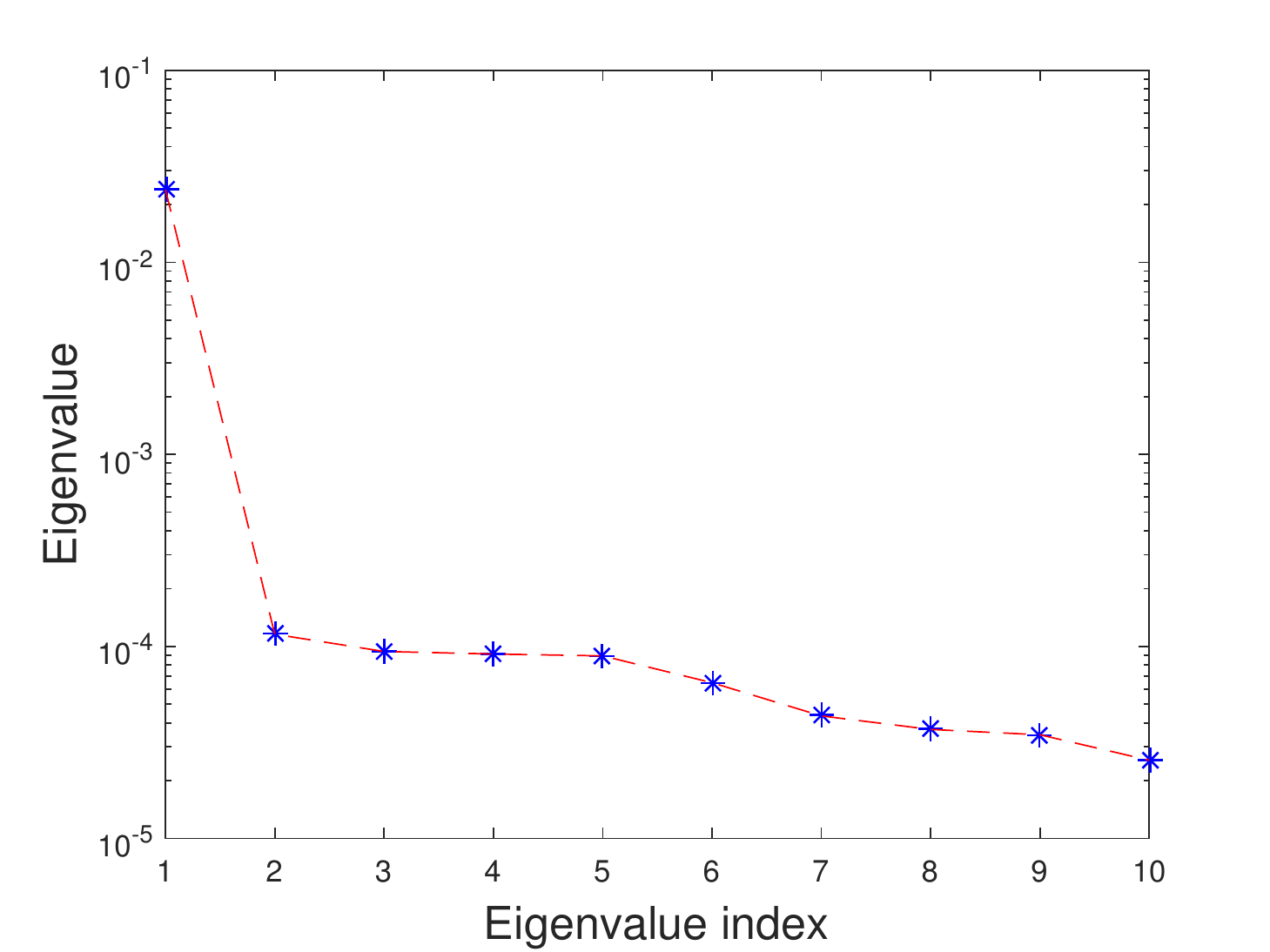} 
		\caption{ Decay of eigenvalues.}
		\label{fig:Example2eigenvalues-a}
	\end{subfigure}
	\begin{subfigure}[b]{0.45\textwidth}
		\includegraphics[width=1.0\linewidth]{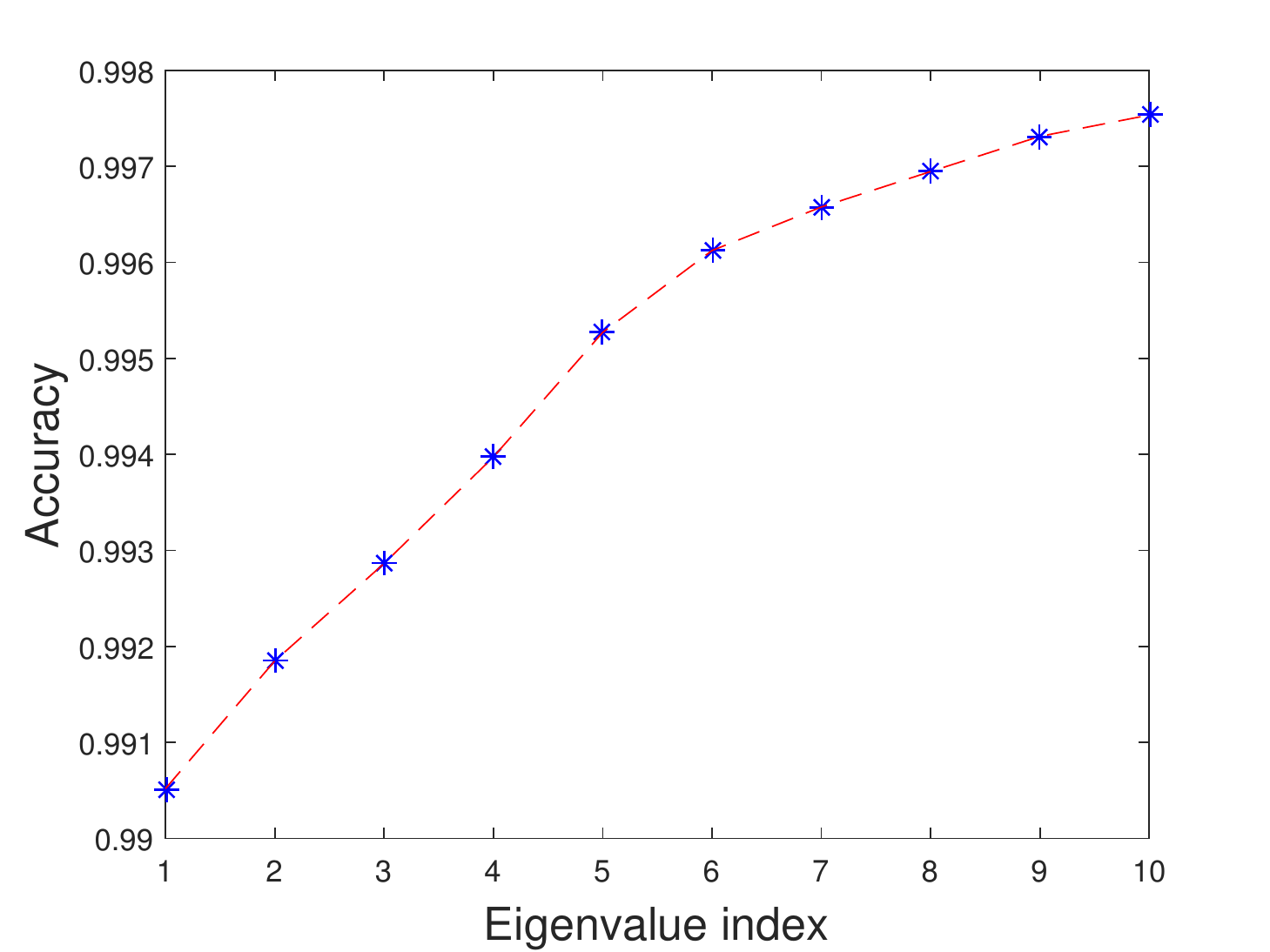} 
		\caption{ $1-\sqrt{\sum_{j=n+1}^{N}\lambda_{j}/\sum_{j=1}^{N}\lambda_{j}}$, $n=1,2,...$.} 
		\label{fig:Example2eigenvalues-b}
	\end{subfigure}
	\caption{The decay properties of the eigenvalues in the problem of Sec.\ref{sec:Example2}.}
	\label{fig:Example2eigenvalues}
\end{figure} 
Since the coefficient $a(x,y,\omega)$ is parameterized by eight random variables, it is expensive to construct the mapping $\textbf{F}:\vxi(\omega)\mapsto \textbf{c}(\omega)$ using the interpolation method with uniform grids. Instead, we use a sparse grid polynomial interpolation approach to approximate the mapping $\textbf{F}$. Specifically, we use Legendre polynomials with total order less than or equal 4 to approximate the mapping, where the total number of nodes is $N_1=2177$; see \cite{Griebel:04}. 

Figure \ref{fig:Example2errors-a} shows the relative errors of the testing error and projection error in $L^2$ norm. Figure \ref{fig:Example2errors-b} shows the corresponding relative errors in $H^1$ norm. The sparse grid polynomial interpolation approach gives a comparable error as the best approximation error. We observe similar convergence results in solving the global problem \eqref{randommultiscaleelliptic} with the coefficient \eqref{coefficientofexample2} (not shown here). Therefore, we can use sparse grid method to construct mappings for problems of moderate number of random variables. 

\begin{figure}[htbp]
	\centering
	\begin{subfigure}[b]{0.45\textwidth}
		\includegraphics[width=1.0\linewidth]{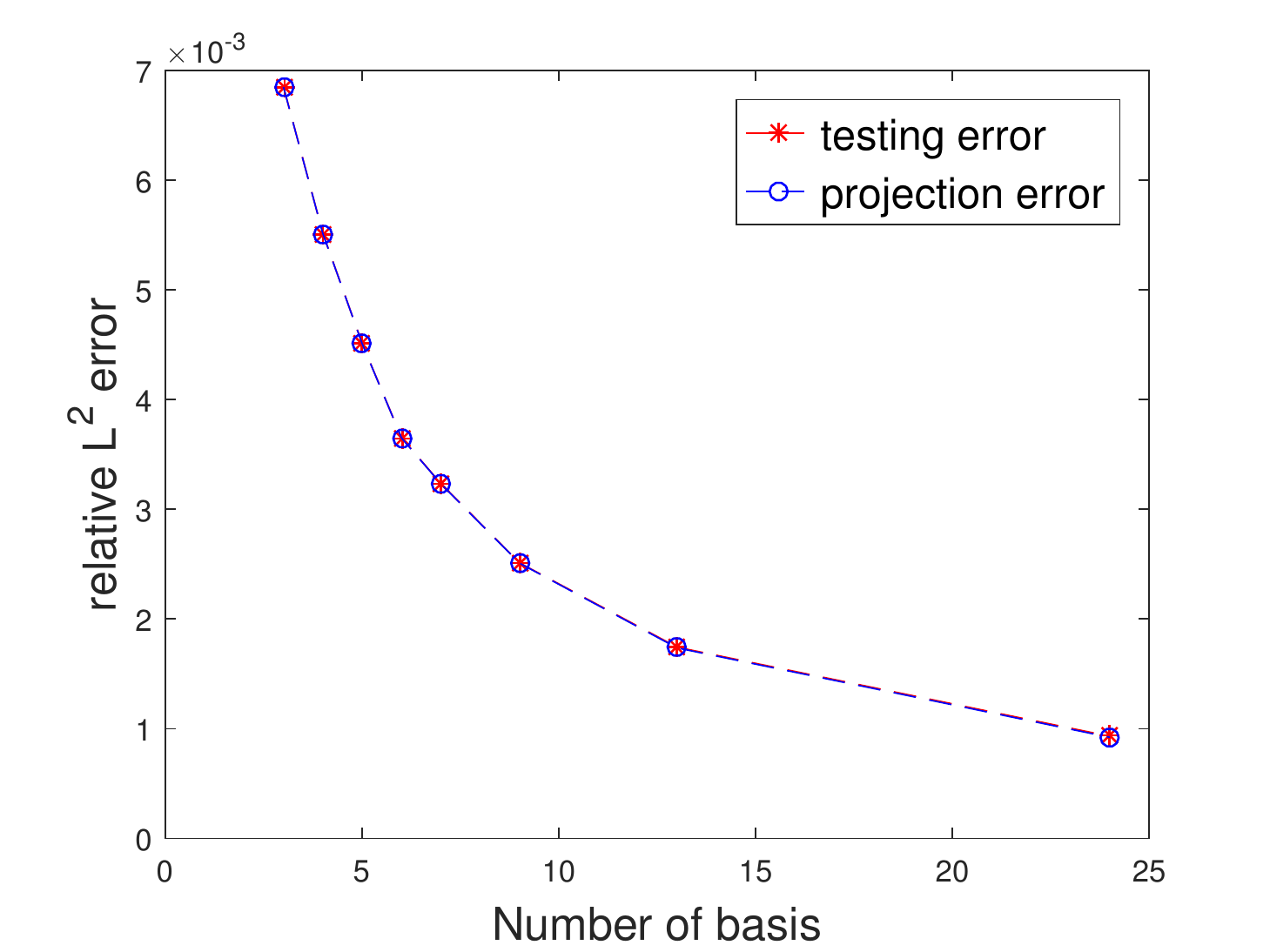} 
		\caption{ Relative error in $L^2$ norm.}
		\label{fig:Example2errors-a}
	\end{subfigure}
	\begin{subfigure}[b]{0.45\textwidth}
		\includegraphics[width=1.0\linewidth]{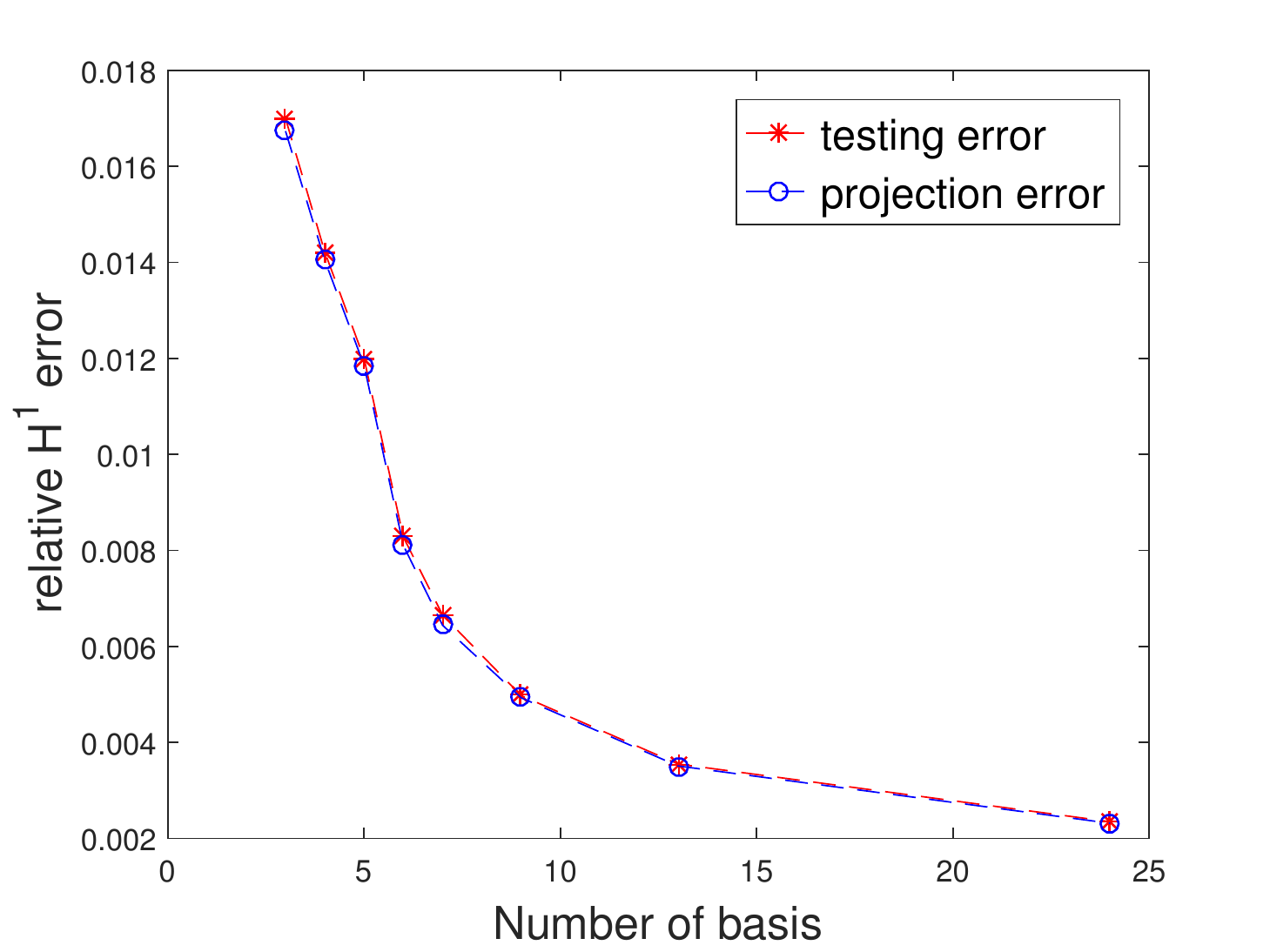} 
		\caption{ Relative error in $H^1$ norm.} 
		\label{fig:Example2errors-b}
	\end{subfigure}
	\caption{ The relative errors with increasing number of basis in the problem of Sec.\ref{sec:Example2}.}
	\label{fig:Example2errors}
\end{figure} 

\subsection{An example with a discontinuous coefficient}\label{sec:ExampleInterface}
\noindent 
We solve the problem \eqref{randommultiscaleelliptic} with a discontinuous coefficient, 
which is an interface problem.  The coefficient is parameterized by twelve random variables and has the following form 
	\begin{align}
a(x,y,\omega) =& \exp\Big(\sum_{i=1}^{6} \sin(2\pi \frac{x\sin(\frac{i\pi}{6}) +y\cos(\frac{i\pi}{6})   }{\epsilon_i}   )\xi_i(\omega) \Big)\cdot I_{D\setminus D_3}(x,y)\nonumber\\
&+\exp\Big(\sum_{i=1}^{6} \sin(2\pi \frac{x\sin(\frac{(i+0.5)\pi}{6}) +y\cos(\frac{(i+0.5)\pi}{6})   }{\epsilon_{i+6}}   )\xi_{i+6}(\omega) \Big)\cdot I_{D_3}(x,y),
\label{coefficientofexampleInterface}
\end{align}
where $\epsilon_i=\frac{1+i}{100}$ for $i=1,\cdots,6$, $\epsilon_{i}=\frac{i+13}{100}$ for $i=7,\cdots,12$,  $\xi_i(\omega)$, $i=1,\cdots,12$ are i.i.d. uniform random variables in $[-\frac{2}{3},\frac{2}{3}]$, and $I_{D_3}$ and $I_{D\setminus D_3}$ are indicator functions. 
The subdomain $D_3$ consists of three small rectangles whose edges are parallel to the edges of domain $D$ with width $10h$ and height $0.8$. And the lower left vertices are located 
at $(0.3,0.1),(0.5,0.1),(0.7,0.1)$ respectively. The contrast ratio in the coefficient \eqref{coefficientofexampleInterface} is $\kappa_a\approx 3\times 10^3$. In Figure \ref{fig:ExampleInterfaceRealizations} we show two realizations of the coefficient \eqref{coefficientofexampleInterface}. 	

	\begin{figure}[htbp]
		\centering
		\includegraphics[width=0.49\linewidth]{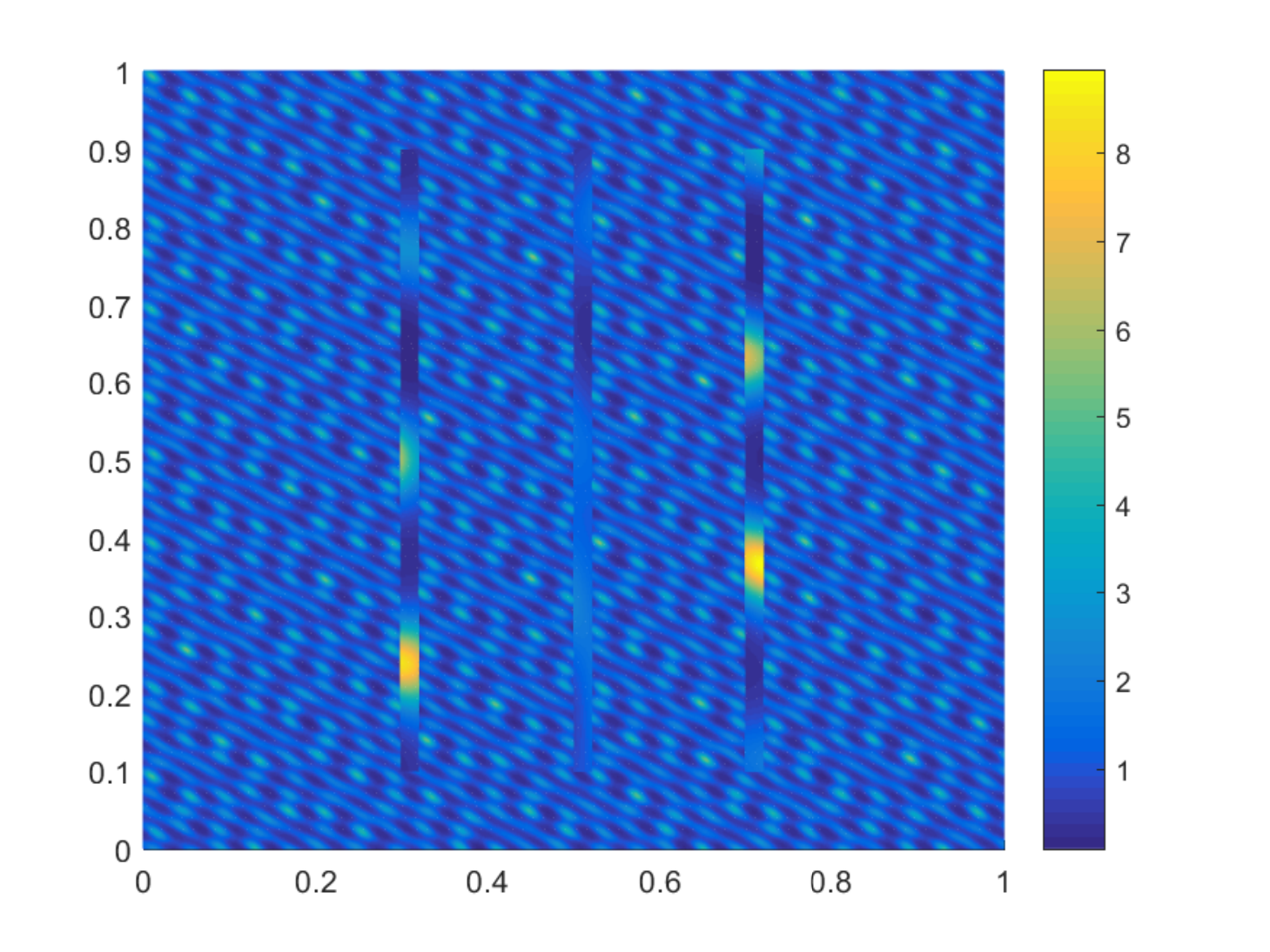}
		\includegraphics[width=0.49\linewidth]{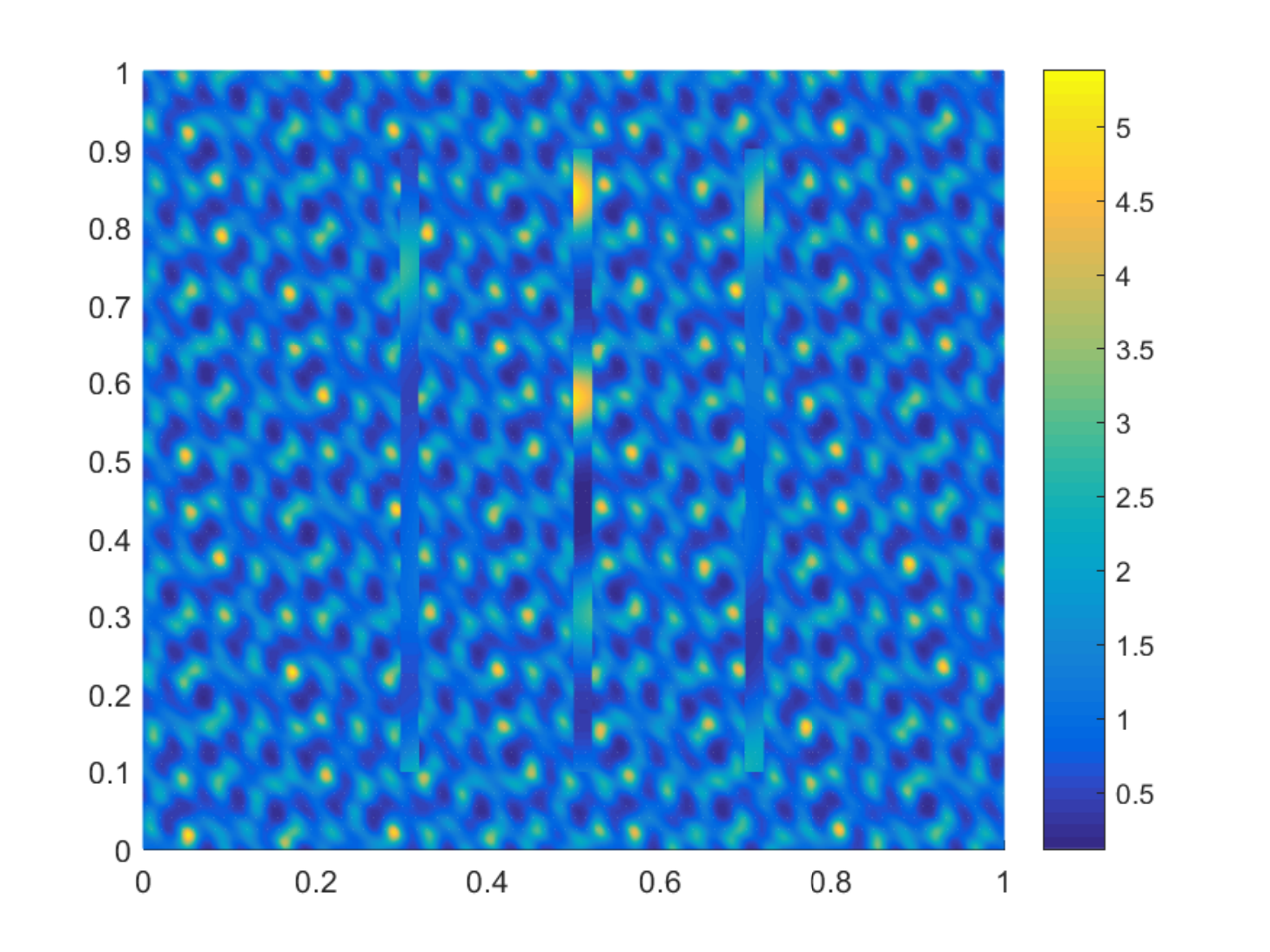} 
		\caption{Two realizations of the coefficient \eqref{coefficientofexampleInterface} in the interface problem.} 
		\label{fig:ExampleInterfaceRealizations}
	\end{figure}
 
	We now solve the local problem of \eqref{randommultiscaleelliptic} with the coefficient \eqref{coefficientofexampleInterface}, where the domain of interest is $D_1=[\frac{1}{4},\frac{3}{4}]\times[\frac{11}{16},\frac{15}{16}]$. The force function is $f(x,y) = \cos(2\pi x)\sin(2\pi y)\cdot I_{D_2}(x,y)$, where $D_2=[\frac{1}{4},\frac{3}{4}]\times[\frac{1}{16},\frac{5}{16}]$.  In Figure \ref{fig:ExampleInterfaceeigenvalues-a} and Figure \ref{fig:ExampleInterfaceeigenvalues-b} we show the magnitude of dominant eigenvalues and approximate accuracy. These results show that only a few data-driven basis functions are enough to approximate all solution samples well. 
	\begin{figure}[tbph] 
		\centering
		\begin{subfigure}[b]{0.45\textwidth}
			\includegraphics[width=1.0\linewidth]{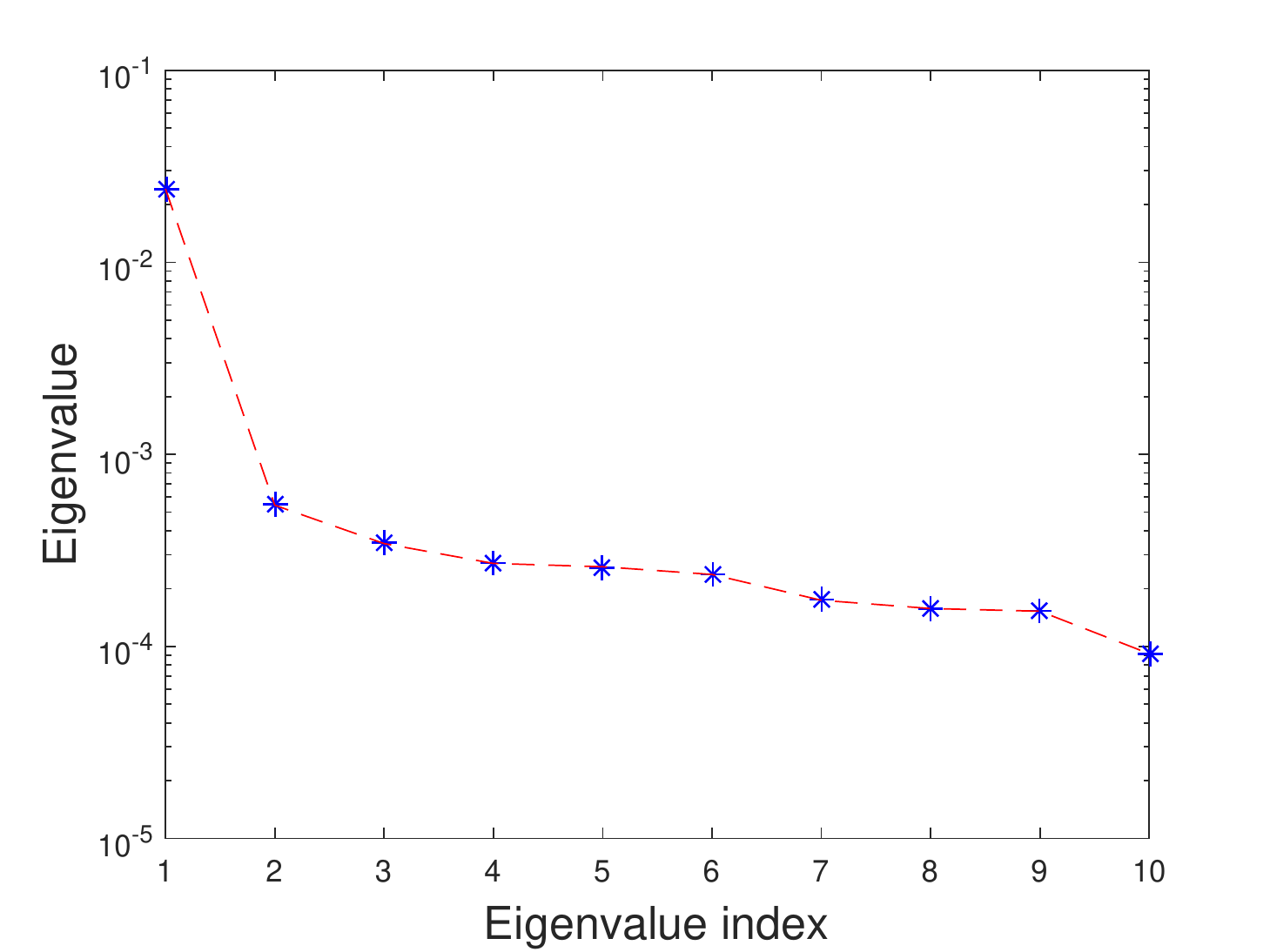} 
			\caption{ Decay of eigenvalues.}
			\label{fig:ExampleInterfaceeigenvalues-a}
		\end{subfigure}
		\begin{subfigure}[b]{0.45\textwidth}
			\includegraphics[width=1.0\linewidth]{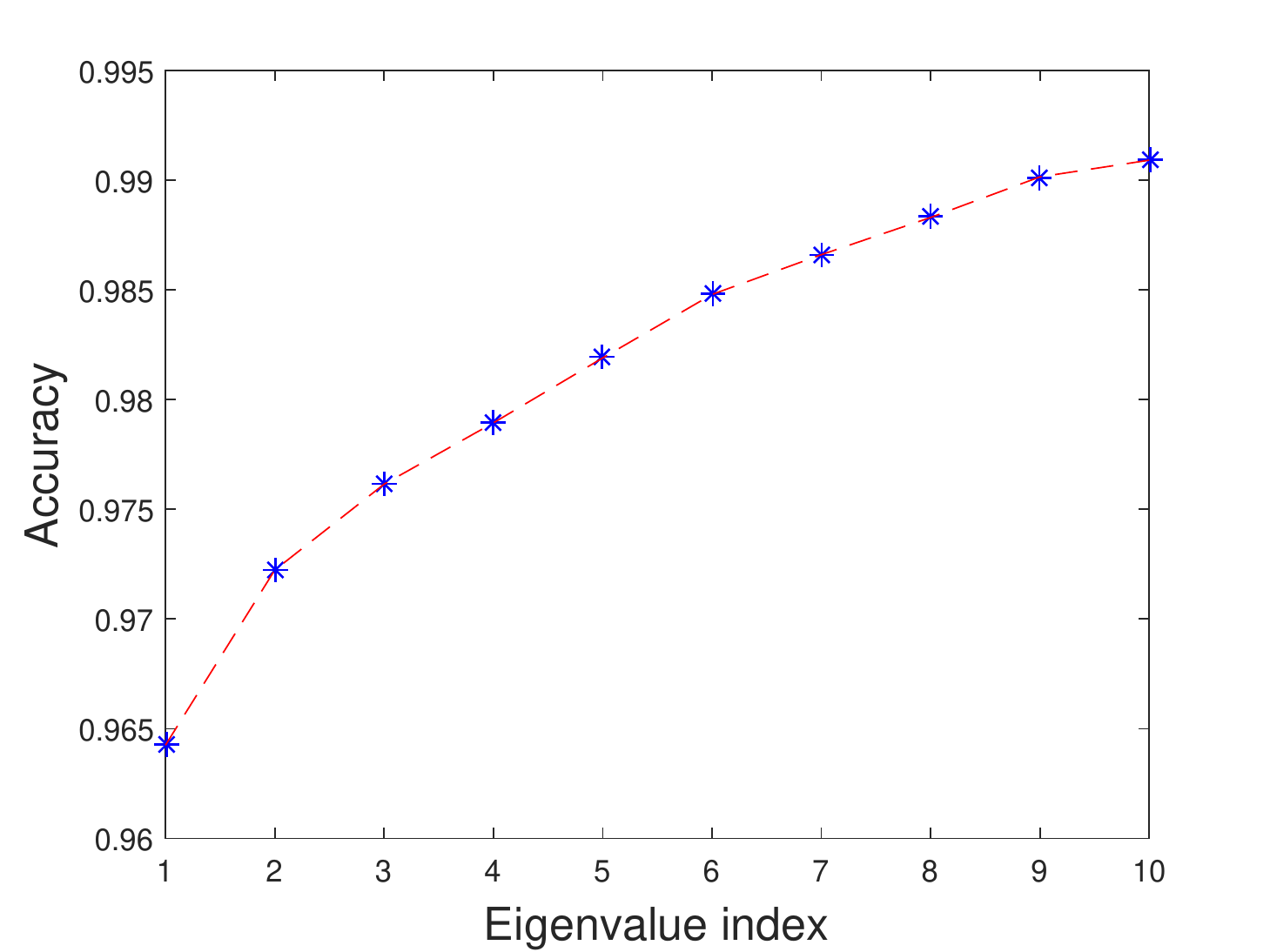} 
			\caption{$1-\sqrt{\sum_{j=n+1}^{N}\lambda_{j}/\sum_{j=1}^{N}\lambda_{j}}$, $n=1,2,...$.} 
			\label{fig:ExampleInterfaceeigenvalues-b}
		\end{subfigure}
		\caption{The decay properties of the eigenvalues in the problem of Sec.\ref{sec:ExampleInterface}.}
		\label{fig:ExampleInterfaceeigenvalues}
	\end{figure}  

    Since the coefficient \eqref{coefficientofexampleInterface} is parameterized by twelve random variables, constructing the mapping $\textbf{F}:\vxi(\omega)\mapsto \textbf{c}(\omega)$ using the sparse grid polynomial interpolation becomes very expensive too. Here we use the least square method combined with the $k-d$ tree algorithm for searching nearest neighbors to approximate the mapping $\textbf{F}$.

    In our method, we first generate $N_1=5000$ data pairs $\{(\vxi^n(\omega),\textbf{c}^n(\omega)\}_{n=1}^{N_1}$ that will be used as training data. 
    Then, we use $N_2=200$ samples for testing in the online stage. For each new testing data point $\vxi(\omega)=[\xi_1(\omega),\cdots,\xi_r(\omega)]^T$ (here $r=12$), we run the $k-d$ tree algorithm to find its $n$ nearest neighbors in the training data set and  apply the least square method to compute the corresponding mapped value $\vec{c}(\omega)=[c_1(\omega), \ldots, c_K(\omega)]^T$. The complexity of constructing a $k-d$ tree is $O(N_1\log N_1)$. Given the $k-d$ tree, for each testing point the complexity of finding its $n$ nearest neighbors is $O(n\log N_1)$ \cite{wald2006building}. Since the $n$ training data points are close to the testing data point $\vxi(\omega)$, for each training data $(\vxi^m(\omega),\textbf{c}^m(\omega)$, $m=1,....n$, we compute the first-order Taylor expansion of each component $c^m_j(\omega)$ at $\vxi(\omega)$ as   
    \begin{align}
    c^m_j(\omega)\approx c_j(\omega)+\sum_{i=1}^{r=12}(\xi^m_i-\xi_i)\frac{\partial c_j}{\partial \xi_i}(\omega),\quad j=1,2,\cdots,K, 
    \label{least-square-system}
    \end{align}
    where $\xi^m_i$, $i=1,...,r$, $c^m_j(\omega)$, $j=1,...,K$ are given training data, $c_j(\omega)$ and 
    $\frac{\partial c_j}{\partial \xi_i}(\omega)$, $j=1,...,K$ are unknowns associated with the testing data point $\vxi(\omega)$. In the $k-d$ tree algorithm, we choose $n=20$, which is slightly greater than $r+1=13$. By solving \eqref{least-square-system} using the least square method, we get the mapped value $\vec{c}(\omega)=[c_1(\omega), \ldots, c_K(\omega)]^T$. Finally, we use the formula \eqref{RB_expansion} to get the numerical solution of Eq.\eqref{randommultiscaleelliptic} with the coefficient \eqref{coefficientofexampleInterface}. 
    
    Because of the discontinuity and high-dimensional random variables in the coefficient \eqref{coefficientofexampleInterface}, the problem \eqref{randommultiscaleelliptic} is more challenging. The nearest neighbors based least square method provides an efficient way to construct mappings and achieves relative errors less than $3\%$ in both $L^2$ norm and $H^1$ norm;
    see Figure \ref{fig:ExampleInterfacelocalerrors}. Alternatively, one can use the neural network method to construct mappings for this type of challenging problems; see Section \ref{sec:Example3}.


\begin{figure}[htbp]
	\centering
	\begin{subfigure}[b]{0.45\textwidth}
		\includegraphics[width=1.0\linewidth]{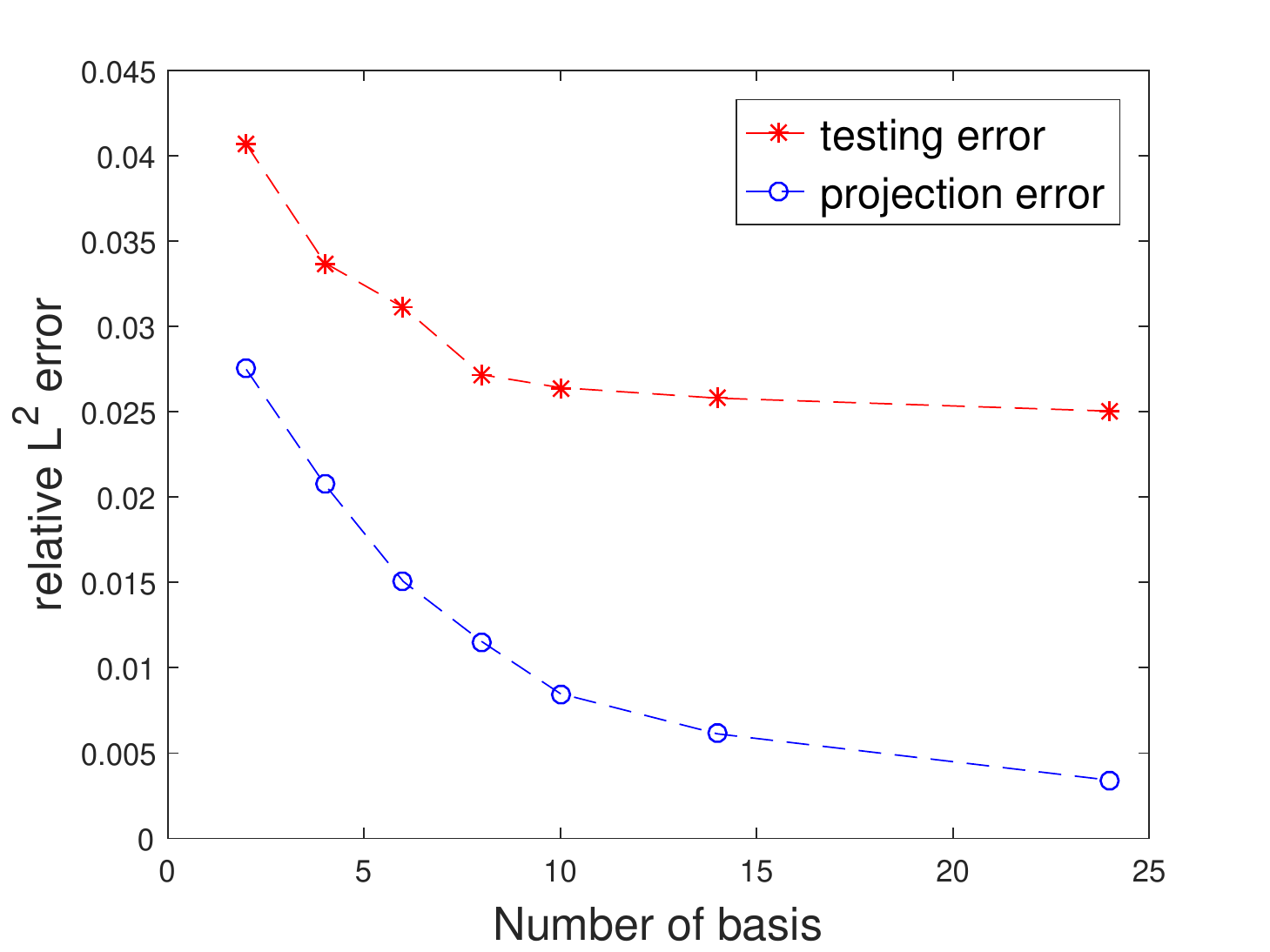}
		\label{fig:ExampleInterfaceerrors-a}
	\end{subfigure}
	\begin{subfigure}[b]{0.45\textwidth}
		\includegraphics[width=1.0\linewidth]{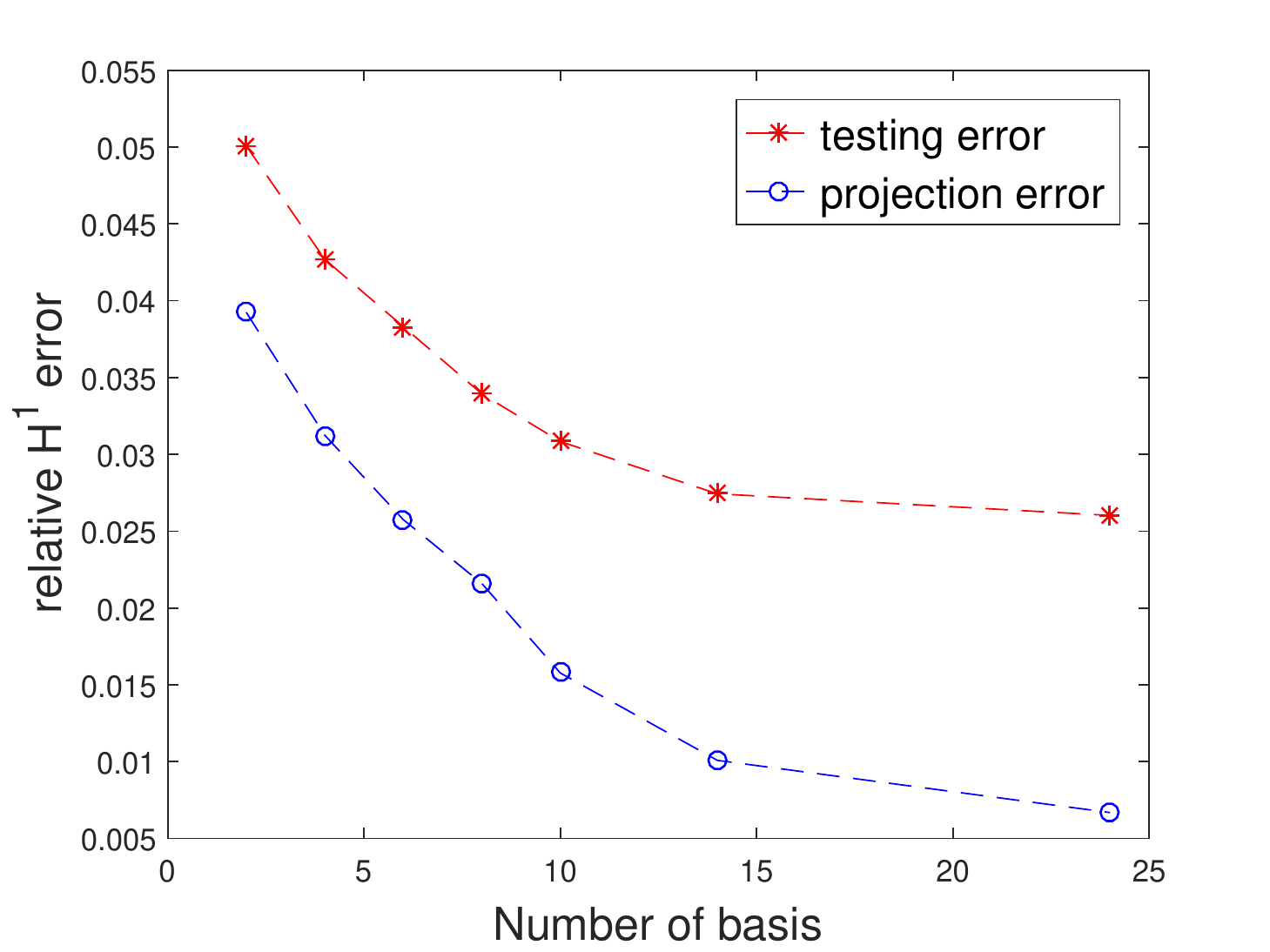}
		\label{fig:ExampleInterfaceerrors-b}
	\end{subfigure}
	\caption{ The relative errors with increasing number of basis in the local problem of Sec.\ref{sec:ExampleInterface} .}
	\label{fig:ExampleInterfacelocalerrors}
\end{figure}

\subsection{An example with high-dimensional random coefficient and force function}\label{sec:Example3}
\noindent
We solve the problem \eqref{randommultiscaleelliptic} with an exponential type coefficient and random force function,
where the total number of random variables is twenty. Specifically, the coefficient is parameterized by eighteen i.i.d. random variables, i.e.
\begin{align}
a(x,y,\omega) = \exp\Big(\sum_{i=1}^{18} \sin(2\pi \frac{x\sin(\frac{i\pi}{18}) +y\cos(\frac{i\pi}{18})   }{\epsilon_i}   )\xi_i(\omega) \Big),
\label{coefficientofexample3}
\end{align}
where $\epsilon_i=\frac{1}{2i+9}$, $i=1,2,\cdots,18$ and $\xi_i(\omega)$, $i=1,...,18$ are i.i.d. uniform random variables in $[-\frac{1}{5},\frac{1}{5}]$. The force function is a Gaussian density function $f(x,y) = \frac{1}{2\pi\sigma^2}\exp(-\frac{(x-\theta_1)^2+(y-\theta_2)^2}{2\sigma^2})$ with a random center $(\theta_1,\theta_2)$ that is a random point uniformly distributed in the subdomain $D_2=[\frac{1}{4},\frac{3}{4}]\times[\frac{1}{16},\frac{5}{16}]$ and $\sigma=0.01$. When $\sigma$ is small, the 
Gaussian density function $f(x,y)$ can be used to approximate the Dirac-$\delta$ function, such as modeling wells in reservoir simulations.

We first solve the local problem of \eqref{randommultiscaleelliptic} with the coefficient \eqref{coefficientofexample3}, where the subdomain of interest is $D_1=[\frac{1}{4},\frac{3}{4}]\times[\frac{11}{16},\frac{15}{16}]$. In Figures \ref{fig:Example3eigenvalues-a} and \ref{fig:Example3eigenvalues-b}, we show the magnitude of leading eigenvalues and the ratio of the accumulated sum of the eigenvalue over the total sum, respectively. We observe similar exponential decay properties of eigenvalues even if the force function contains randomness. These results show that we can still build a set of data-driven basis functions to solve problem  \eqref{randommultiscaleelliptic} with coefficient \eqref{coefficientofexample3}.

\begin{figure}[tbph] 
	\centering
	\begin{subfigure}[b]{0.45\textwidth}
		\includegraphics[width=1.0\linewidth]{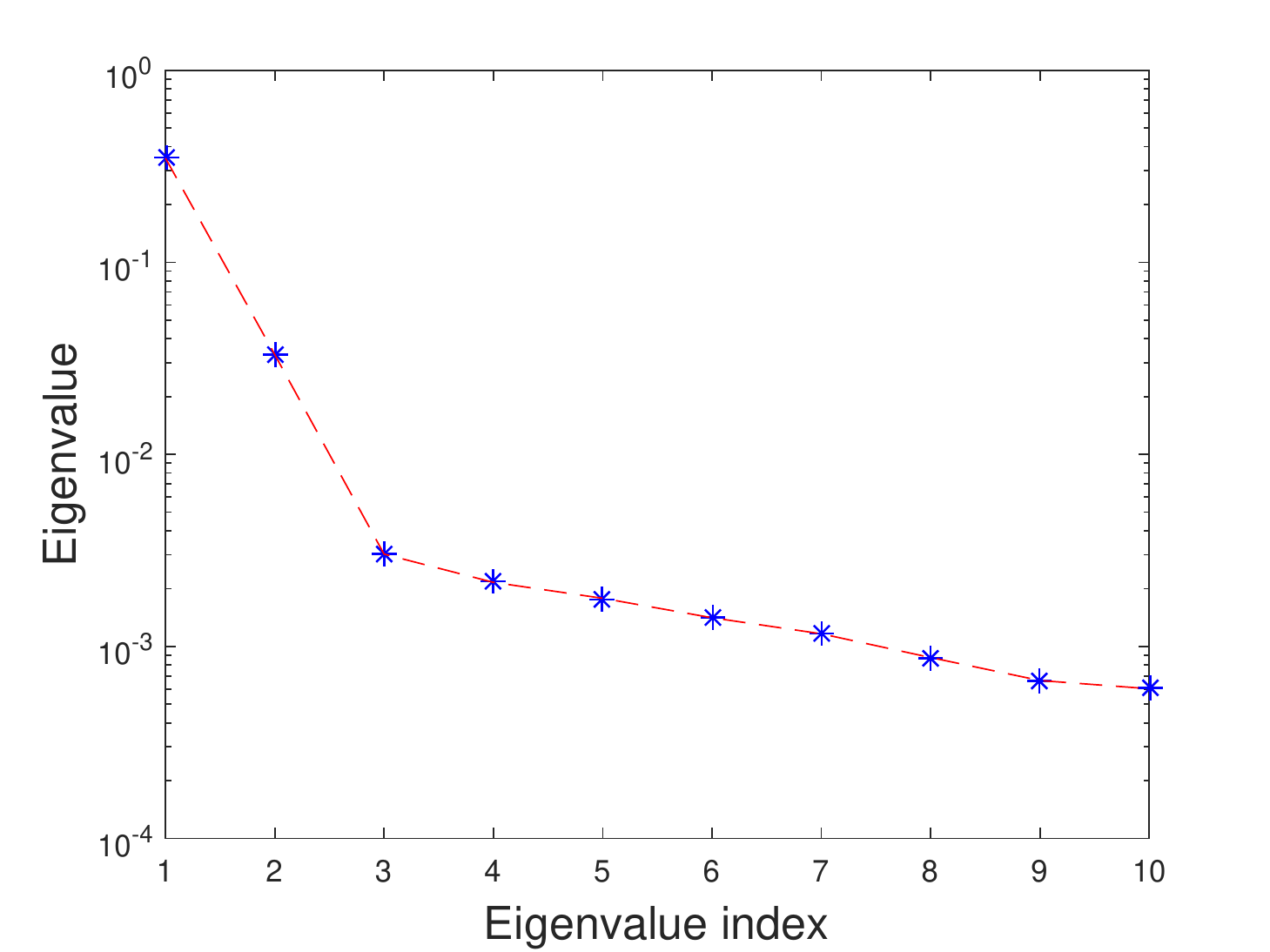} 
		\caption{ Decay of eigenvalues.}
		\label{fig:Example3eigenvalues-a}
	\end{subfigure}
	\begin{subfigure}[b]{0.45\textwidth}
		\includegraphics[width=1.0\linewidth]{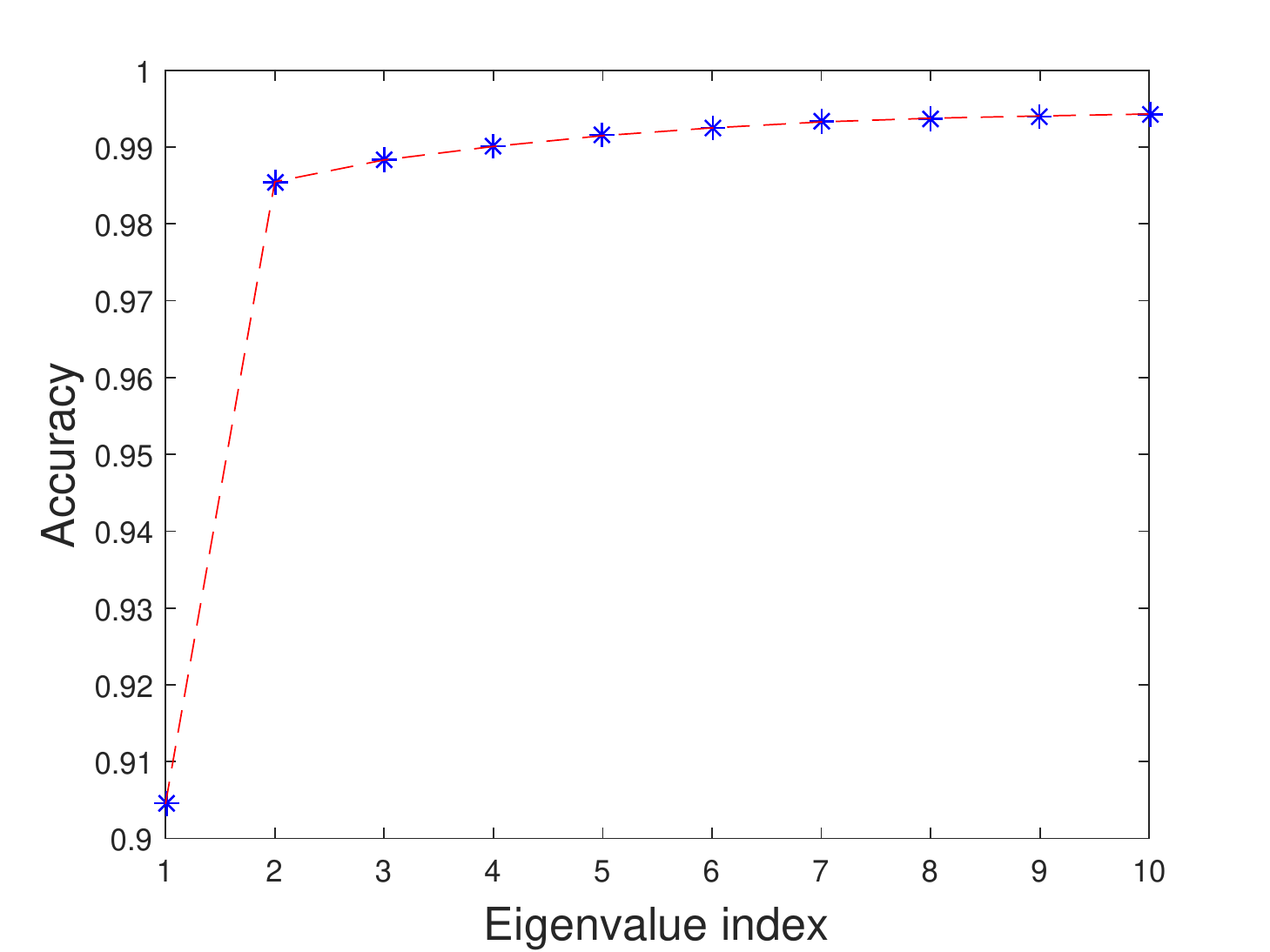} 
		\caption{ $1-\sqrt{\sum_{j=n+1}^{N}\lambda_{j}/\sum_{j=1}^{N}\lambda_{j}}$, $n=1,2,...$.} 
		\label{fig:Example3eigenvalues-b}
	\end{subfigure}
	\caption{The decay properties of the eigenvalues in the problem of Sec.\ref{sec:Example3}.}
	\label{fig:Example3eigenvalues}
\end{figure}  
 
Notice that both the coefficient and force contain randomness here. We put the random variables $\vxi(\omega)$ in the coefficient and the random variables $\gvec{\theta}(\omega)$ in the force together when we construct the mapping $\textbf{F}$. Moreover, the dimension of randomness, 18+2=20,  is too large even for sparse grids. Here we construct the mapping $\textbf{F}:(\vxi(\omega),\gvec{\theta}(\omega))\mapsto \textbf{c}(\omega)$ using the neural network as depicted in Figure \ref{fig:DNNstructure2}. The neural network has 4 hidden layers and each layer has 50 units. Naturally, the number of the input units is 20 and the number of the output units is $K$. The layer between input units and first layer of hidden units is an affine transform. So is the layer between output units and last layer of hidden units. Each two layers of hidden units are connected by an affine transform, a tanh (hyperbolic tangent) activation and a residual connection, i.e. $\textbf{h}_{l+1}=\tanh(\textbf{A}_l \textbf{h}_l+\textbf{b}_l)+\textbf{h}_l$, $l=1,2,3$, where $\textbf{h}_l$ is $l$-th layer of hidden units, $\textbf{A}_l$ is a 50-by-50 matrix and $\textbf{b}_l$ is a 50-by-1 vector. Under the same setting of neural network, if the rectified linear unit (ReLU), which is piecewise linear, is used
as the activation function, we observe a much bigger error. Therefore  we choose the hyperbolic tangent activation function and implement the residual neural network (ResNet) here \cite{he2016deep}. 

\begin{figure}[tbph]
\tikzset{global scale/.style={
    scale=#1,
    every node/.append style={scale=#1}
  }
}
	\centering
	\begin{tikzpicture}[global scale=0.6]
	\tikzstyle{inputvariables}  = [circle, very thick, fill=yellow,draw=black, minimum height=0.1cm, text width = 0.2cm]
	\tikzstyle{hiddenvariables}  = [circle, thick, draw =black, fill=blue,minimum height=0.1cm, text width = 0.2cm]
	\tikzstyle{outputvariables}  = [circle, very thick, draw=black, fill=red,minimum height=0.1cm, text width = 0.2cm]
	\tikzstyle{dottedvariables}  = [thick, fill=white, minimum size = 0.2cm]
	\tikzstyle{textnode} = [thick, fill=white, minimum size=0.1cm ]
	\node[inputvariables] (x1) at (0,0) {$\xi_1$};
	\node[inputvariables, below=0.4cm of x1] (x2) {$\xi_2$};
	\node[dottedvariables, below=0.1cm of x2] (x3) {$\vdots$};
	\node[inputvariables, below=0.1cm of x3] (x4) {$\xi_{r_1}$};
	\node[inputvariables, below=0.4cm of x4] (x5) {$\theta_{1}$};
	\node[inputvariables, below=0.1cm of x5] (x6) {$\theta_{2}$};
	\node[dottedvariables, below=0.1cm of x6] (x7) {$\vdots$};
	\node[inputvariables, below=0.4cm of x7] (x8) {$\theta_{r_2}$};
	\draw[-,thick,decorate, decoration={brace, raise=0.3cm}] (x4.south west)--(x1.north west);
	\node[textnode, above left of=x3,left=0.2cm] {$\gvec{\xi}(\omega)$};
	\draw[-,thick,decorate, decoration={brace, raise=0.3cm}] (x8.south west)--(x5.north west);
	\node[textnode, above left of=x7,left=0.2cm] {$\gvec{\theta}(\omega)$};
	\node[hiddenvariables] (h1) at (3,-1) {};
	\node[hiddenvariables, below=0.4cm of h1] (h2) {};
	\node[dottedvariables, below=0.1cm of h2] (h3) {$\vdots$};
	\node[dottedvariables, below=0.1cm of h3] (h4) {$\vdots$};
	\node[hiddenvariables, below=0.4cm of h4] (h5) {};
	\node[hiddenvariables, below=0.4cm of h5] (h6) {};
	
	\node[dottedvariables] (h31) at (5.5,-1) {$\cdots$};
	\node[dottedvariables, below=0.5cm of h31] (h32) {$\cdots$};
	\node[dottedvariables, below=0.5cm of h32] (h33) {$\cdots$};
	\node[dottedvariables, below=0.5cm of h33] (h34) {$\dots$};
	\node[dottedvariables, below=0.5cm of h34] (h35) {$\cdots$};
	\node[dottedvariables, below=0.5cm of h35] (h36) {$\cdots$};
	
	\node[hiddenvariables] (h21) at (8,-1) {};
	\node[hiddenvariables, below=0.4cm of h21] (h22) {};
	\node[dottedvariables, below=0.1cm of h22] (h23) {$\vdots$};
	\node[dottedvariables, below=0.1cm of h23] (h24) {$\vdots$};
	\node[hiddenvariables, below=0.4cm of h24] (h25) {};
	\node[hiddenvariables, below=0.4cm of h25] (h26) {};
	\node[outputvariables] (y1) at (11,0) {$c_1$};
	\node[outputvariables, below=0.4cm of y1] (y2) {$c_2$};
	\node[outputvariables, below=0.4cm of y2] (y3) {$c_3$};
	\node[dottedvariables, below=0.2cm of y3] (y4) {$\vdots$};
	\node[dottedvariables, below=0.2cm of y4] (y5) {$\vdots$};
	\node[dottedvariables, below=0.2cm of y5] (y6) {$\vdots$};
	\node[dottedvariables, below=0.2cm of y6] (y7) {$\vdots$};
	\node[outputvariables, below=0.2cm of y7] (y8) {$c_k$};
	\draw[-,thick,decorate, decoration={brace, raise=0.3cm}] (y1.north east)--(y8.south east);
	\node[textnode, above right of=y5, right=0.2cm] {$\mathbf{c}(\omega)$};
	\path [-] (x1) edge node {} (h1);
	\path [-] (x2) edge node {} (h1);
	\path [-] (x4) edge node {} (h1);
	\path [-] (x5) edge node {} (h1);
	\path [-] (x6) edge node {} (h1);
	\path [-] (x8) edge node {} (h1);
	\path [-] (x1) edge node {} (h2);
	\path [-] (x2) edge node {} (h2);
	\path [-] (x4) edge node {} (h2);
	\path [-] (x5) edge node {} (h2);
	\path [-] (x6) edge node {} (h2);
	\path [-] (x8) edge node {} (h2);
	\path [-] (x1) edge node {} (h5);
	\path [-] (x2) edge node {} (h5);
	\path [-] (x4) edge node {} (h5);
	\path [-] (x5) edge node {} (h5);
	\path [-] (x6) edge node {} (h5);
	\path [-] (x8) edge node {} (h5);
	\path [-] (x1) edge node {} (h6);
	\path [-] (x2) edge node {} (h6);
	\path [-] (x4) edge node {} (h6);
	\path [-] (x5) edge node {} (h6);
	\path [-] (x6) edge node {} (h6);
	\path [-] (x8) edge node {} (h6);
	\path [-] (h1) edge node {} (h31);
	\path [-] (h2) edge node {} (h31);
	\path [-] (h5) edge node {} (h31);
	\path [-] (h6) edge node {} (h31);
	\path [-] (h1) edge node {} (h32);
	\path [-] (h2) edge node {} (h32);
	\path [-] (h5) edge node {} (h32);
	\path [-] (h6) edge node {} (h32);
	\path [-] (h1) edge node {} (h33);
	\path [-] (h2) edge node {} (h33);
	\path [-] (h5) edge node {} (h33);
	\path [-] (h6) edge node {} (h33);
	\path [-] (h1) edge node {} (h34);
	\path [-] (h2) edge node {} (h34);
	\path [-] (h5) edge node {} (h34);
	\path [-] (h6) edge node {} (h34);
	\path [-] (h1) edge node {} (h35);
	\path [-] (h2) edge node {} (h35);
	\path [-] (h5) edge node {} (h35);
	\path [-] (h6) edge node {} (h35);
	\path [-] (h1) edge node {} (h36);
	\path [-] (h2) edge node {} (h36);
	\path [-] (h5) edge node {} (h36);
	\path [-] (h6) edge node {} (h36);
	
	\path [-] (h21) edge node {} (h31);
	\path [-] (h22) edge node {} (h31);
	\path [-] (h25) edge node {} (h31);
	\path [-] (h26) edge node {} (h31);
	\path [-] (h21) edge node {} (h32);
	\path [-] (h22) edge node {} (h32);
	\path [-] (h25) edge node {} (h32);
	\path [-] (h26) edge node {} (h32);
	\path [-] (h21) edge node {} (h33);
	\path [-] (h22) edge node {} (h33);
	\path [-] (h25) edge node {} (h33);
	\path [-] (h26) edge node {} (h33);
	\path [-] (h21) edge node {} (h34);
	\path [-] (h22) edge node {} (h34);
	\path [-] (h25) edge node {} (h34);
	\path [-] (h26) edge node {} (h34);
	\path [-] (h21) edge node {} (h35);
	\path [-] (h22) edge node {} (h35);
	\path [-] (h25) edge node {} (h35);
	\path [-] (h26) edge node {} (h35);
	\path [-] (h21) edge node {} (h36);
	\path [-] (h22) edge node {} (h36);
	\path [-] (h25) edge node {} (h36);
	\path [-] (h26) edge node {} (h36);
	\path [-] (h21) edge node {} (y1);
	\path [-] (h22) edge node {} (y1);
	\path [-] (h25) edge node {} (y1);
	\path [-] (h26) edge node {} (y1);
	\path [-] (h21) edge node {} (y2);
	\path [-] (h22) edge node {} (y2);
	\path [-] (h25) edge node {} (y2);
	\path [-] (h26) edge node {} (y2);
	\path [-] (h21) edge node {} (y3);
	\path [-] (h22) edge node {} (y3);
	\path [-] (h25) edge node {} (y3);
	\path [-] (h26) edge node {} (y3);
	\path [-] (h21) edge node {} (y8);
	\path [-] (h22) edge node {} (y8);
	\path [-] (h25) edge node {} (y8);
	\path [-] (h26) edge node {} (y8);
	
	\node[textnode, font=\fontsize{15}{6}\selectfont, above=1.0cm of h31] (Text1){Hidden units};
	\node[textnode, font=\fontsize{15}{6}\selectfont, left =1.8cm of Text1] (Text2) {Input units};
	\node[textnode, font=\fontsize{15}{6}\selectfont, right = 1.8cm of Text1] {Output units};
	\end{tikzpicture}
	\caption{Structure of neural network, where $r_1=18$ and $r_2=2$.}
	\label{fig:DNNstructure2}
\end{figure}
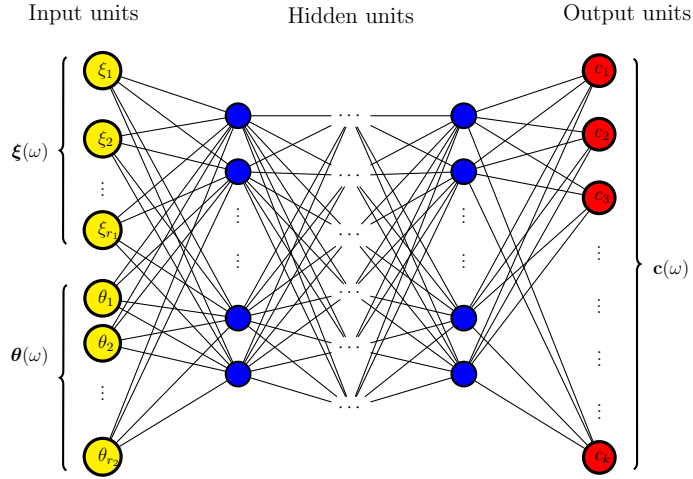 
We use $N_1=5000$ samples for network training in the offline stage and use $N_2=200$ samples for testing in the online stage. The sample data pairs for training are $\{(\vxi^n(\omega),\gvec{\theta}^n(\omega)),\textbf{c}^n(\omega)\}_{n=1}^{N_1}$, where 
$\vxi^n(\omega)\in [-\frac{1}{5},\frac{1}{5}]^{18}$, $\gvec{\theta}^n(\omega))\in [\frac{1}{4},\frac{3}{4}]\times[\frac{1}{16},\frac{5}{16}]$, and $\textbf{c}^n(\omega)\in R^{K}$.  We define the loss function of network training as  
\begin{align}
loss\big(\{\textbf{c}^n\},\{\textbf{\^{c}}^n\}\big) = \frac{1}{N_1}\sum_{n=1}^{N_1}\frac{1}{K}|\textbf{c}^{n}-\textbf{\^{c}}^{n}|^2, 
\end{align}
where $\textbf{c}^{n}$ are the training data and $\textbf{\^{c}}^n$ are the output of the neural network.

Figure \ref{fig:Example3locall2err-a} shows the value of loss function during training procedure. Figure \ref{fig:Example3locall2err-b} shows the corresponding mean relative error of the testing samples in $L^2$ norm. 
Eventually the relative error of the neural network reaches about $1.5\times 10^{-2}$. 
Figure \ref{fig:Example3locall2err-c} shows the corresponding mean relative error of the testing samples in $H^1$ norm. We remark that many existing methods become extremely expensive or infeasible when the problem is parameterized by high-dimensional random variables like this one. 
  
\begin{figure}[htbp]
	\centering
	\begin{subfigure}[b]{0.32\textwidth}
		$K=5$\\
		\includegraphics[width=1.0\linewidth]{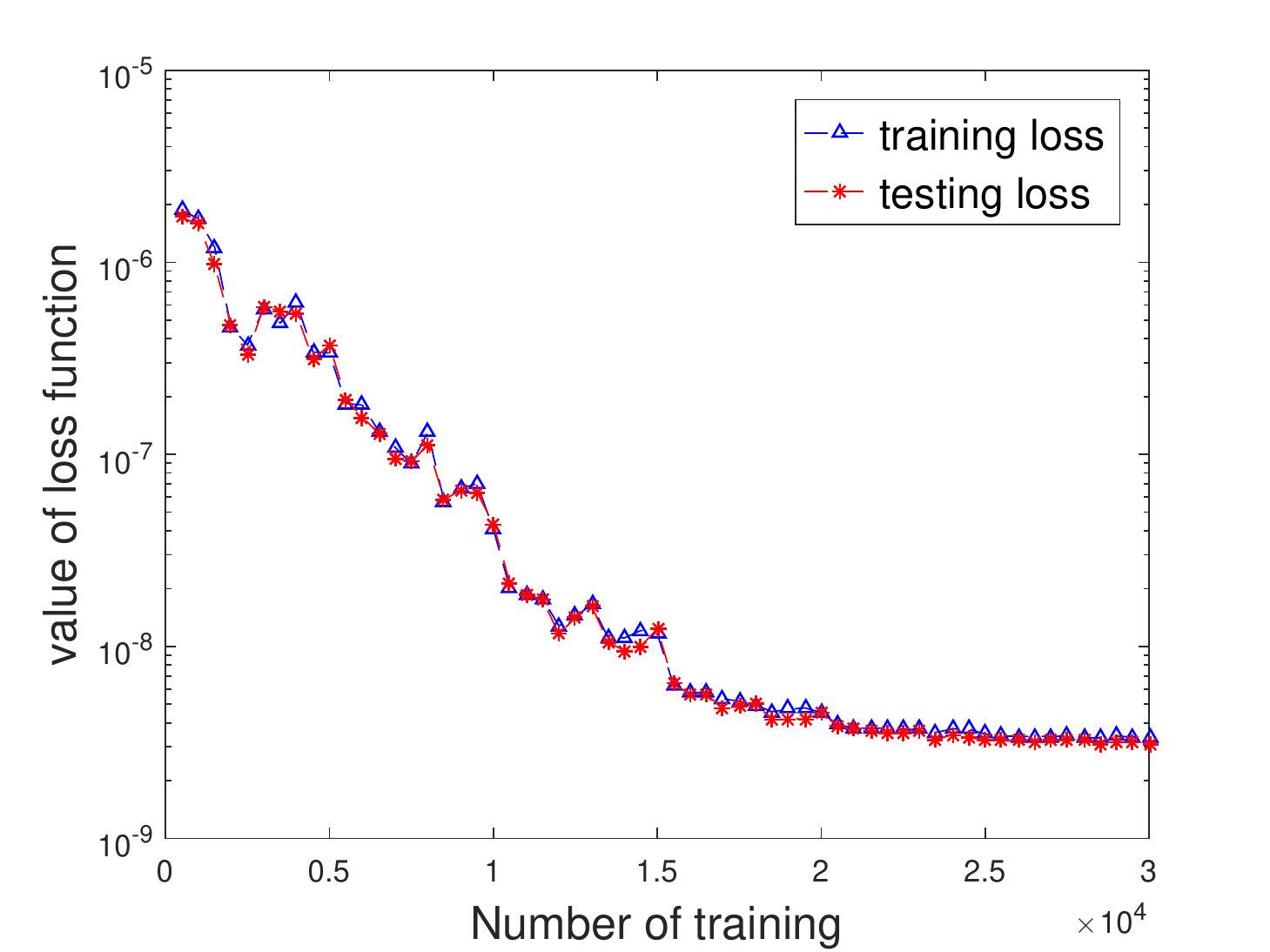} \\
		$K=10$\\
		\includegraphics[width=1.0\linewidth]{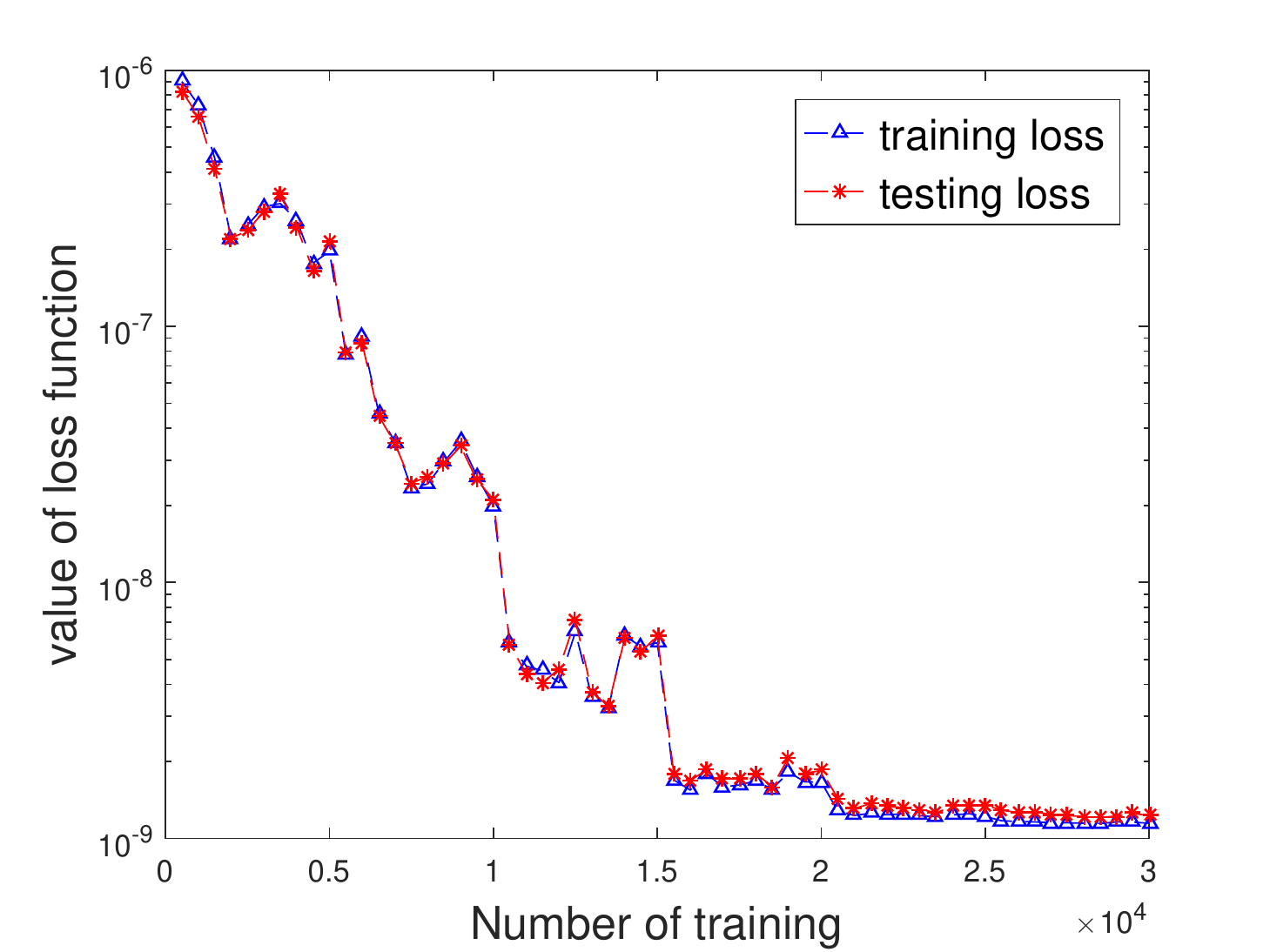} \\
		$K=20$\\
		\includegraphics[width=1.0\linewidth]{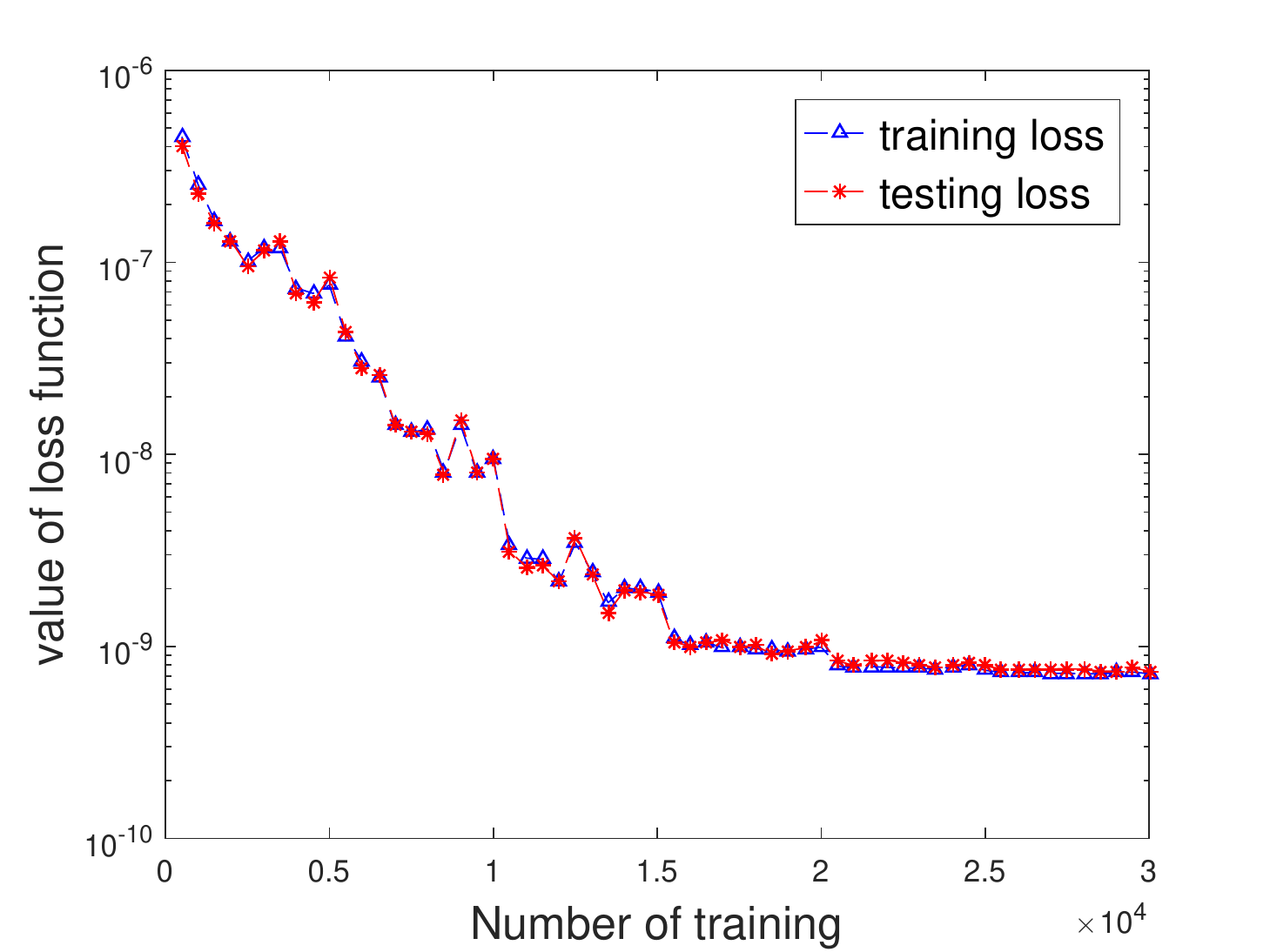} 
		\caption{ Loss.}
		\label{fig:Example3locall2err-a}
	\end{subfigure}
	\begin{subfigure}[b]{0.32\textwidth}
		~~\\
		\includegraphics[width=1.0\linewidth]{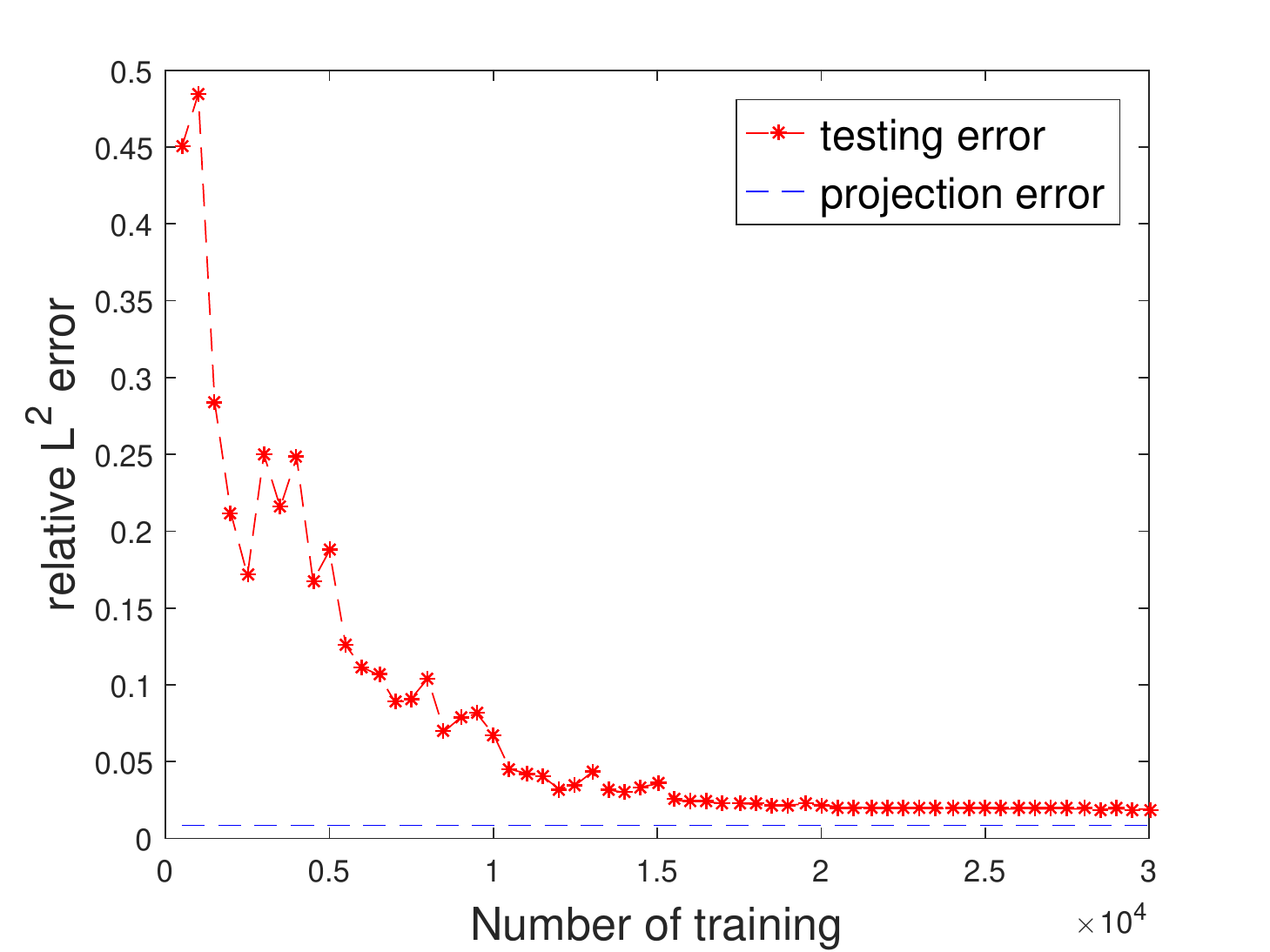} \\
		~~\\
		\includegraphics[width=1.0\linewidth]{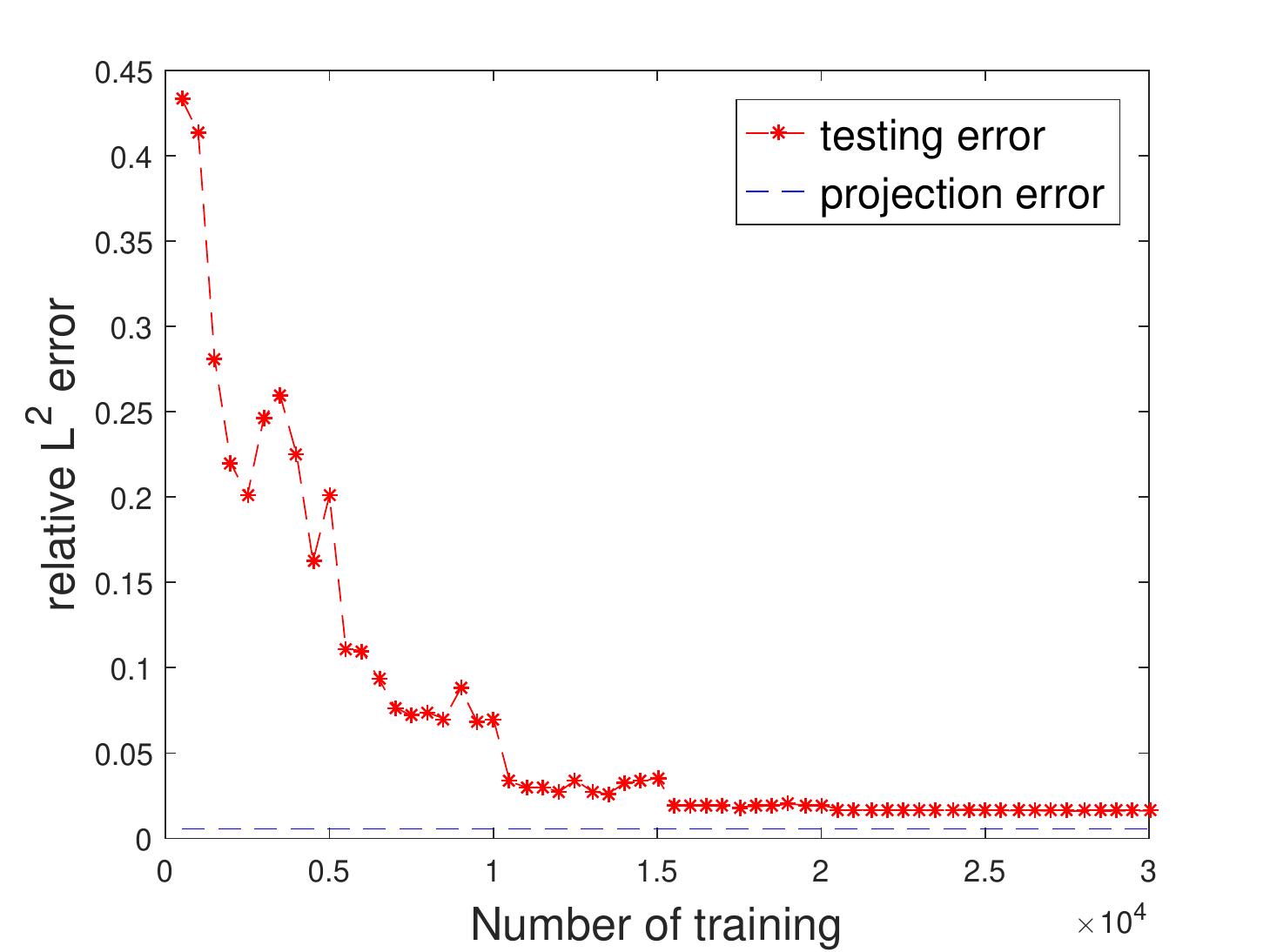} \\
		~~\\
		\includegraphics[width=1.0\linewidth]{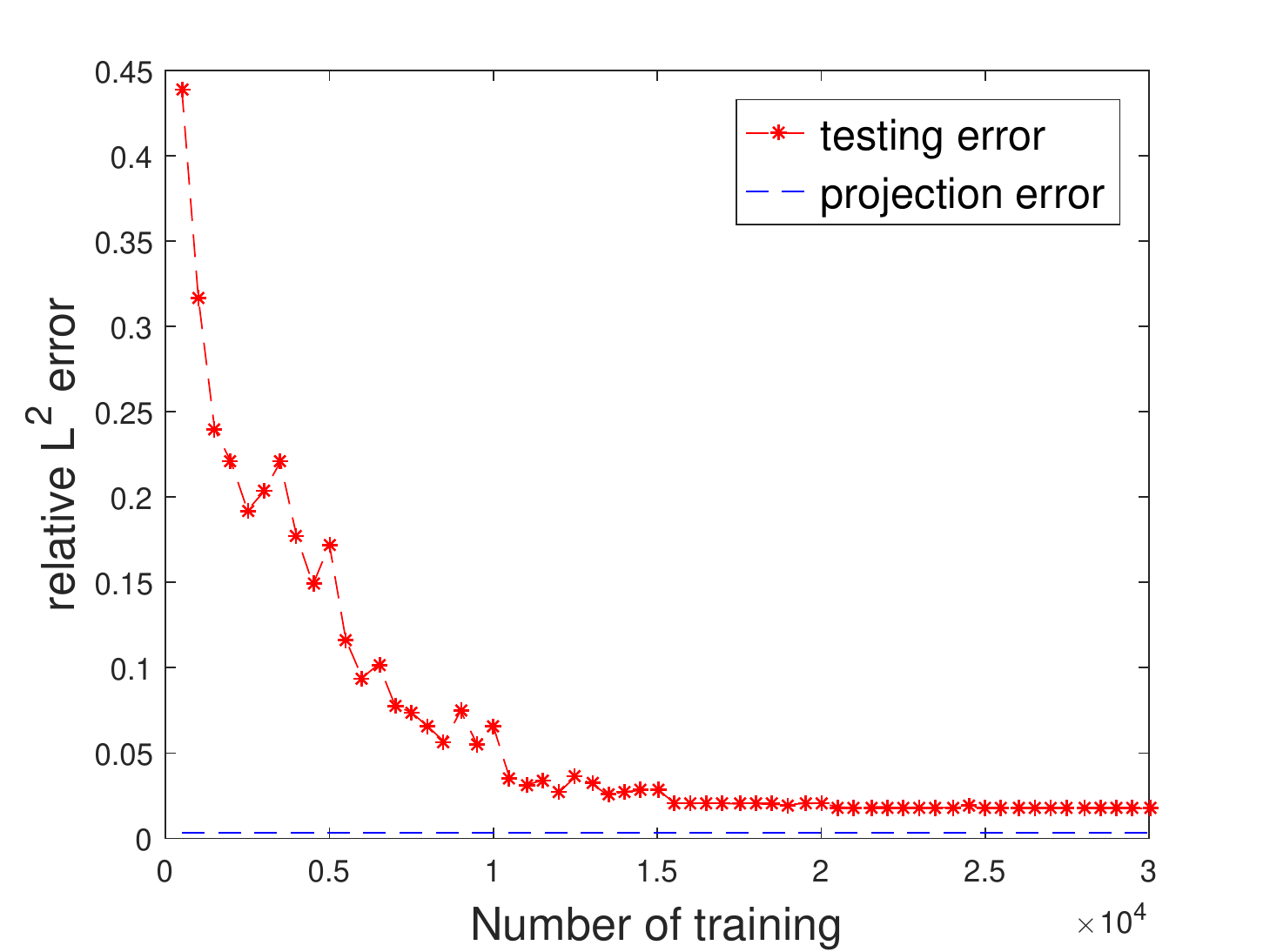} 
		\caption{ Relative $L^2$ error.} 
		\label{fig:Example3locall2err-b}
	\end{subfigure}
	\begin{subfigure}[b]{0.32\textwidth}
		~~\\
		\includegraphics[width=1.0\linewidth]{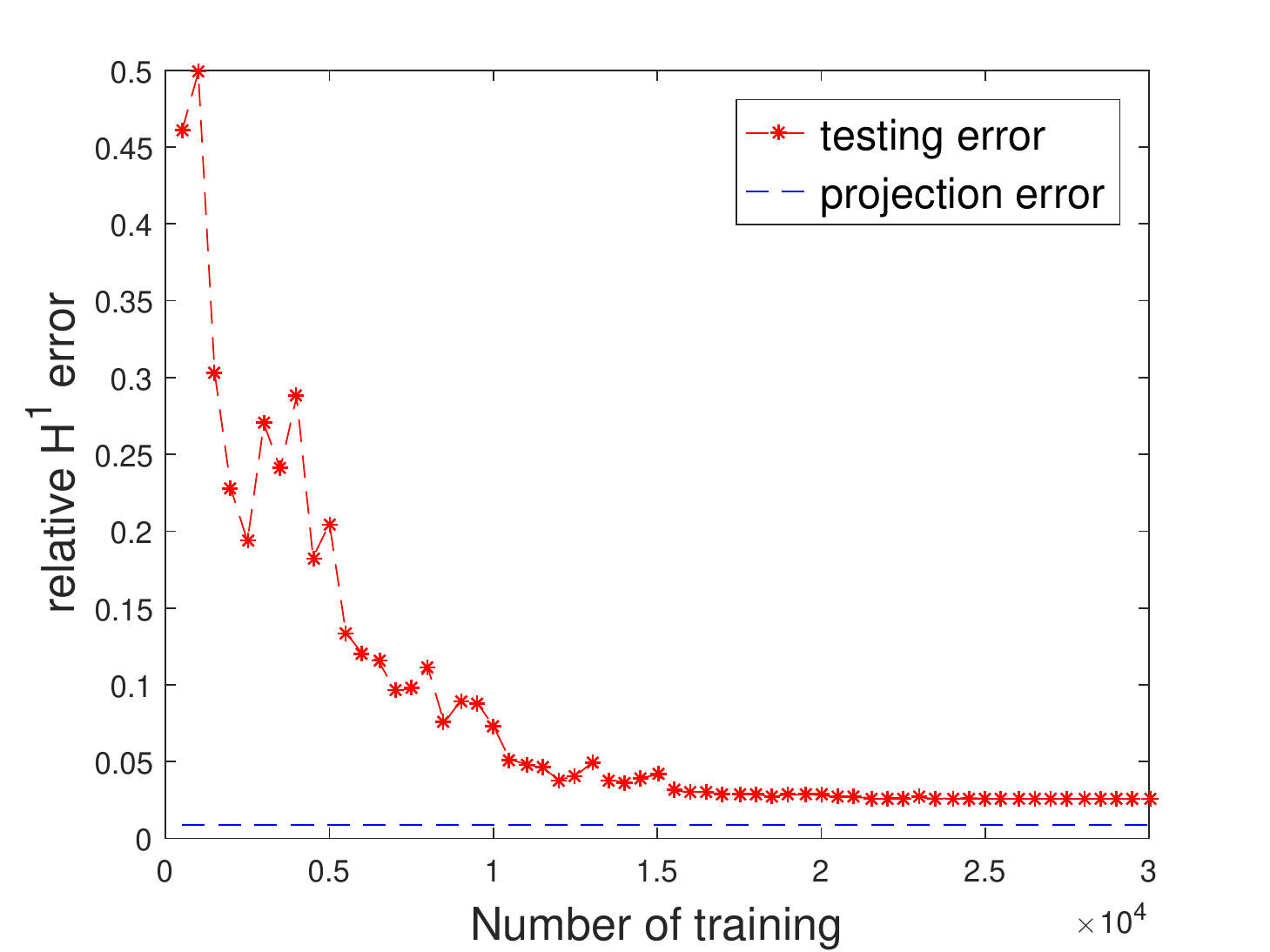}\\ 
		~~\\
		\includegraphics[width=1.0\linewidth]{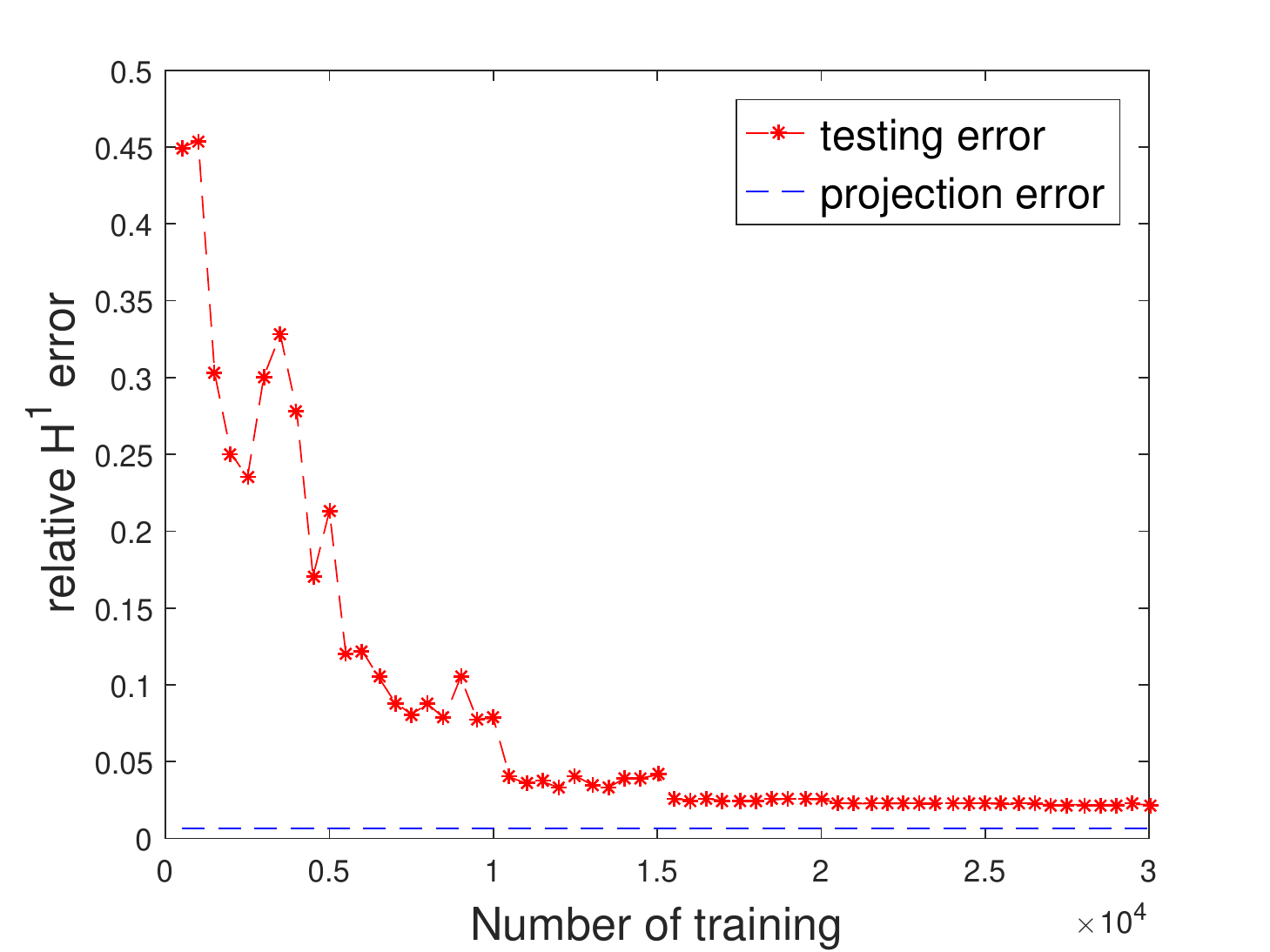}\\
		~~\\
		\includegraphics[width=1.0\linewidth]{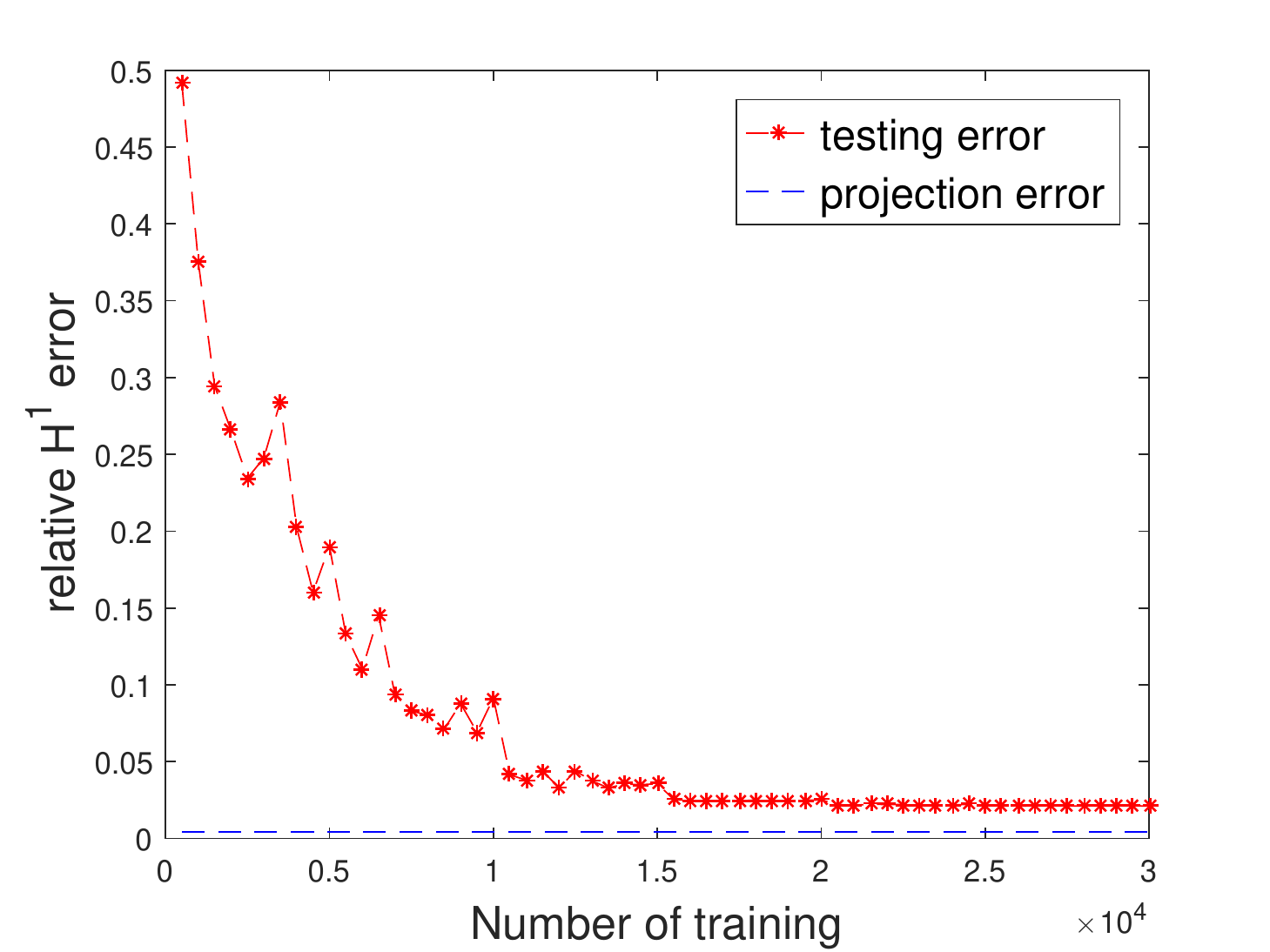}
		\caption{ Relative $H^1$ error.} 
		\label{fig:Example3locall2err-c}
	\end{subfigure}
	\caption{ First column: the value of loss function during training procedure. Second column and third column: the mean relative errors of the testing set during training procedure  in $L^2$  and $H^1$ norm respectively.} 
	\label{fig:Example3locall2err}
\end{figure}

\subsection{An example with unknown random coefficient and source function}\label{sec:Example4}
\noindent
Here we present an example where the models of the random coefficient and source are unknown. Only a set of sample solutions are provided as well as a few censors can be placed at certain locations for solution measurements. This kind of scenario appears often in practice. We take the least square fitting method as described in Section \ref{sec:LS}. Our numerical experiment is still based on \eqref{randommultiscaleelliptic}, which is used to generate solution samples (instead of experiments or measurements in real practice). But once the data are generated, we do not assume any knowledge of the coefficient or the source when computing a new solution. 


To be specific, the coefficient takes the form
\begin{align}
a(x,y,\omega) = \exp\Big(\sum_{i=1}^{24} \sin(2\pi \frac{x\sin(\frac{i\pi}{24}) +y\cos(\frac{i\pi}{24})   }{\epsilon_i}   )\xi_i(\omega) \Big)
\label{coefficientofexample4}
\end{align}
where $\epsilon_i=\frac{1+i}{100}$, $i=1,2,\cdots,24$ and $\xi_i(\omega)$, $i=1,...,24$ are i.i.d. uniform random variables in $[-\frac{1}{6},\frac{1}{6}]$. The force function is a random function $f(x,y) = \sin(\pi(\theta_1x+2\theta_2))\cos(\pi(\theta_3y+2\theta_4))\cdot I_{D_2}(x,y)$ with  i.i.d. uniform random variables $\theta_1,\theta_2,\theta_3,\theta_4$ in $[0,2]$. We first generate $N=2000$ solutions samples (using standard FEM) $u(x_j, \omega_i), i=1, \ldots, N, j=1, \ldots, J$, where $x_j$ are the points where solution samples are measured. Then a set of $K$ data-driven basis $\phi_k(x_j), j=1, \ldots, J, k=1, \dots, K$ are extracted from the solution samples as before.  

Next we determine $M$ good sensing locations from the data-driven basis so that the least square problem \eqref{eq:LS} is not ill-conditioned. We follow the method proposed in \cite{Kutz2017Sensor}. Define $\Phi=[\boldsymbol{\phi}_1, \ldots, \boldsymbol{\phi}_K]\in R^{J\times K}$, where $\boldsymbol{\phi}_k=[\phi_k(x_1), \ldots, \phi_k(x_J)]^T$. If $M=K$, QR factorization with column pivoting is performed on $\Phi^T$. If $M>K$, QR factorization with pivoting is performed on $\Phi\Phi^T$. The first $M$ pivoting indices provide the measurement  locations. Once a new solution is measured at these $M$ selected locations, the least square problem \eqref{eq:LS} is solved to determine the coefficients $c_1, c_2, \ldots, c_K$ and the new solution is approximated by $u(x_j,\omega)=\sum_{k=1}^K c_k\phi_k(x_j)$.

 
In Figure \ref{fig:Example4localerrors} and Figure \ref{fig:Example4globalerrors}, we show the results of the local problem and global problem, respectively. In these numerical results, we 
compared the error between the reconstructed solutions and the reference solution. We find the our proposed method works well for problem \eqref{randommultiscaleelliptic} with a non-parametric coefficient or source as well.

\begin{figure}[htbp]
	\centering
	\begin{subfigure}[b]{0.45\textwidth}
		\includegraphics[width=1.0\linewidth]{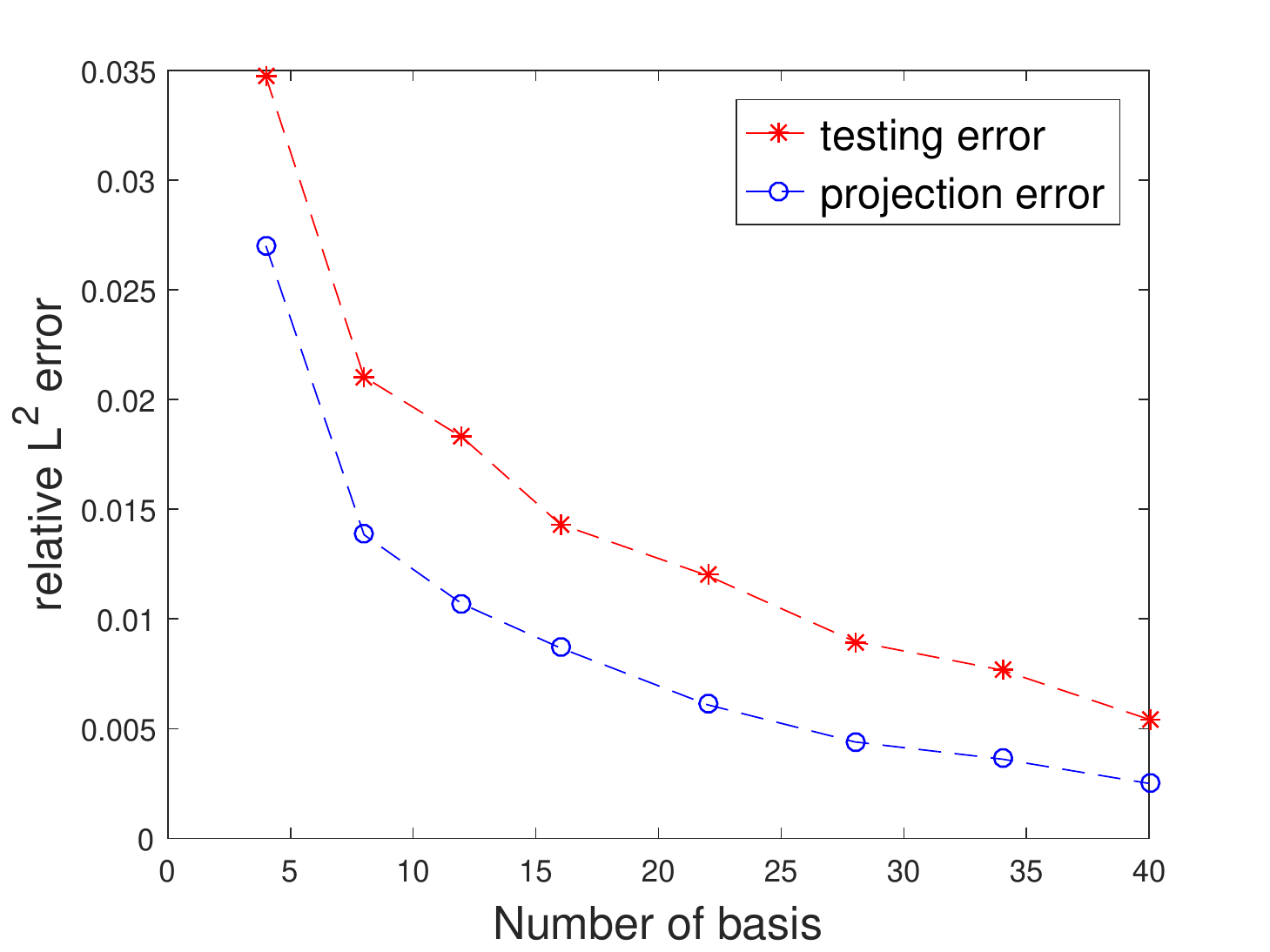}
		\label{fig:Example4errors-a}
	\end{subfigure}
	\begin{subfigure}[b]{0.45\textwidth}
		\includegraphics[width=1.0\linewidth]{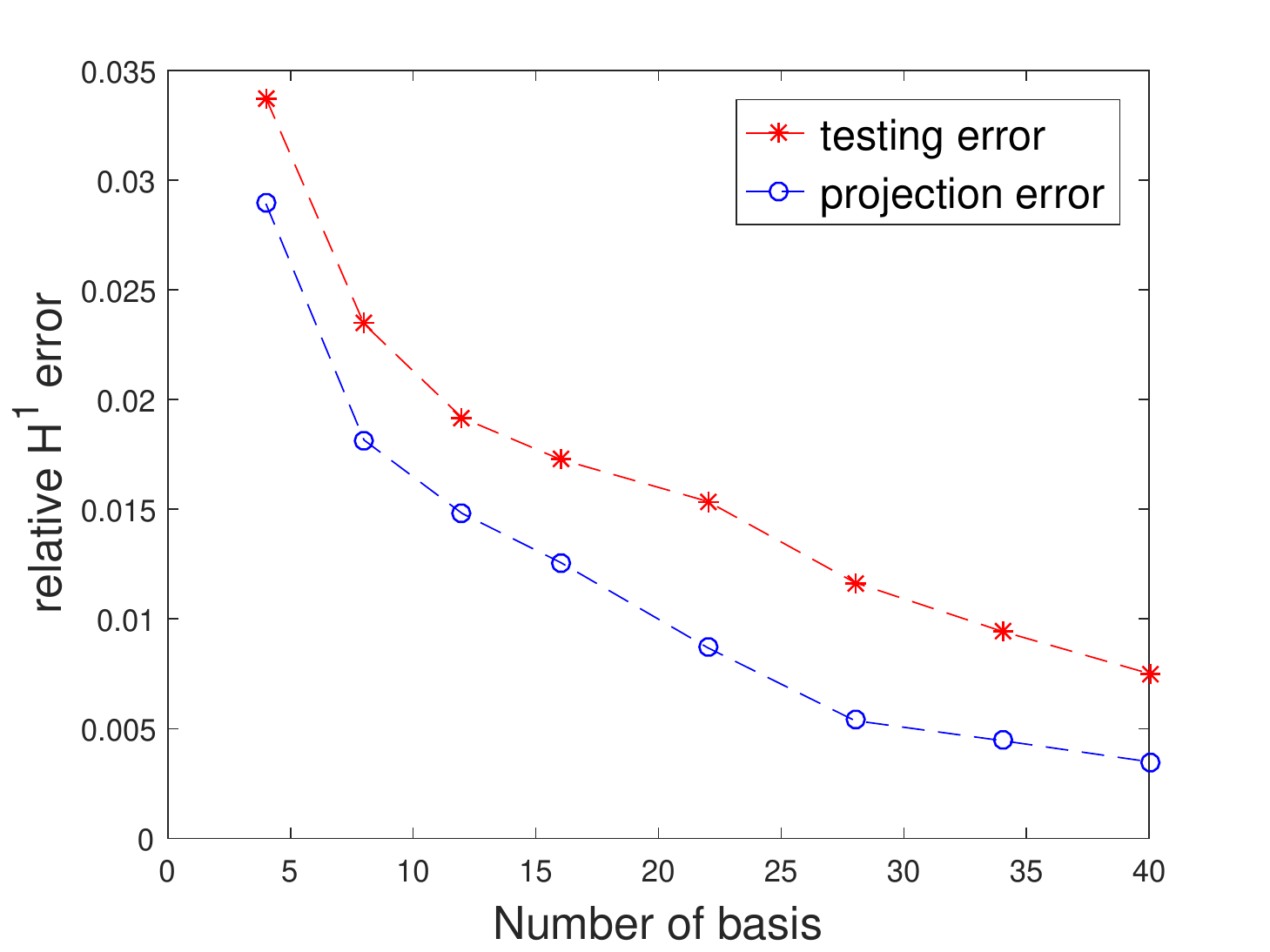}
		\label{fig:Example4errors-b}
	\end{subfigure}
	\caption{ The relative errors with increasing number of basis in the local problem of Sec.\ref{sec:Example4} .}
	\label{fig:Example4localerrors}
\end{figure}

\begin{figure}[htbp]
	\centering
	\begin{subfigure}[b]{0.45\textwidth}
		\includegraphics[width=1.0\linewidth]{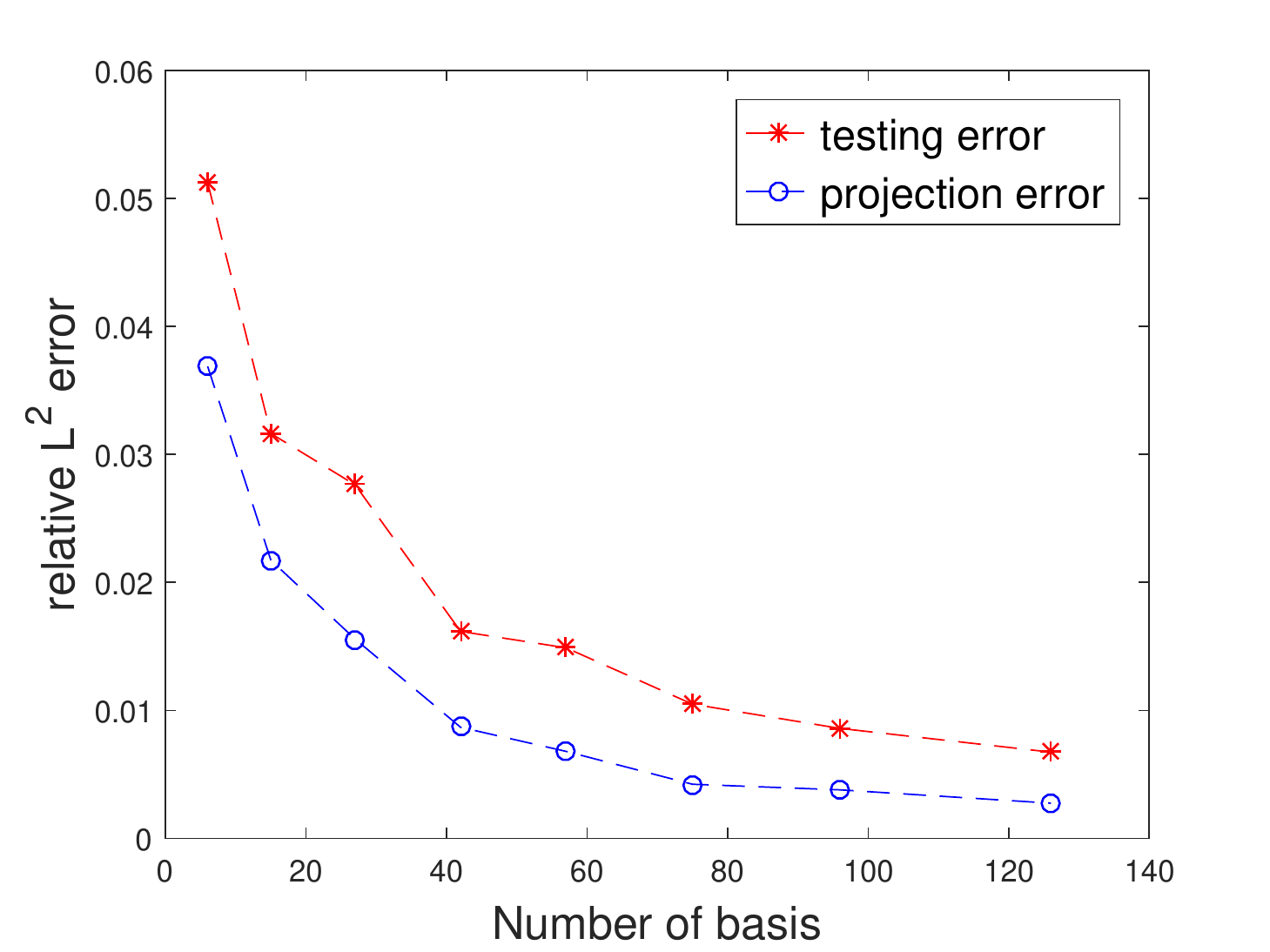}
	\end{subfigure}
	\begin{subfigure}[b]{0.45\textwidth}
		\includegraphics[width=1.0\linewidth]{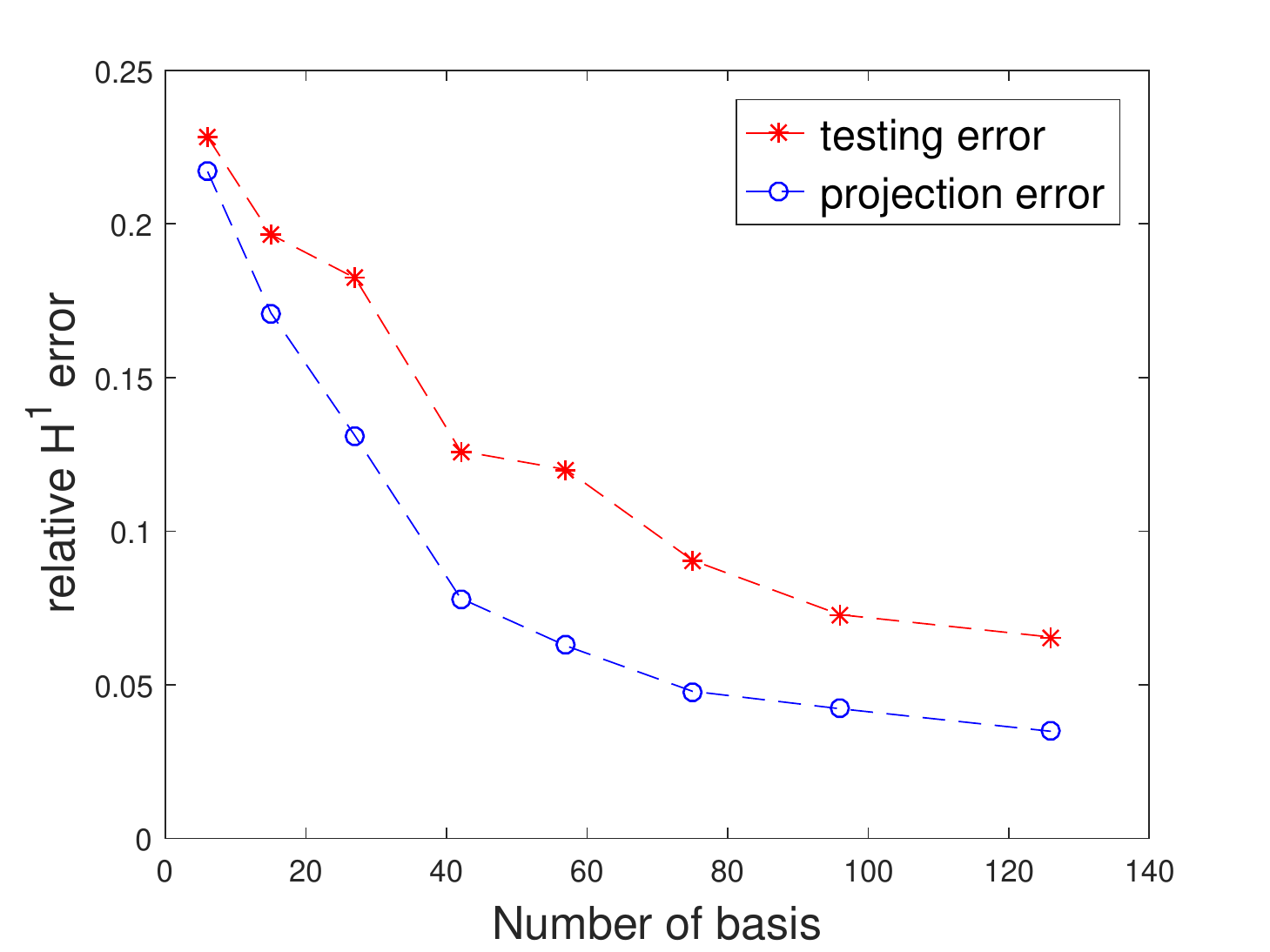}
	\end{subfigure}
	\caption{ The relative errors with increasing number of basis in the global problem of Sec.\ref{sec:Example4}.}
	\label{fig:Example4globalerrors}
\end{figure}
 
\section{Conclusion} \label{sec:Conclusion}
\noindent 
In this paper, we propose a data-driven approach to solve elliptic PDEs with multiscale and random coefficient which arise in various applications, such as heterogeneous porous media flow problems in water aquifer and oil reservoir simulations.  The key idea for our method, which is motivated by the high separable approximation of the underlying Green's function, is to extract a problem specific low dimensional structure in the solution space and construct its basis from the data. Once the data-driven basis is available, depending on different setups, we design several ways to compute a new solution efficiently.

Error analysis based on sampling error of the coefficients and the projection error of the data-driven basis is presented to provide some guidance in the implementation of our method. Numerical examples show that the proposed method is very efficient especially when the problem has relative high dimensional random input.




\section*{Acknowledgements}
\noindent
The research of S. Li is partially supported by the Doris Chen Postgraduate Scholarship. 
The research of Z. Zhang is supported by the Hong Kong RGC General Research Funds (Projects 27300616, 17300817, and 17300318), National Natural Science Foundation of China (Project 11601457), Seed Funding Programme for Basic Research (HKU), and Basic Research Programme (JCYJ20180307151603959) of The Science, Technology and Innovation Commission of Shenzhen Municipality. The research of H. Zhao is partially supported by NSF grant DMS-1622490 and DMS-1821010. This research is made possible by a donation to the Big Data Project Fund, HKU, from Dr Patrick Poon whose generosity is gratefully acknowledged.



\bibliographystyle{siam}
\bibliography{ZWpaper}

\begin{thebibliography}{10}

\bibitem{abdulle2013multilevel}
{\sc A.~Abdulle, A.~Barth, and C.~Schwab}, {\em Multilevel {M}onte {C}arlo
  methods for stochastic elliptic multiscale {P}{D}{E}s}, Multiscale Modeling
  \& Simulation, 11 (2013), pp.~1033--1070.

\bibitem{Ghanem:08}
{\sc M.~Arnst and R.~Ghanem}, {\em Probabilistic equivalence and stochastic
  model reduction in multiscale analysis}, Comput. methods Appl. Mech. Engrg,
  197(43) (2008), pp.~3584--3592.

\bibitem{Zabaras:06}
{\sc B.~V. Asokan and N.~Zabaras}, {\em A stochastic variational multiscale
  method for diffusion in heterogeneous random media}, Journal of Computational
  Physics, 218 (2006), pp.~654--676.

\bibitem{Babuska:07}
{\sc I.~Babuska, F.~Nobile, and R.~Tempone}, {\em A stochastic collocation
  method for elliptic partial differential equations with random input data},
  SIAM J. Numer. Anal., 45 (2007), pp.~1005--1034.

\bibitem{babuska:04}
{\sc I.~Babuska, R.~Tempone, and G.~Zouraris}, {\em Galerkin finite element
  approximations of stochastic elliptic partial differential equations}, SIAM
  J. Numer. Anal., 42 (2004), pp.~800--825.

\bibitem{PateraMaday:2004}
{\sc M.~Barrault, Y.~Maday, N.~C. Nguyen, and A.~T. Patera}, {\em An empirical
  interpolation method: application to efficient reduced-basis discretization
  of partial differential equations}, Comptes Rendus Mathematique, 339(9)
  (2004), pp.~667--672.

\bibitem{bebendorf2003}
{\sc M.~Bebendorf and W.~Hackbusch}, {\em Existence of {H}-matrix approximants
  to the inverse {F}{E}-matrix of elliptic operators with {L}
  infinity-coefficients}, Numerische Mathematik, 95 (2003), pp.~1--28.

\bibitem{BebendorfHackbusch:2003}
\leavevmode\vrule height 2pt depth -1.6pt width 23pt, {\em Existence of
  {H}-matrix approximants to the inverse {F}{E}-matrix of elliptic operators
  with {L} infinity coefficients}, Numerische Mathematik, 95 (2003), pp.~1--28.

\bibitem{Willcox2015PODsurvey}
{\sc P.~Benner, S.~Gugercin, and K.~Willcox}, {\em A survey of projection-based
  model reduction methods for parametric dynamical systems}, SIAM Review, 57
  (2015), pp.~483--531.

\bibitem{HolmesLumleyPOD:1993}
{\sc G.~Berkooz, P.~Holmes, and J.~L. Lumley}, {\em The proper orthogonal
  decomposition in the analysis of turbulent flows}, Annual review of fluid
  mechanics, 25(1) (1993), pp.~539--575.

\bibitem{BrysonZhaoZhong:2019}
{\sc J.~Bryson, H.~Zhao, and Y.~Zhong}, {\em Intrinsic complexity and scaling
  laws: from random fields to random vectors}, SIAM Journal on Multiscale
  Modeling and Simulation, 17 (2019), pp.~460--481.

\bibitem{Griebel:04}
{\sc H.~J. Bungartz and M.~Griebel}, {\em Sparse grids}, Acta Numerica, 13
  (2004), pp.~147--269.

\bibitem{ChengHouYanZhang:13}
{\sc M.~Cheng, T.~Y. Hou, M.~Yan, and Z.~Zhang}, {\em A data-driven stochastic
  method for elliptic {PDE}s with random coefficients}, SIAM J. UQ, 1 (2013),
  pp.~452--493.

\bibitem{chung2018cluster}
{\sc E.~Chung, Y.~Efendiev, W.~Leung, and Z.~Zhang}, {\em Cluster-based
  generalized multiscale finite element method for elliptic {P}{D}{E}s with
  random coefficients}, Journal of Computational Physics, 371 (2018),
  pp.~606--617.

\bibitem{dolzmann1995estimates}
{\sc G.~Dolzmann and S.~M{\"u}ller}, {\em Estimates for green's matrices of
  elliptic systems byl p theory}, Manuscripta mathematica, 88 (1995),
  pp.~261--273.

\bibitem{efendiev2015multilevel}
{\sc Y.~Efendiev, C.~Kronsbein, and F.~Legoll}, {\em Multilevel {M}onte {C}arlo
  approaches for numerical homogenization}, Multiscale Modeling \& Simulation,
  13 (2015), pp.~1107--1135.

\bibitem{EngquistZhao:2018}
{\sc B.~Engquist and H.~Zhao}, {\em Approximate separability of the {G}reen's
  function of the {H}elmholtz equation in the high frequency limit},
  Communications on Pure and Applied Mathematics, 71 (2018), pp.~2220--2274.

\bibitem{Ghanem:91}
{\sc R.~Ghanem and P.~Spanos}, {\em Stochastic finite elements: a spectral
  approach.}, Springer-Verlag, New York, 1991.

\bibitem{graham2011quasi}
{\sc I.~Graham, F.~Kuo, D.~Nuyens, R.~Scheichl, and I.~Sloan}, {\em
  Quasi-{M}onte {C}arlo methods for elliptic {P}{D}{E}s with random
  coefficients and applications}, Journal of Computational Physics, 230 (2011),
  pp.~3668--3694.

\bibitem{Grahamquasi:2015}
{\sc I.~G. Graham, F.~Y. Kuo, J.~A. Nichols, R.~Scheichl, C.~Schwab, and I.~H.
  Sloan}, {\em Quasi-{M}onte {C}arlo finite element methods for elliptic
  {P}{D}{E}s with lognormal random coefficients}, Numerische Mathematik, 131(2)
  (2015), pp.~329--368.

\bibitem{gruter1982green}
{\sc M.~Gr{\"u}ter and K.~Widman}, {\em The green function for uniformly
  elliptic equations}, Manuscripta Mathematica, 37 (1982), pp.~303--342.

\bibitem{he2016deep}
{\sc K.~He, X.~Zhang, S.~Ren, and J.~Sun}, {\em Deep residual learning for
  image recognition}, in Proceedings of the IEEE conference on computer vision
  and pattern recognition, 2016, pp.~770--778.

\bibitem{hou2015heterogeneous}
{\sc T.~Hou and P.~Liu}, {\em A heterogeneous stochastic {FEM} framework for
  elliptic {PDE}s}, Journal of Computational Physics, 281 (2015), pp.~942--969.

\bibitem{ZhangHouLiu:15}
{\sc T.~Hou, P.~Liu, and Z.~Zhang}, {\em A localized data-driven stochastic
  method for elliptic {PDE}s with random coefficients}, Bull. Inst. Math. Acad.
  Sin. (N.S.), 1 (2016), pp.~179--216.

\bibitem{hou2019model}
{\sc T.~Hou, D.~Ma, and Z.~Zhang}, {\em A model reduction method for multiscale
  elliptic {PDE}s with random coefficients using an optimization approach},
  Multiscale Modeling \& Simulation, 17 (2019), pp.~826--853.

\bibitem{WuanHou:06}
{\sc T.~Y. Hou, W.~Luo, B.~Rozovskii, and H.~M. Zhou}, {\em Wiener chaos
  expansions and numerical solutions of randomly forced equations of fluid
  mechanics}, J. Comput. Phys., 216 (2006), pp.~687--706.

\bibitem{Kevrekidis:2003}
{\sc I.~G. Kevrekidis, C.~W. Gear, J.~M. Hyman, P.~G. Kevrekidid, O.~Runborg,
  and C.~Theodoropoulos}, {\em Equation-free, coarse-grained multiscale
  computation: {E}nabling mocroscopic simulators to perform system-level
  analysis}, Communications in Mathematical Sciences, 1(4) (2003),
  pp.~715--762.

\bibitem{Kutz2017Sensor}
{\sc K.~Manohar, B.~Brunton, J.~Kutz, and S.~Brunton}, {\em Data-driven sparse
  sensor placement for reconstruction}, arXiv:1701.07569,  (2017).

\bibitem{matthies:05}
{\sc H.~G. Matthies and A.~Keese}, {\em Galerkin methods for linear and
  nonlinear elliptic stochastic partial differential equations}, Comput. Method
  Appl. Mech. Eng., 194 (2005), pp.~1295--1331.

\bibitem{Najm:09}
{\sc H.~N. Najm}, {\em Uncertainty quantification and polynomial chaos
  techniques in computational fluid dynamics}, Annual Review of Fluid
  Mechanics, 41 (2009), pp.~35--52.

\bibitem{Webster:08}
{\sc F.~Nobile, R.~Tempone, and C.~Webster}, {\em A sparse grid stochastic
  collocation method for partial differential equations with random input
  data}, SIAM J. Numer. Anal., 46 (2008), pp.~2309--2345.

\bibitem{RozzaPatera:2007}
{\sc G.~Rozza, D.~B. Huynh, and A.~T. Patera}, {\em Reduced basis approximation
  and a posteriori error estimation for affinely parametrized elliptic coercive
  partial differential equations}, Archives of Computational Methods in
  Engineering, 15(3) (2007), pp.~1--47.

\bibitem{sapsis:09}
{\sc T.~Sapsis and P.~Lermusiaux}, {\em Dynamically orthogonal field equations
  for continuous stochastic dynamical systems}, Physica D: Nonlinear Phenomena,
  238 (2009), pp.~2347--2360.

\bibitem{Sirovich:1987}
{\sc L.~Sirovich}, {\em Turbulence and the dynamics of coherent structures.
  {I}. {C}oherent structures}, Quarterly of applied mathematics, 45(3) (1987),
  pp.~561--571.

\bibitem{wald2006building}
{\sc I.~Wald and V.~Havran}, {\em On building fast kd-trees for ray tracing,
  and on doing that in {O} ({N} log {N})}, in 2006 IEEE Symposium on
  Interactive Ray Tracing, IEEE, 2006, pp.~61--69.

\bibitem{Zabaras:13}
{\sc J.~Wan and N.~Zabaras}, {\em A probabilistic graphical model approach to
  stochastic multiscale partial differential equations}, Journal of
  Computational Physics, 250 (2013), pp.~477--510.

\bibitem{Wan:06}
{\sc X.~L. Wan and G.~Karniadakis}, {\em Multi-element generalized polynomial
  chaos for arbitrary probability measures}, SIAM J. Sci. Comp., 28 (2006),
  pp.~901--928.

\bibitem{Xiu:09}
{\sc D.~Xiu}, {\em Fast numerical methods for stochastic computations: a
  review}, Commun. Comput. Phys., 5 (2009), pp.~242--272.

\bibitem{Xiu:05}
{\sc D.~Xiu and J.~S. Hesthaven}, {\em High-order collocation methods for
  differential equations with random inputs}, SIAM J. Sci. Comp., 27 (2005),
  pp.~1118--1139.

\bibitem{Xiu:03}
{\sc D.~Xiu and G.~Karniadakis}, {\em Modeling uncertainty in flow simulations
  via generalized polynomial chaos}, J. Comput. Phys., 187 (2003),
  pp.~137--167.

\bibitem{ZhangCiHouMMS:15}
{\sc Z.~Zhang, M.~Ci, and T.~Y. Hou}, {\em A multiscale data-driven stochastic
  method for elliptic {PDE}s with random coefficients}, SIAM Multiscale Model.
  Simul., 13 (2015), pp.~173--204.

\end{thebibliography}


\end{document}